\documentclass[11pt]{article}
\usepackage{amsmath,amssymb,amsthm,amscd,eucal}
\usepackage{mathptmx}      
\usepackage[usenames]{color} 
\usepackage{amsmath}
\usepackage{tikz}
\usetikzlibrary{backgrounds}
\usepackage{genyoungtabtikz}

\usepackage[letterpaper,margin=1.25in]{geometry}

\newtheorem{thm}[equation]{Theorem}
\newtheorem{cor}[equation]{Corollary}

\newtheorem{prop}[equation]{Proposition}

\newtheorem*{conj*}{Conjecture}
\theoremstyle{definition}
\newtheorem{definition}[equation]{Definition}

\newtheorem{rem}[equation]{Remark}
\newtheorem{examp}[equation]{Example}

\numberwithin{equation}{section}

 \newcommand\stirling[2]{\left\{\genfrac{}{}{0pt}{}{#1}{#2}\right\}}
 
  \newcommand\stirlinginline[2]{\{\genfrac{}{}{0pt}{}{#1}{#2}\}}

\renewcommand{\S}{\mathsf{S}}  
\newcommand{\GG}{\mathsf{G}}  
\newcommand{\G}{\mathsf{G}}  
\newcommand{\HH}{\mathsf{H}}  
\renewcommand{\P}{\mathsf{P}}
\newcommand{\Z}{\mathsf{Z}} 
\newcommand{\FF}{\mathbb{F}}  
 
\newcommand{\CC}{\mathbb{C}} 
\newcommand{\ZZ}{\mathbb{Z}}  
\newcommand{\M}{\mathsf{M}} 
\newcommand{\VV}{\mathsf{V}}
\newcommand{\UU}{\mathsf{U}}

\def\ot{\otimes}
\def\End{ \mathsf{End}}
\def\Res{\mathsf{Res}}
\def\Ind{\mathsf{Ind}}
\newcommand{\ef}{\mathsf{e}}
\newcommand{\im}{\mathsf{im}\,}
\renewcommand{\ker}{\mathsf{ker}\,}
\newcommand{\modu}{\mathsf{M}_n}
\newcommand{\vf}{\mathsf{v}}
\renewcommand{\v}{\mathsf{v}}
\newcommand{\m}{\mathsf{m}}
\newcommand{\half}{\frac{1}{2}}
\newcommand{\xx}{\mathsf{X}}
\newcommand{\yy}{\mathsf{Y}}
\newcommand{\lam}{\lambda}
\newcommand{\Ks}{\mathsf{K}}  
\newcommand{\vr}{\varrho} 
\newcommand{\rs}{\mathsf{r}}
\newcommand{\sff}{\mathsf{s}}
\newcommand{\EE}{\mathsf{E}}  
\newcommand{\tf}{\mathsf{t}}  
\newcommand{\bs}{\mathsf{b}}
\newcommand{\B}{\mathsf{B}}
\newcommand{\dimm}{\mathsf{dim}}
\newcommand{\spann}{\mathsf{span}}

\font \nmm = cmbx9

\newcommand{\blvertedge}{ 
\begin{array}{c}
\begin{tikzpicture}[xscale=.25,yscale=.25,line width=1.0pt] 
\foreach \i in {1}  { \path (\i,1.25) coordinate (T\i); \path (\i,.25) coordinate (B\i); } 
\draw[blue] (B1) -- (T1);
\foreach \i in {1}  { \filldraw[fill=black,draw=black,line width = 1pt] (T\i) circle (5pt); \filldraw[fill=black,draw=black,line width = 1pt]  (B\i) circle (5pt); } 
\end{tikzpicture}
\end{array}}

\newcommand{\emptydiagram}{ 
\begin{array}{c}
\begin{tikzpicture}[xscale=.25,yscale=.25,line width=1.0pt] 
\foreach \i in {1}  { \path (\i,1.25) coordinate (T\i); \path (\i,.25) coordinate (B\i); } 
\foreach \i in {1}  { \filldraw[fill=black,draw=black,line width = 1pt] (T\i) circle (5pt); \filldraw[fill=black,draw=black,line width = 1pt]  (B\i) circle (5pt); } 
\end{tikzpicture}
\end{array}}

\def\abcd{{
\begin{tikzpicture}[xscale=.5,yscale=.5,line width=1.25pt] 
\foreach \i in {1,2} 
{ \path (\i,1.25) coordinate (T\i); \path (\i,.25) coordinate (B\i); } 
\filldraw[fill= black!12,draw=black!12,line width=4pt]  (T1) -- (T2) -- (B2) -- (B1) -- (T1);
\draw[blue] (T2) -- (B1);\draw[blue] (T1) -- (B2);\draw[blue] (T1) -- (B1);\draw[blue] (T2) -- (B2);
\draw[blue] (T1) -- (T2);\draw[blue] (B1) -- (B2);
\foreach \i in {1,2} 
{ \fill (T\i) circle (4pt); \fill (B\i) circle (4pt); } 
\end{tikzpicture}}}

\def\orbitabcd{{
\begin{tikzpicture}[xscale=.5,yscale=.5,line width=1.25pt] 
\foreach \i in {1,2} 
{ \path (\i,1.25) coordinate (T\i); \path (\i,.25) coordinate (B\i); } 
\filldraw[fill= black!12,draw=black!12,line width=4pt]  (T1) -- (T2) -- (B2) -- (B1) -- (T1);
\draw[blue] (T2) -- (B1);\draw[blue] (T1) -- (B2);\draw[blue] (T1) -- (B1);\draw[blue] (T2) -- (B2);
\draw[blue] (T1) -- (T2);\draw[blue] (B1) -- (B2);
\foreach \i in {1,2} 
{ \filldraw[fill=white,draw=black,line width = 1pt] (T\i) circle (4pt); \filldraw[fill=white,draw=black,line width = 1pt]  (B\i) circle (4pt); } 
\end{tikzpicture}}}

\def\orbitabcld{{
\begin{tikzpicture}[xscale=.5,yscale=.5,line width=1.25pt] 
\foreach \i in {1,2} 
{ \path (\i,1.25) coordinate (T\i); \path (\i,.25) coordinate (B\i); } 
\filldraw[fill= black!12,draw=black!12,line width=4pt]  (T1) -- (T2) -- (B2) -- (B1) -- (T1);
\draw[blue] (T1) -- (B2);\draw[blue] (T1) -- (B1);
\draw[blue] (B1) -- (B2);
\foreach \i in {1,2} 
{ \filldraw[fill=white,draw=black,line width = 1pt] (T\i) circle (4pt); \filldraw[fill=white,draw=black,line width = 1pt]  (B\i) circle (4pt); } 
\end{tikzpicture}}}

\def\abcld{{
\begin{tikzpicture}[xscale=.5,yscale=.5,line width=1.25pt] 
\foreach \i in {1,2} 
{ \path (\i,1.25) coordinate (T\i); \path (\i,.25) coordinate (B\i); } 
\filldraw[fill= black!12,draw=black!12,line width=4pt]  (T1) -- (T2) -- (B2) -- (B1) -- (T1);
\draw[blue] (T1) -- (B2);\draw[blue] (T1) -- (B1);
\draw[blue] (B1) -- (B2);
\foreach \i in {1,2} 
{ \fill (T\i) circle (4pt); \fill (B\i) circle (4pt); } 
\end{tikzpicture}}}

\def\abdlc{{
\begin{tikzpicture}[xscale=.5,yscale=.5,line width=1.25pt] 
\foreach \i in {1,2} 
{ \path (\i,1.25) coordinate (T\i); \path (\i,.25) coordinate (B\i); } 
\filldraw[fill= black!12,draw=black!12,line width=4pt]  (T1) -- (T2) -- (B2) -- (B1) -- (T1);
\draw[blue] (B1) -- (B2);\draw[blue] (T2) -- (B2);
\draw[blue] (B1) -- (T2);
\foreach \i in {1,2} 
{ \fill (T\i) circle (4pt); \fill (B\i) circle (4pt); } 
\end{tikzpicture}}}
 
\def\ablcld{{
\begin{tikzpicture}[xscale=.5,yscale=.5,line width=1.25pt] 
\foreach \i in {1,2} 
{ \path (\i,1.25) coordinate (T\i); \path (\i,.25) coordinate (B\i); } 
\filldraw[fill= black!12,draw=black!12,line width=4pt]  (T1) -- (T2) -- (B2) -- (B1) -- (T1);
\draw[blue] (B1) -- (B2);
\foreach \i in {1,2} 
{ \fill (T\i) circle (4pt); \fill (B\i) circle (4pt); } 
\end{tikzpicture}}}

\def\aclbld{{
\begin{tikzpicture}[xscale=.5,yscale=.5,line width=1.25pt] 
\foreach \i in {1,2} 
{ \path (\i,1.25) coordinate (T\i); \path (\i,.25) coordinate (B\i); } 
\filldraw[fill= black!12,draw=black!12,line width=4pt]  (T1) -- (T2) -- (B2) -- (B1) -- (T1);
\draw[blue] (B1) -- (T1);
\foreach \i in {1,2} 
{ \fill (T\i) circle (4pt); \fill (B\i) circle (4pt); } 
\end{tikzpicture}}}

\def\orbitaclbld{{
\begin{tikzpicture}[xscale=.5,yscale=.5,line width=1.25pt] 
\foreach \i in {1,2} 
{ \path (\i,1.25) coordinate (T\i); \path (\i,.25) coordinate (B\i); } 
\filldraw[fill= black!12,draw=black!12,line width=4pt]  (T1) -- (T2) -- (B2) -- (B1) -- (T1);
\draw[blue] (B1) -- (T1);
\foreach \i in {1,2} 
{ \filldraw[fill=white,draw=black,line width = 1pt] (T\i) circle (4pt); \filldraw[fill=white,draw=black,line width = 1pt]  (B\i) circle (4pt); } 
\end{tikzpicture}}}

\def\albcld{{
\begin{tikzpicture}[xscale=.5,yscale=.5,line width=1.25pt] 
\foreach \i in {1,2} 
{ \path (\i,1.25) coordinate (T\i); \path (\i,.25) coordinate (B\i); } 
\filldraw[fill= black!12,draw=black!12,line width=4pt]  (T1) -- (T2) -- (B2) -- (B1) -- (T1);
\draw[blue] (B2) -- (T1);
\foreach \i in {1,2} 
{ \fill (T\i) circle (4pt); \fill (B\i) circle (4pt); } 
\end{tikzpicture}}}

\def\alblcld{{
\begin{tikzpicture}[xscale=.5,yscale=.5,line width=1.25pt] 
\foreach \i in {1,2} 
{ \path (\i,1.25) coordinate (T\i); \path (\i,.25) coordinate (B\i); } 
\filldraw[fill= black!12,draw=black!12,line width=4pt]  (T1) -- (T2) -- (B2) -- (B1) -- (T1);
\foreach \i in {1,2} 
{ \fill (T\i) circle (4pt); \fill (B\i) circle (4pt); } 
\end{tikzpicture}}}

\def\orbita|bc|d{{
\begin{tikzpicture}[xscale=.5,yscale=.5,line width=1.25pt] 
\foreach \i in {1,2} 
{ \path (\i,1.25) coordinate (T\i); \path (\i,.25) coordinate (B\i); } 
\filldraw[fill= black!12,draw=black!12,line width=4pt]  (T1) -- (T2) -- (B2) -- (B1) -- (T1);
\draw[blue] (T1) -- (B2); 
\foreach \i in {1,2} 
{ \filldraw[fill=white,draw=black,line width = 1pt] (T\i) circle (4pt); \filldraw[fill=white,draw=black,line width = 1pt]  (B\i) circle (4pt); } 
\end{tikzpicture}}}

\def\adlblc{{
\begin{tikzpicture}[xscale=.5,yscale=.5,line width=1.25pt] 
\foreach \i in {1,2} 
{ \path (\i,1.25) coordinate (T\i); \path (\i,.25) coordinate (B\i); } 
\filldraw[fill= black!12,draw=black!12,line width=4pt]  (T1) -- (T2) -- (B2) -- (B1) -- (T1);
\draw[blue] (B1) -- (T2);
\foreach \i in {1,2} 
{ \fill (T\i) circle (4pt); \fill (B\i) circle (4pt); } 
\end{tikzpicture}}}

\def\orbitadlblc{{
\begin{tikzpicture}[xscale=.5,yscale=.5,line width=1.25pt] 
\foreach \i in {1,2} 
{ \path (\i,1.25) coordinate (T\i); \path (\i,.25) coordinate (B\i); } 
\filldraw[fill= black!12,draw=black!12,line width=4pt]  (T1) -- (T2) -- (B2) -- (B1) -- (T1);
\draw[blue] (B1) -- (T2);
\foreach \i in {1,2} 
{ \filldraw[fill=white,draw=black,line width = 1pt] (T\i) circle (4pt); \filldraw[fill=white,draw=black,line width = 1pt]  (B\i) circle (4pt); } 
\end{tikzpicture}}}

\def\albdlc{{
\begin{tikzpicture}[xscale=.5,yscale=.5,line width=1.25pt] 
\foreach \i in {1,2} 
{ \path (\i,1.25) coordinate (T\i); \path (\i,.25) coordinate (B\i); } 
\filldraw[fill= black!12,draw=black!12,line width=4pt]  (T1) -- (T2) -- (B2) -- (B1) -- (T1);
\draw[blue] (B2) -- (T2);
\foreach \i in {1,2} 
{ \fill (T\i) circle (4pt); \fill (B\i) circle (4pt); } 
\end{tikzpicture}}}

\def\orbitalblcld{{
\begin{tikzpicture}[xscale=.5,yscale=.5,line width=1.25pt] 
\foreach \i in {1,2} 
{ \path (\i,1.25) coordinate (T\i); \path (\i,.25) coordinate (B\i); } 
\filldraw[fill= black!12,draw=black!12,line width=4pt]  (T1) -- (T2) -- (B2) -- (B1) -- (T1);
\foreach \i in {1,2} 
{ 
\filldraw[fill=white,draw=black,line width = 1pt] (T\i) circle (4pt); \filldraw[fill=white,draw=black,line width = 1pt]  (B\i) circle (4pt); 
} 
\end{tikzpicture}}}

\def\ablcd{{
\begin{tikzpicture}[xscale=.5,yscale=.5,line width=1.25pt] 
\foreach \i in {1,2} 
{ \path (\i,1.25) coordinate (T\i); \path (\i,.25) coordinate (B\i); } 
\filldraw[fill= black!12,draw=black!12,line width=4pt]  (T1) -- (T2) -- (B2) -- (B1) -- (T1);
\draw[blue] (B2) -- (B1);\draw[blue] (T2) -- (T1);
\foreach \i in {1,2} 
{ \fill (T\i) circle (4pt); \fill (B\i) circle (4pt); } 
\end{tikzpicture}}}

\def\aclbd{{
\begin{tikzpicture}[xscale=.5,yscale=.5,line width=1.25pt] 
\foreach \i in {1,2} 
{ \path (\i,1.25) coordinate (T\i); \path (\i,.25) coordinate (B\i); } 
\filldraw[fill= black!12,draw=black!12,line width=4pt]  (T1) -- (T2) -- (B2) -- (B1) -- (T1);
\draw[blue] (T1) -- (B1); 
\draw[blue] (T2) -- (B2); 
\foreach \i in {1,2} 
{ \fill (T\i) circle (4pt); \fill (B\i) circle (4pt); } 
\end{tikzpicture}}}

\def\acdlb{{
\begin{tikzpicture}[xscale=.5,yscale=.5,line width=1.25pt] 
\foreach \i in {1,2} 
{ \path (\i,1.25) coordinate (T\i); \path (\i,.25) coordinate (B\i); } 
\filldraw[fill= black!12,draw=black!12,line width=4pt]  (T1) -- (T2) -- (B2) -- (B1) -- (T1);
\draw[blue] (T2) -- (B1) -- (T1) -- (T2);
\foreach \i in {1,2} 
{ \fill (T\i) circle (4pt); \fill (B\i) circle (4pt); } 
\end{tikzpicture}}}

\def\orbitacdlb{{
\begin{tikzpicture}[xscale=.5,yscale=.5,line width=1.25pt] 
\foreach \i in {1,2} 
{ \path (\i,1.25) coordinate (T\i); \path (\i,.25) coordinate (B\i); } 
\filldraw[fill= black!12,draw=black!12,line width=4pt]  (T1) -- (T2) -- (B2) -- (B1) -- (T1);
\draw[blue] (T2) -- (B1) -- (T1) -- (T2);
\foreach \i in {1,2} 
{ \filldraw[fill=white,draw=black,line width = 1pt] (T\i) circle (4pt); \filldraw[fill=white,draw=black,line width = 1pt]  (B\i) circle (4pt);  } 
\end{tikzpicture}}}

\def\adlbc{{
\begin{tikzpicture}[xscale=.5,yscale=.5,line width=1.25pt] 
\foreach \i in {1,2} 
{ \path (\i,1.25) coordinate (T\i); \path (\i,.25) coordinate (B\i); } 
\filldraw[fill= black!12,draw=black!12,line width=4pt]  (T1) -- (T2) -- (B2) -- (B1) -- (T1);
\draw[blue] (T2) -- (B1);
\draw[blue] (T1) -- (B2);
\foreach \i in {1,2} 
{ \fill (T\i) circle (4pt); \fill (B\i) circle (4pt); } 
\end{tikzpicture}}}

\def\albcd{{
\begin{tikzpicture}[xscale=.5,yscale=.5,line width=1.25pt] 
\foreach \i in {1,2} 
{ \path (\i,1.25) coordinate (T\i); \path (\i,.25) coordinate (B\i); } 
\filldraw[fill= black!12,draw=black!12,line width=4pt]  (T1) -- (T2) -- (B2) -- (B1) -- (T1);
\draw[blue] (T2) -- (B2)--(T1)--(T2);
\foreach \i in {1,2} 
{ \fill (T\i) circle (4pt); \fill (B\i) circle (4pt); } 
\end{tikzpicture}}}

\def\alblcd{{
\begin{tikzpicture}[xscale=.5,yscale=.5,line width=1.25pt] 
\foreach \i in {1,2} 
{ \path (\i,1.25) coordinate (T\i); \path (\i,.25) coordinate (B\i); } 
\filldraw[fill= black!12,draw=black!12,line width=4pt]  (T1) -- (T2) -- (B2) -- (B1) -- (T1);
\draw[blue] (T2) -- (T1);
\foreach \i in {1,2} 
{ \fill (T\i) circle (4pt); \fill (B\i) circle (4pt); } 
\end{tikzpicture}}}

\def\orbitalblcd{{
\begin{tikzpicture}[xscale=.5,yscale=.5,line width=1.25pt] 
\foreach \i in {1,2} 
{ \path (\i,1.25) coordinate (T\i); \path (\i,.25) coordinate (B\i); } 
\filldraw[fill= black!12,draw=black!12,line width=4pt]  (T1) -- (T2) -- (B2) -- (B1) -- (T1);
\draw[blue] (T2) -- (T1);
\foreach \i in {1,2} 
{ \filldraw[fill=white,draw=black,line width = 1pt] (T\i) circle (4pt); \filldraw[fill=white,draw=black,line width = 1pt]  (B\i) circle (4pt);  } 
\end{tikzpicture}}}

\def\orbitablcd{{
\begin{tikzpicture}[xscale=.50,yscale=.50,line width=1.25pt] 
\foreach \i in {1,2} 
{ \path (\i,1.25) coordinate (T\i); \path (\i,.25) coordinate (B\i); } 
\filldraw[fill= black!12,draw=black!12,line width=4pt]  (T1) -- (T2) -- (B2) -- (B1) -- (T1);
\draw[blue] (B2) -- (B1);\draw[blue] (T2) -- (T1);
\foreach \i in {1,2} 
{ \filldraw[fill=white,draw=black,line width = 1pt] (T\i) circle (4pt); \filldraw[fill=white,draw=black,line width = 1pt]  (B\i) circle (4pt); } 
\end{tikzpicture}}}

\def\orbitadlbc{{
\begin{tikzpicture}[xscale=.50,yscale=.50,line width=1.25pt] 
\foreach \i in {1,2} 
{ \path (\i,1.25) coordinate (T\i); \path (\i,.25) coordinate (B\i); } 
\filldraw[fill= black!12,draw=black!12,line width=4pt]  (T1) -- (T2) -- (B2) -- (B1) -- (T1);
\draw[blue] (T2) -- (B1);
\draw[blue] (T1) -- (B2);
\foreach \i in {1,2} 
{ \filldraw[fill=white,draw=black,line width = 1pt] (T\i) circle (4pt); \filldraw[fill=white,draw=black,line width = 1pt]  (B\i) circle (4pt);  } 
\end{tikzpicture}}}

\def\orbitaclbd{{
\begin{tikzpicture}[xscale=.50,yscale=.50,line width=1.25pt] 
\foreach \i in {1,2} 
{ \path (\i,1.25) coordinate (T\i); \path (\i,.25) coordinate (B\i); } 
\filldraw[fill= black!12,draw=black!12,line width=4pt]  (T1) -- (T2) -- (B2) -- (B1) -- (T1);
\draw[blue] (T1) -- (B1); 
\draw[blue] (T2) -- (B2); 
\foreach \i in {1,2} 
{ \filldraw[fill=white,draw=black,line width = 1pt] (T\i) circle (4pt); \filldraw[fill=white,draw=black,line width = 1pt]  (B\i) circle (4pt); } 
\end{tikzpicture}}}

\def\orbitalbcd{{
\begin{tikzpicture}[xscale=.50,yscale=.50,line width=1.25pt] 
\foreach \i in {1,2} 
{ \path (\i,1.25) coordinate (T\i); \path (\i,.25) coordinate (B\i); } 
\filldraw[fill= black!12,draw=black!12,line width=4pt]  (T1) -- (T2) -- (B2) -- (B1) -- (T1);
\draw[blue] (T2) -- (B2)--(T1)--(T2);
\foreach \i in {1,2} 
{ \filldraw[fill=white,draw=black,line width = 1pt] (T\i) circle (4pt); \filldraw[fill=white,draw=black,line width = 1pt]  (B\i) circle (4pt);  } 
\end{tikzpicture}}}

\def\orbitabdlc{{
\begin{tikzpicture}[xscale=.50,yscale=.50,line width=1.25pt] 
\foreach \i in {1,2} 
{ \path (\i,1.25) coordinate (T\i); \path (\i,.25) coordinate (B\i); } 
\filldraw[fill= black!12,draw=black!12,line width=4pt]  (T1) -- (T2) -- (B2) -- (B1) -- (T1);
\draw[blue] (B1) -- (B2);\draw[blue] (T2) -- (B2);
\draw[blue] (B1) -- (T2);
\foreach \i in {1,2} 
{ \filldraw[fill=white,draw=black,line width = 1pt] (T\i) circle (4pt); \filldraw[fill=white,draw=black,line width = 1pt]  (B\i) circle (4pt); } 
\end{tikzpicture}}}

\def\orbitablcld{{
\begin{tikzpicture}[xscale=.50,yscale=.50,line width=1.25pt] 
\foreach \i in {1,2} 
{ \path (\i,1.25) coordinate (T\i); \path (\i,.25) coordinate (B\i); } 
\filldraw[fill= black!12,draw=black!12,line width=4pt]  (T1) -- (T2) -- (B2) -- (B1) -- (T1);
\draw[blue] (B1) -- (B2);
\foreach \i in {1,2} 
{ \filldraw[fill=white,draw=black,line width = 1pt] (T\i) circle (4pt); \filldraw[fill=white,draw=black,line width = 1pt]  (B\i) circle (4pt); } 
\end{tikzpicture}}}

\def\orbitalbcld{{
\begin{tikzpicture}[xscale=.50,yscale=.50,line width=1.25pt] 
\foreach \i in {1,2} 
{ \path (\i,1.25) coordinate (T\i); \path (\i,.25) coordinate (B\i); } 
\filldraw[fill= black!12,draw=black!12,line width=4pt]  (T1) -- (T2) -- (B2) -- (B1) -- (T1);
\draw[blue] (B2) -- (T1);
\foreach \i in {1,2} 
{ \filldraw[fill=white,draw=black,line width = 1pt] (T\i) circle (4pt); \filldraw[fill=white,draw=black,line width = 1pt]  (B\i) circle (4pt); } 
\end{tikzpicture}}}

\def\orbitalbdlc{{
\begin{tikzpicture}[xscale=.50,yscale=.50,line width=1.25pt] 
\foreach \i in {1,2} 
{ \path (\i,1.25) coordinate (T\i); \path (\i,.25) coordinate (B\i); } 
\filldraw[fill= black!12,draw=black!12,line width=4pt]  (T1) -- (T2) -- (B2) -- (B1) -- (T1);
\draw[blue] (B2) -- (T2);
\foreach \i in {1,2} 
{ \filldraw[fill=white,draw=black,line width = 1pt] (T\i) circle (4pt); \filldraw[fill=white,draw=black,line width = 1pt]  (B\i) circle (4pt); } 
\end{tikzpicture}}}

\begin{document}

\title{Partition Algebras and the Invariant Theory \\ of the Symmetric Group}

\author{Georgia Benkart \\
 {\small University of Wisconsin-Madison}
 \\{\small Madison, WI 53706, USA}\\
\texttt{\small benkart@math.wisc.edu}
\and Tom Halverson\footnote{
The second author gratefully acknowledges partial support from Simons Foundation grant 283311.}\\
{\small Macalester College} \\{\small Saint Paul, MN 55105, USA}\\
\texttt{\small halverson@macalester.edu}
} 
\date{September 20, 2017}

\maketitle

\begin{abstract}
The symmetric group $\S_n$ and the partition algebra $\P_k(n)$ centralize one another in their actions on the 
 $k$-fold tensor power $\mathsf{M}_n^{\otimes k}$ of the $n$-dimensional permutation module $\M_n$ of $\S_n$.  The duality afforded by
 the commuting actions determines an algebra homomorphism \ $\Phi_{k,n}: \P_k(n) \to \End_{\S_n}(\modu^{\otimes k})$ \ from the partition algebra to the centralizer algebra $\End_{\S_n}(\modu^{\otimes k})$,  which is a surjection for all \,$k, n \in \ZZ_{\ge 1}$, and an isomorphism when $n \ge 2k$. \, We 
present results that can be derived from the duality between $\S_n$ and $\P_k(n)$; for example, (i) expressions for the multiplicities of the irreducible $\S_n$-summands of $\mathsf{M}_n^{\otimes k}$, 
 (ii) formulas for the dimensions of the irreducible modules for  
the centralizer algebra $\End_{\S_n}(\modu^{\otimes k})$, (iii) a bijection between vacillating tableaux and set-partition tableaux,
(iv) identities relating Stirling numbers of the second kind and the number of fixed points of
permutations, and (v) character values for the partition algebra $\P_k(n)$.   When $2k >n$,  the map $\Phi_{k,n}$ has a nontrivial kernel  which is generated as a two-sided ideal by a single idempotent.  
We describe the kernel and image of $\Phi_{k,n}$ in terms
of the orbit basis of $\P_k(n)$ and explain how the surjection $\Phi_{k,n}$ can also be used to
obtain the fundamental theorems of invariant theory for the symmetric group.
\end{abstract}



\section{Introduction}
\label{sec:intro}

\emph{Throughout we assume $\FF$ is a field of characteristic zero.} \   The symmetric group $\S_n$ has a natural action on its $n$-dimensional permutation
module $\M_n$ over $\FF$  by permuting the basis elements.  The focus of this article is on tensor powers
$\M_n^{\ot k}$ of $\M_n$, which are $\S_n$-modules under the diagonal action.  In \cite{J}, V. Jones constructed a surjective algebra homomorphism
\begin{equation}
\Phi_{k,n}:  \P_k(n) \rightarrow \End_{\S_n}(\M_n^{\ot k})
\end{equation}
from the partition algebra $\P_k(n)$ onto the centralizer algebra  $\End_{\S_n}(\M_n^{\ot k})$ of transformations that
commute with the action of $\S_n$ on $\M_n^{\ot k}$.   When $n \ge 2k$, this surjection is an isomorphism. 
The Schur-Weyl duality afforded by $\Phi_{k,n}$  between the partition algebra $\P_k(n)$ and the symmetric group $\S_n$ in their commuting actions on
$\M_n^{\ot k}$  enables information to flow back and forth between $\S_n$ and
$\P_k(n)$.      Indeed,  the symmetric group $\S_n$ has been used to
\begin{itemize}
\item[$\bullet$]  develop the combinatorial representation theory of the partition algebras  $\P_k(n)$ as $k$ varies  
\cite{M3,MR,H,FaH,HL,HR,BH,BHH}, 
\item[$\bullet$] study the Potts lattice model of interacting spins in statistical mechanics \cite{M1,M2,M3};
\end{itemize}
while the partition algebra $\P_k(n)$ has been used to
\begin{itemize}
\item[$\bullet$]   study eigenvalues of random permutation matrices  \cite{FaH},
\item[$\bullet$] prove results about Kronecker coefficients for $\S_n$-modules  \cite{BDO1,BDO2,BEG}, 
\item[$\bullet$] investigate  the centralizer algebras of the binary tetrahedral, octahedral, and icosahedral subgroups of the special unitary
group $\mathsf{SU}_2$ acting on  tensor powers of $\VV = \CC^2$ 
via the McKay correspondence  \cite{BBH},    
\item[$\bullet$] construct a non-homogeneous basis for the ring of symmetric functions and show that the
irreducible characters of $\S_n$ can be obtained by evaluating this basis on the eigenvalues of the permutation
matrices \cite{OZ}, which has enabled Orellana and Zabrocki to establish results on reduced Kronecker coefficients. 
\end{itemize}

The irreducible modules $\S_n^\lambda$  for the symmetric group $\S_n$ are indexed by partitions $\lambda$ of $n$, indicated here by $\lambda \vdash n$.     The $n$-dimensional permutation module $\M_n$  for $\S_n$ has basis elements
$\vf_1, \vf_2, \dots, \v_n$, which are permuted by the elements  $\sigma$ of $\S_n$, 
$\sigma.\vf_i = \vf_{\sigma(i)}$.  The module $\M_n$  has a decomposition  into irreducible $\S_n$-modules, 
$\M_n = \S_n^{[n]} \oplus \S_n^{[n-1,1]}$, where the vector $\v_1 + \v_2 + \cdots + \v_n$ spans
a copy of the trivial $\S_n$-module $\S_n^{[n]}$, and  the vectors $\v_i-\v_{i+1}$
for $1 \leq i \le n-1$ form a basis for a copy of the $(n-1)$-dimensional module $\S_n^{[n-1,1]}$ corresponding to the two-part partition $[n-1,1]$. 
This module is often referred to as the ``reflection representation,''  since the transposition switching $i$ and $j$ is a reflection about the hyperplane orthogonal to  $\vf_i - \vf_{j}$.   The $\S_n$-action on $\M_n$  extends to give the tensor power $\M_n^{\ot k}$ the structure of an $\S_n$-module:  
$$\sigma.\left(\vf_{i_1} \ot \vf_{i_2} \ot \cdots \ot v_{i_k}\right) = \sigma.\vf_{i_1} \ot \sigma.\vf_{i_2} \ot \cdots \ot \sigma.\vf_{i_k} = \vf_{\sigma(i_1)} \ot \vf_{\sigma(i_2)} \ot \cdots \ot \vf_{\sigma(i_k)}.$$
 
In Sections \ref{sec:res-ind}  and \ref{sec:setpart},  we describe two different ways to determine
the multiplicity  $\mathsf{m}_{k,n}^\lambda$ of $\S_n^\lambda$ in the decomposition
$$\M_n^{\ot k} = \bigoplus_{\lambda \vdash n}  \mathsf{m}_{k,n}^\lambda \, \S_n^\lambda$$ 
of $\M_n^{\ot k}$ into irreducible $\S_n$-summands.  
The first approach, using restriction and induction on the pair
$(\S_n,\S_{n-1})$,  leads naturally to a bijection between the set of vacillating $k$-tableaux
of shape $\lambda$  and the set of paths in the Bratteli diagram $\mathcal B(\S_n, \S_{n-1})$ from the partition $[n]$ at the top of $\mathcal B(\S_n, \S_{n-1})$  to $\lambda$ at level $k$.  Both sets have cardinality $\mathsf{m}_{k,n}^\lambda$.  
 The second way, adopted from \cite{BHH},  uses permutation modules for $\S_n$ and results in an exact expression for the multiplicity
of an irreducible $\S_n$-summand of $\M_n^{\ot k}$.
 In particular,  we describe how 
  \cite[Thm.~5.5]{BHH}  implies that
$$\mathsf{m}_{k,n}^\lambda = \sum_{t=|\lambda^\#|}^n  \stirling{k}{t} f^{\lambda/[n-t]},$$
where the Stirling number of the second kind  $\stirlinginline{k}{t}$
 counts the number of ways to partition a set of  $k$ objects into $t$ nonempty 
disjoint blocks (subsets), $f^{\lambda/[n-t]}$ is the number of standard tableaux of skew
shape $\lambda/[n-t]$, and $\lambda^\#$ is the partition obtained from $\lambda$ by removing one copy of its
largest part.    Therefore,
\begin{equation*}
\m_{k,n}^\lambda 
= \bigg |\left\{ (\pi,\mathsf{S})\ \bigg\vert
\begin{array}{l}
\pi  \text{ is a set partition of $\{1,2,\ldots, k\}$ into $t$ parts, where $|\lambda^\#|\le t \le n$} \\
\mathsf{S}  \text{ is a standard  tableau of shape $\lambda/[n-t]$}
\end{array}
\right\}\bigg | \, .
\end{equation*}

In Section \ref{subsec:setparttabs},  we consider set-partition tableaux (tableaux whose boxes are filled with sets of numbers
from $\{0,1,\dots, k\})$  and demonstrate a bijection between set-partition tableaux and vacillating tableaux.  Different bijections
between these same sets are given in \cite{HL} and \cite{CDDSY}. However,  the bijections in those papers work only when $n \ge 2k$. The bijection 
here has the advantage of working for all $k,n \in \ZZ_{\ge1}$.
Such set-partition tableaux also appear in the recent work of Orellana and Zabrocki \cite[Prop. 2]{OZ}.    

At the core of these results is the duality between the representation
theories of the symmetric group $\S_n$ and the centralizer algebra,  
\begin{equation} \Z_{k,n}: =\End_{\S_n}(\M_n^{\ot k}) =\{\varphi \in \End(\M_n^{\ot k}) \mid \varphi \sigma (x) = 
\sigma \varphi(x), \  \  \sigma \in \S_n, x \in \M_n^{\ot k}\},\end{equation} 
of its action on $\M_n^{\ot k}$.   
Schur-Weyl duality tells us that the 
irreducible modules $\Z_{k,n}^\lambda$  for the semisimple associative algebra 
 $\Z_{k,n}$ 
are indexed by the subset $\Lambda_{k,\S_n}$  of  partitions $\lambda$ of  $n$ such that $\S_n^\lambda$ occurs with multiplicity at least one in the
decomposition of $\M_n^{\ot k}$ into irreducible $\S_n$-summands.     Moreover,
\begin{align} 
 \bullet &\quad \M_n^{\otimes k} \ \cong \ \underbrace{\bigoplus_{\lambda \in \Lambda_{k,\S_n}} \mathsf{m}_{k,n}^\lambda \S_n^\lambda}_{\text{as an $\S_n$-module}}
\ \cong \  \underbrace{\bigoplus_{\lambda \in \Lambda_{k,\S_n}} f^\lambda \Z_{k,n}^\lambda}_{\text{as a $\Z_{k,n}$-module} }; \label{SW1} \\
 \bullet &\quad \dim(\Z_{k,n}^\lambda)   = \mathsf{m}_{k,n}^\lambda   \hspace{.78in}   \text{(the multiplicity  of $\S_n^\lambda$ in $\M_n^{\ot k}$)}; \label{SW2} \\
  \bullet &\quad  \mathsf{mult}(\Z_{k,n}^\lambda) = \dim(\S_n^\lambda) =  f^\lambda  \quad   \text{(the number of standard tableaux of shape $\lambda$);} \label{SW3} \\  
 \bullet &\quad  \dim(\Z_{k,n}) = \displaystyle{\sum_{\lambda\in \Lambda_{k,\S_n}} (\mathsf{m}_{k,n}^\lambda)^2} = \m_{2k,n}^{[n]} = \dim\big(\Z_{2k,n}^{[n]}\big). \label{SW4}
\end{align} 
The first equality in (\ref{SW4}) is a consequence of (\ref{SW2})  and Artin-Wedderburn theory, since
$\dim(\Z_{k,n})$ is the sum of the squares of the dimensions of its irreducible modules $\Z_{k,n}^\lambda$.
The second equality in (\ref{SW4}) comes from the isomorphism $\Z_{k,n} =\End_{\S_n}(\M_n^{\ot k}) \cong (\M_n^{\ot 2k})^{\S_n}$, since the   
$\S_n$-invariants in $\M_n^{\ot 2k}$ correspond to copies of the trivial one-dimensional module $\S_n^{[n]}$ indexed by the one-part partition $[n]$,
 and  $\m_{2k,n}^{[n]} = \dim\big(\Z_{2k,n}^{[n]}\big)$ is the number of them in $\M_n^{\ot 2k}$.   
 
 Three additional Schur-Weyl duality results can be found in \cite[Cor.~2.5]{BM}.  From those results 
 we know that for any finite group $\G$ and any finite-dimensional
$\G$-module $\VV$,  the dimension of the space $\left(\VV^{\ot k}\right)^\G$ of $\G$-invariants is the average of the character values $\chi_{\VV^{\ot k}}(g) = \chi_{{}_\VV}(g)^k$  as $g$ ranges over the elements of $\G$  (compare also \cite[Sec.~2.2]{FuH}).  When $\VV$ is isomorphic to its dual $\G$-module,  the dimension of the centralizer algebra $\End_\G(\VV^{\ot k})$  is the average of the character values $\chi_{\VV^{\ot {2k}}}(g)
= \chi_{{}_\VV}(g)^{2k}$.    Specializing those results and part (i) of  \cite[Cor.~2.5]{BM}
to the case $\G =\S_n$ and  $\VV = \M_n$ gives the following: 
\begin{align}\bullet&\qquad \dim\big(\M_n^{\ot k}\big)^{\S_n} =  \frac{1}{n!} \sum_{\sigma \in \S_n} \chi_{\M_n}(\sigma)^k =  \frac{1}{n!} \sum_{\sigma \in \S_n} \mathsf{F}(\sigma)^k; 
 \label{eq:dim1}\\
\bullet&\qquad \dim(\Z_{k,n}) = \dim(\Z_{2k,n}^{[n]} )=  \frac{1}{n!} \sum_{\sigma  \in \S_n} \mathsf{F}(\sigma)^{2k};  \label{eq:dim2} \\
\bullet&\qquad \dim(\Z_{k,n}^\lambda) =   \frac{1}{n!} \sum_{\sigma \in \S_n}  \chi_{\M_n}(\sigma)^k \ \chi_{\lambda}(\sigma^{-1})
=   \frac{1}{n!}  \sum_{\sigma \in \S_n} \mathsf{F}(\sigma)^k \ \chi_{\lambda}(\sigma); \label{eq:dim3} 
\end{align}
where $\chi_{\M_n}(\sigma) = \mathsf{F}(\sigma)$ (the number of fixed points of $\sigma$), and $\chi_\lambda$ is the character of the irreducible 
$\S_n$-module $\S_n^\lambda$. In \eqref{eq:dim3}, we have used the fact that $\sigma$ and $\sigma^{-1}$ have
the same cycle structures, hence the same character values.

In \cite{FaH}, Farina and Halverson develop the character theory of partition algebras and use partition algebra characters
to prove \eqref{eq:dim2}.   The arguments in \cite{FaH},  which are based on results from \cite{H}, require the assumption $n \ge 2k$ so that 
$\Z_{k,n} \cong \P_k(n)$.   The results in \eqref{eq:dim1}-\eqref{eq:dim3} are true for all $k,n \ge 1$.

When $\VV$ is a finite-dimensional module for a finite group $\G$, the tensor power $\VV^{\ot k}$ has a multiplicity-free decomposition 
$$\VV^{\ot k} = \bigoplus_{\nu} \left( \G^\nu \ot \Z_{k,\G}^\nu\right)$$
into irreducible summands as a bimodule for $\G \times \Z_{k,\G}$, where $\Z_{k,\G}=\End_{\G}(\VV^{\ot k})$.  
Consequently, the characters of $\G$, $\Z_{k,\G}$ and $\G \times \Z_{k,\G}$  are intertwined by the following equation, 
 \begin{equation*}\label{g}
\psi_{\VV^{\ot k}}(g \,\times\, z) = \sum_{\nu}  \chi_\nu(g) \  \xi_\nu(z). 
\end{equation*}

The irreducible characters of $\G$ are orthonormal with respect to the  well-known inner product on class functions of $\G$ 
defined by $\langle \alpha, \beta \rangle = \vert \G \vert^{-1} \sum_{g \in \G} \alpha(g) \beta(g^{-1})$ 
(see, for example, \cite[Thm.~2.12]{FuH}  or \cite[Thm.~1.93]{Sa}).  Therefore, since $\psi_{\VV^{\ot k}}( \, \mathbf{\cdot} \,\times  z): \G \rightarrow \FF$ is a class function on $\G$ for each $z\in \Z_{k,\G}$,  we have  
\begin{equation*}
\xi_\nu(z) =  \langle \psi_{\VV^{\ot k}}( \mathbf{\cdot} \,\times z),  \chi_\nu \rangle= \frac{1}{|\G|} \sum_{g \in \G} \psi_{\VV^{\ot k}}(g \,\times\, z)\chi_{\nu}(g^{-1}).
\end{equation*}
Thus,  the commuting actions of $\G$ and $\Z_{k,\G}$ on $\VV^{\ot k}$ lead to one further Schur-Weyl duality result, namely, an expression for 
the irreducible characters $\xi_\nu$
of $\Z_{k,\G}$:
\begin{align}\label{eq:SWchar}
 \bullet &  \qquad \xi_\nu(z) =  \frac{1}{|\G|} \sum_{g \in \G} \psi_{\VV^{\ot k}}(g \,\times\, z)\chi_{\nu}(g^{-1}).
 \hskip1.51in
 \end{align}
In Section \ref{subsec:chars}, we explain how these ideas, when combined with results from \cite{H}, provide the formula in Theorem \ref{T:partitionchars}
for the characters of the partition algebra $\P_k(n)$.

The surjective algebra homomorphism
\begin{equation}\label{eq:Phikn} \Phi_{k,n}:\P_k(n) \to \Z_{k,n} = \End_{\S_n}(\modu^{\otimes k} ) \end{equation}  
enables us to study $\M_n^{\ot k}$ using partition algebra considerations.   The set partitions of the set $[1,2k]: =\{1,2,\ldots,2k\}$ index an $\FF$-basis for $\P_k(n)$, and thus,  $\P_k(n)$ has dimension equal to the Bell number $\mathsf{B}(2k) = \sum_{t=1}^{2k}  \stirlinginline{2k}{t}$.  

 We let $\Pi_{2k}$ be the set of set partitions of $[1,2k]$;   for example, 
$\big\{1,7,8,10 \,|\,  2,5 \,|\,  4,9,11\,|\, 3,6,12,14$ $\,|\, 13\big\}$ is a set partition in
$ \Pi_{14}$ with 5 blocks. 

The algebra $\P_k(n)$ has  two distinguished bases --  the diagram basis $\{ d_\pi \mid \pi \in \Pi_{2k}\}$  and the orbit basis $\{ x_\pi \mid \pi \in \Pi_{2k}\}$. We describe the change of basis matrix  between the diagram basis and the orbit basis in terms of the M\"obius function of the set-partition lattice (see Section \ref{subsec:change}).  Section \ref{subsec:mult} is devoted to a description of multiplication in the orbit basis of
$\P_k(n)$.    The actions of the diagram basis element $d_\pi$ and the orbit basis element $x_\pi$ on $\M_n^{\ot k}$
afforded by the representation $\Phi_{k,n}$ are given explicitly in Section \ref{subsec:repn} for any $\pi \in \Pi_{2k}$.  
The orbit basis is key to understanding the image and the kernel of $\Phi_{k,n}$, as we explain in Section \ref{subsec:kernels}.
An analog of the orbit basis exists in a broader context, and in Section \ref{subsec:orbitbasis}  we 
show how to construct an orbit basis for the centralizer algebra of a tensor power of any permutation module for an arbitrary finite group $\G$.

The centralizer algebra $\Z_{k,n}  = \End_{\S_n}(\M_n^{\ot k})$ has a basis consisting of the images $\Phi_{k,n}(x_\pi)$, where $\pi$ ranges over the set partitions in $\Pi_{2k}$ with
no more than $n$ blocks.   As a consequence of that result and \eqref{eq:dim2},  we have
\begin{equation}
\dim(\Z_{k,n}) = \sum_{t =1}^n \stirling{2k}{t} = \frac{1}{n!} \sum_{\sigma \in \S_n} \mathsf{F}(\sigma)^{2k}.
\end{equation}
We prove additional relations between Stirling numbers of the second kind and fixed points of permutations in Section \ref{subsec:multperm} (more specifically,
see Theorem \ref{T:cent}, Proposition \ref{P:fixed}, and Corollary \ref{C:fixedBell}). 

The diagram basis elements $d_\pi$ in $\P_{k+1}(n)$ corresponding to set partitions $\pi$  having  $k+1$ and $2(k+1)$ in the same block form a subalgebra  $\P_{k+\half}(n)$ of $\P_{k+1}(n)$  under diagram multiplication.
Identifying $\P_k(n)$ with the span of the diagrams in $\P_{k+\half}(n)$ which have a block consisting
of the two elements $k+1,2(k+1)$ gives 
a tower of algebras,  
\begin{equation}\label{eq:tower} \FF  = \P_0(n) \cong  \P_\half (n) \subset  \P_1(n) \subset \cdots \subset \P_k(n) \subset \P_{k+\half}(n) \subset \P_{k+1}(n) \subset  \cdots \  . \end{equation}
If we regard  $\M_n$ as a module for the symmetric group $\S_{n-1}$ by restriction, where elements of $\S_{n-1}$ fix the last basis vector $\vf_{n}$, then there is a surjective algebra homomorphism 
\begin{equation}
\Phi_{k+\half,n}:\P_{k+\half}(n) \rightarrow \End_{\S_{n-1}}(\modu^{\ot {k}})\cong
\End_{\S_{n-1}}(\modu^{\ot {k}} \ot \vf_n),
\end{equation} 
which is an isomorphism when $n \geq 2k+1$.    The intermediate algebras $\P_{k+\half}(n)$ have proven very useful in  
developing the structure and representation theory  of partition algebras (see for example, \cite{MR,HR}), and they come into play here in the construction of vacillating $k$-tableaux.  The centralizer algebras  $\End_{\S_{n-1}}(\modu^{\ot {k}})$ are also closely tied
to the restriction and induction functors that produce the Bratteli diagram ${\mathcal B}(\S_n,\S_{n-1})$ in Section \ref{subsec:BratSn}. \

For a group $\G$ and a finite-dimensional $\G$-module $\VV$,  the tensor product $\VV \ot \VV^{\ast}$ of $\VV$ with its dual module $\VV^{\ast}$
 is isomorphic as a $\G$-module to $\End(\VV)$ via the mapping $\vf \ot \phi \mapsto
 A_{\vf,\phi}$, where $A_{\vf,\phi}(\mathsf{u}) = \phi(\mathsf{u})\vf$. 
Since  $g.(\vf \ot \phi) = g.\vf \ot g.\phi$,   where $(g.\phi)(\mathsf{u}) = \phi(g^{-1}.\mathsf{u})$,
 and  $g.A = gAg^{-1}$ as transformations on $\VV$ for all $A \in \End(\VV)$, 
$$g.(\vf \ot \phi) = \vf \ot \phi \iff  gA_{\vf, \phi}g^{-1} = A_{\vf,\phi} \iff 
g A_{\phi,\vf} = A_{\phi,\vf}g \iff  A_{\vf, \phi} \in \End_{\G}(\VV).$$
Thus,  the space of $\G$-invariants, $(\VV \ot \VV^\ast)^\GG = \{ \vf \ot \phi \mid g.(\vf \ot \phi) = \vf \ot \phi\}$, is
isomorphic to the centralizer algebra  $\End_\G(\VV).$

In the particular case of the symmetric group $\S_n$ and its permutation module $\M_n$, we can identify the centralizer
 algebra $\End_{\S_n}(\M_n^{\ot k})$ with the space 
$$
(\M_n^{\ot 2k})^{\S_n} \cong \left(\M_n^{\ot k} \ot (\M_n^{\ot k})^\ast\right)^{\S_n}
$$
of $\S_n$-invariants,  as $\M_n$ is isomorphic to its dual  $\S_n$-module $\M_n^\ast$ (this was used  in (\ref{SW4}) and \eqref{eq:dim2}).   
 The fact that $\Phi_{k,n}: \P_k(n) \to \End_{\S_n}(\M_n^{\ot k})$ is a surjection tells us that $\P_k(n)$ generates all of the $\S_n$-tensor invariants
 $(\M_n^{\ot 2k})^{\S_n}$  and provides the First Fundamental Theorem of Invariant Theory for $\S_n$ (see Theorem \ref{T:1stfund} originally proved by Jones \cite{J}).   When $2k > n$, the surjection $\Phi_{k,n}$ has a nontrivial kernel.  As $2k$ increases in relation to $n$, the kernel becomes quite significant, and $\End_{\S_n}(\M_n^{\ot k})$ is only a  shadow of the full partition algebra. (This is illustrated in the table of  dimensions in Figure \ref{table:BellNumbers}.)  We have shown in \cite{BH}, that 
 when $2k > n$, the kernel of $\Phi_{k,n}$ is generated as a two-sided ideal by a single essential idempotent $\ef_{k,n}$ (see \eqref{eq:ef}).     Identifying $\End_{\S_n}(\M_n^{\ot k}) \cong \left(\M_n^{\ot 2k}\right)^{\S_n}$ with $\P_k(n)/\ker \Phi_{k,n}$ \  
 ($= \P_k(n)/ \langle \ef_{k,n} \rangle$ when $2k > n$),   we have the following:  

\begin{thm} {\rm \cite[Thm.~5.19]{BH}} \ (Second Fundamental Theorem of Invariant Theory for $\S_n$)  \ For all $k,n\in \ZZ_{\ge 1}$,   $\im \Phi_{k,n}
= \End_{\S_n}(\modu^{\ot k})$ is generated by the partition algebra generators and relations in Theorem \ref{T:present} (a)-(c) together with the one additional relation $\ef_{k,n} = 0$ in the case that $2k > n$.   When $k \ge n$,  the relation $\ef_{k,n} = 0$ can be replaced with  $\ef_{n,n} = 0$.
\end{thm} 

Classical Schur-Weyl duality provides an analogous result for the general linear group $\mathsf{GL}_n$ and its action by matrix multiplication on the space
$\VV$ on $n \times 1$ column vectors (more details can be found in \cite[Sec.~9.1]{GW}).  The surjection  $\FF \S_k \rightarrow  \End_{\mathsf{GL}_n}(\VV^{\ot k})$, given by the place permutation action of $\S_k$ on the tensor factors of $\VV$,   is an isomorphism if $n \geq k$, 
and in that case, $\FF \S_k \cong \left(\VV^{\ot k} \ot (\VV^{\ot k})^\ast \right)^{\mathsf{GL}_n}$.  
When $k \ge n+1$,  the kernel is
generated by a single essential idempotent $\sum_{\sigma \in \S_{n+1}}  (-1)^{\mathsf{sgn}(\sigma)} \sigma$  in the group algebra $\FF\S_k$.   Thus, the second fundamental theorem comes from the standard presentation by generators and relations for the symmetric group $\S_k$ by imposing the additional relation $\sum_{\sigma \in \S_{n+1}}  (-1)^{\mathsf{sgn}(\sigma)} \sigma = 0$
when $k \ge n+1$. 

In \cite{Br},  Brauer  introduced algebras, now known as \emph{Brauer algebras}, that centralize the action of the orthogonal group $\mathsf{O}_n$ and symplectic group $\mathsf{Sp}_{2n}$ on
tensor powers of their defining modules.  The surjective algebra homomorphisms
$\mathsf{B}_k(n) \to \End_{\mathsf{O}_n}(\VV^{\ot k})$ ($\dim(\VV)=n$) and  $\mathsf{B}_k(-2n) \to
\End_{\mathsf{Sp}_{2n}}(\VV^{\ot k})$ ($\dim(\VV) = 2n$) defined in \cite{Br} provide 
the First Fundamental Theorem of 
Invariant Theory for these groups.  (An exposition of these results appears in \cite[Sec.~4.3.2]{GW}.)  
Generators for the kernels of these surjections give the Second Fundamental
Theorem of Invariant Theory for the orthogonal and symplectic groups.  
As shown in the recent work of Hu and Xiao \cite{HX}, Lehrer and Zhang \cite{LZ1,LZ2}, and Rubey and Westbury \cite{RW1}, the kernels of these surjections are  principally generated by a single idempotent when $k \ge n+1$.   In \cite[Sec.~7.4]{RW1} (see also \cite{RW2}), Rubey and Westbury  show that the central idempotent $E = \big((n+1)!\big)^{-1}
\sum
d$, obtained by summing all  the Brauer diagrams $d$ with $2(n+1)$ nodes,
generates the kernel of
the surjection $\mathsf{B}_k(-2n) \to \End_{\mathsf{Sp}_{2n}} (\VV^{\ot k})$ for all $k \ge n+1$.   As a result,
they obtain the fundamental theorems of invariant theory for the symplectic groups from  
Brauer algebra considerations.     Bowman, Enyang, and Goodman \cite{BEG} adopt a cellular 
basis approach to describing the kernels in the orthogonal and symplectic cases, as well as in the case
of the general linear group $\mathsf{GL}_n$ acting on mixed tensor powers $\VV^{\ot k} \ot (\VV^*)^{\ot \ell}$
of its natural $n$-dimensional module $\VV$ and its dual $\VV^*$.   A surjection of the walled Brauer
algebra $\mathsf{B}_{k,\ell}(n) \to \End_{\mathsf{GL}_n}(\VV^{\ot k} \ot (\VV^*)^{\ot \ell})$  is used for this purpose. (The algebra $\mathsf{B}_{k,\ell}(n)$ and some of its representation-theoretic properties including its action on $\VV^{\ot k} \ot (\VV^*)^{\ot \ell}$ can be found, for example,  in \cite{BCHLLS}.) 

Although this  article is largely a survey discussing
recent work on partition algebras and the fundamental theorems
of invariant theory for symmetric groups, it features new results.
Among them are  a new bijection between vacillating tableaux and set-partition tableaux,   a new expression for the characters of  
partition algebras,  new identities relating Stirling numbers of the
second kind and fixed points of permutations, as well as a
general method for constructing the orbit basis for any permutation module of an arbitrary finite group $\G$.

 Our main point of emphasis is that the Schur-Weyl duality  afforded by the surjection
$\Phi_{k,n}: \P_k(n) \to \End_{\S_n}(\M_n^{\otimes k})$  furnishes an effective framework for
studying the symmetric group $\S_n$ and its invariant theory.   
Here   $n$ is fixed, and  
the  tensor power $k$ is allowed to grow arbitrarily large. When $2k > n$,  the representation $\Phi_{k,n}$
has a  nontrivial kernel, which can best be described using  the orbit basis of the partition algebra and which can be used
to give the Second Fundamental Theorem of Invariant Theory for $\S_n$.

\section{Restriction-induction  Bratteli diagrams and vacillating tableaux}\label{sec:res-ind}  
In this section, we discuss the restriction and induction functors for a group-subgroup pair $(\GG,\HH)$  and the Bratteli diagram that
comes from applying them and then specialize to the case of the pair
$(\S_n, \S_{n-1})$.    Further details can be found, for example,  in \cite[Sec.~4]{BHH}.   

\subsection{Generalities on restriction and induction} \label{subsec:genresind} 
Let $(\GG,\HH)$ be a pair consisting of a finite group $\GG$ and a subgroup $\HH$ of  $\GG$.  
 Let   $\UU^0$ be the trivial one-dimensional $\GG$-module, and assume for $k \in  \ZZ_{\ge 0}$ that  the $\GG$-module $\UU^{k}$ has been defined.    
Let $\UU^{k + \half}$ be the $\HH$-module defined by restriction to $\HH$:  \ $\UU^{k+\half} = \Res_{\HH}^{\GG}(\UU^k)$,  and 
then let  $\UU^{k+1}$ be  the $\GG$-module specified by induction to $\GG$: \   $\UU^{k+1} = \Ind_{\HH}^{\GG}(\UU^{k+\half}) = \FF\GG \ot_{\FF\HH}\UU^{k+\half}$.   In this way,  $\UU^\ell$ is defined inductively for all $\ell \in \half \ZZ_{\ge 0}$,   and $\UU^k = \big( \Ind_\HH^\GG  \Res_\HH^\GG\big)^k(\UU^0)$
for all $k \in \ZZ_{\ge 0}$.    The module  $\M := \mathsf{Ind}_\HH^\GG(\mathsf{Res}_\HH^\GG(\UU^0)) = \UU^1$  is isomorphic to $\GG/\HH$ as a $\GG$-module, where $\GG$ acts on the left cosets
of $\GG/\HH$ by multiplication.

For a $\GG$-module $\xx$ and an $\HH$-module  $\yy$, the ``tensor identity"  says that
\begin{equation}
\Ind^{\GG}_{\HH}(\Res^{\GG}_{\HH}(\xx) \otimes \yy) \cong \xx \otimes \Ind^{\GG}_{\HH}(\yy).
\end{equation}
The mapping $(g \otimes_{\FF\HH} x) \otimes y \mapsto g x \otimes (g \otimes_{\FF\HH} y)$ can be used to establish this isomorphism. (See, for example, \cite[(3.18)]{HR}.)
Hence, when $\xx= \UU^{k}$ and $\yy = \mathsf{Res}_\HH^\GG(\UU^0)$, this  implies that
\begin{equation*}
\begin{split}
\Ind^{\GG}_{\HH}(\Res^{\GG}_{\HH}(\UU^k)) & \cong \Ind^{\GG}_{\HH}(\Res^{\GG}_{\HH}(\UU^k)\otimes \mathsf{Res}_\HH^\GG(\UU^0))  \cong \UU^k \otimes \Ind^{\GG}_{\HH}( \mathsf{Res}_\HH^\GG(\UU^0)) = \UU^{k} \otimes \M.
\end{split}
\end{equation*}
Therefore, by induction, we have the following isomorphisms for all $k \in \ZZ_{\ge 0}$:  
\begin{equation}\label{eq:tensid}
 \M^{\otimes k} \cong \UU^k\quad (\text{as $\GG$-modules)} \qquad\text{and}\qquad  \Res^\GG_\HH(\M^{\otimes k}) \cong \UU^{k + \frac{1}{2}}\quad (\text{as $\HH$-modules}).
\end{equation}

Consider the centralizer algebra, 
$$\Z_{k,\GG}:=\End_\GG(\M^{\otimes k}) = \{ \varphi \in \End(\M^{\ot k}) \mid  \varphi(g.x) = g.\varphi(x)
\ \text{for all} \ x \in \M^{\ot k}, g \in \GG\}$$ 
of the $\GG$-action on $\M^{\ot k}$.   
It follows that there are algebra isomorphisms:
\begin{align}\begin{split}\label{eq:tensisos}  \Z_{k,\GG} &:=\End_\GG(\M^{\otimes k}) \cong \End_\GG(\UU^k), \\ 
\Z_{k+\half,\HH}& :=\End_\HH(\Res^\GG_\HH(\M^{\otimes k})) \cong \End_\HH(\UU^{k+\frac{1}{2}}).
\end{split}\end{align}    
Suppose for $k \in \ZZ_{\ge 0}$ that 

\begin{itemize}
\item[$\bullet$] $\Lambda_{k,\GG} \subseteq \Lambda_\GG$ indexes the irreducible $\GG$-modules, and hence also the irreducible $\Z_{k,\GG}$-modules, occurring  in $\UU^{k} \cong \M^{\otimes k}$;

\item[$\bullet$]  $\Lambda_{k+\half,\HH}\subseteq \Lambda_\HH$  indexes the irreducible $\HH$-modules, and hence also the irreducible $\Z_{k+\frac{1}{2},\HH}$-{modules}, occurring  in $\UU^{k + \frac{1}{2}}\cong\Res^\GG_\HH(\M^{\otimes k})$.
\end{itemize}

\subsection{The restriction-induction Bratteli diagram} \label{subsec:Bratdiag}

The \emph{restriction-induction Bratteli diagram} for the pair $(\GG,\HH)$  is an infinite, rooted tree $\mathcal{B}(\GG,\HH)$ whose vertices are organized into rows labeled by half integers $\ell$  in $\half\ZZ_{\ge 0}$.  For $\ell =k \in \ZZ_{\ge 0}$, the vertices on row
$k$ are the elements of $\Lambda_{k,\GG}$, and the vertices on row $\ell = k+\half$ are the elements of $\Lambda_{k+\half,\HH}$. 
The vertex on row $0$ is the root, the label of the trivial $\GG$-module, and the vertex on row $\half$ is  the label
of the trivial $\HH$-module.

 The edges of $\mathcal{B}(\GG,\HH)$ are constructed from the restriction and induction rules for $\HH\subseteq \GG$. Let 
$\{\GG^\lambda\}_{\lambda \in \Lambda_\GG}$ and $\{\HH^\alpha\}_{\alpha \in \Lambda_\HH}$ be the sets of irreducible modules for these
groups over $\FF$. By Frobenius reciprocity, the multiplicity $c^\lambda_\alpha$ of $\HH^\alpha$ in the restricted module $\Res_{\HH}^\GG ( \GG^\lambda)$ equals the multiplicity of $\G^\lambda$ in the induced module $\Ind_{\HH}^\GG(\HH^\alpha)$, and thus
\begin{equation}\label{eq:ResInd}
\Res_{\HH}^\GG ( \GG^\lambda) = \bigoplus_{\alpha  \in \Lambda_\HH}  c^\lambda_\alpha\,   \HH^\alpha
\qquad\hbox{and}\qquad
\Ind_{\HH}^\GG(\HH^\alpha) = \bigoplus_{\lambda \in \Lambda_\GG} c^\lambda_\alpha\, \GG^\lambda.
\end{equation}
In  $\mathcal{B}(\GG,\HH)$ we draw $c^\lambda_\alpha$ edges from $\lambda \in \Lambda_{k,\GG}$ to $\alpha \in \Lambda_{k+\frac{1}{2},\HH}$ and $c^\lambda_\alpha$ edges from $\alpha \in \Lambda_{k+\frac{1}{2},\HH}$ to $\lambda \in \Lambda_{k+1,\GG}$.  The Bratteli diagram is constructed in such a way that 
 \begin{itemize}
\item[$\bullet$] the number of paths from the root at level 0 to $\lambda \in\Lambda_{k,\GG}$ equals the multipicity of $\GG^\lambda$ in $\UU^{k}\cong\M^{\otimes k}$ and thus also equals the dimension of the irreducible $\Z_{k,\GG}$-module  $\Z_{k,\GG}^\lambda$ by (\ref{SW2}) (these numbers are computed in Pascal-triangle-like fashion and are placed beneath each vertex);

\item[$\bullet$] the number of paths  from the root at level 0 to $\alpha \in\Lambda_{k+\frac{1}{2},\HH}$ equals the multiplicity of $\HH^\alpha$ in $\UU^{k + \frac{1}{2}}$ and thus also equals the dimension of
the $\Z_{k+\half,\HH}$-module  $\Z_{k+\half,\HH}^\alpha$ (and is indicated below each vertex);

\item[$\bullet$] the sum of the squares of the labels on row $k$ (resp. row $k+\half$)  equals $\dim(\Z_{k,\GG})$ (resp. $\dim(\Z_{k+\half,\HH})$.
\end{itemize}
The restriction-induction Bratteli diagram for the pair $(\S_5,\S_4)$ is displayed in  Figure \ref{fig:Sbratteli}.

\subsection{Restriction and induction for the symmetric group pair  $(\S_n, \S_{n-1})$} \label{subsec:BratSn}

Assume $\lambda = [\lambda_1, \lambda_2, \dots, \lambda_n]$ is a partition of $n$ with parts $
\lambda_1 \geq \lambda_2 \geq \ldots \geq \lambda_n \ge 0$,  and identify $\lambda$ with its Young diagram.    
Thus, if $\lambda = [6,5,3,3] \vdash 17$,  then

\begin{center}{$\lambda= [6,5,3,3]=\begin{array}{c} 
{\begin{tikzpicture}[scale=.35,line width=5pt] 
\tikzstyle{Pedge} = [draw,line width=.7pt,-,black]    
\foreach \i in {1,...,7} 
{\path (\i,1) coordinate (T\i); 
\path (\i,0) coordinate (B\i);
\foreach \i in {1,...,6}  \path (\i,-1) coordinate (S\i);
\foreach \i in {1,...,4}  \path (\i,-2) coordinate (R\i);  
\foreach \i in {1,...,4}  \path (\i,-3) coordinate (Q\i);  
} 
\filldraw[fill= white!10,draw=white!10,line width=8pt]   (T1) -- (T7) -- (B7) -- (B1) -- (T1);
\filldraw[fill= white!10,draw=white!10,line width=8pt]   (B1) -- (B5) -- (S5) -- (S1) -- (B1);
\filldraw[fill= white!10,draw=white!10,line width=8pt]   (S1) -- (S4) -- (R4) -- (R1) -- (S1);
\filldraw[fill= white!10,draw=white!10,line width=8pt]   (R1) -- (R4) -- (Q4) -- (Q1) -- (R1);
\filldraw[fill= gray!60,draw=gray!60,line width=.35pt]   (B2) -- (B3) -- (S3) -- (S2) -- (B2);
\path (T1) edge[Pedge] (Q1);
\path(T2)  edge[Pedge] (Q2);
\path (T3)  edge[Pedge] (Q3);
\path (T4)  edge[Pedge] (Q4);
\path (T5)  edge[Pedge] (S5);
\path (T6)  edge[Pedge] (S6);
\path (T7)  edge[Pedge] (B7);
\path(T1) edge[Pedge] (T7);
\path(B1) edge[Pedge] (B7);
\path(S1) edge[Pedge] (S6);
\path(R1) edge[Pedge] (R4); 
\path(Q1) edge[Pedge] (Q4); 
\end{tikzpicture}}\end{array}.$}\end{center} 
\noindent The  \emph{hook length}  $h(b)$ of a box $b$ in the diagram  is 
 1 plus the number of boxes to the right of $b$ in the same row, plus  the number of boxes below  $b$ in the same column,  and  $h(b) = 1+ 3 + 2 = 6$ for the shaded box above.   
The dimension $f^\lambda$  of the irreducible $\S_n$-module $\S_n^\lambda$  can be computed by the  hook-length formula,  
\begin{equation}\label{eq:hook} f^\lambda = \frac{n!}{\prod_{b \in \lambda} h(b)},\end{equation}
where the denominator is the product of the hook lengths as $b$ ranges over the boxes in the Young diagram of $\lambda$. 
This is  equal to the number of standard Young tableaux of shape $\lambda$, where a standard Young tableau $T$ is a filling of the boxes in the Young diagram of $\lambda$ with the numbers $\{1,2,\ldots,n\}$ such that the entries increase in every row from left to right  and in every column from top to bottom.

The restriction and induction rules for  irreducible symmetric group modules $\S_n^\lambda$ are well known (and can be found, for example, in \cite[Thm.~2.43]{JK}):
\begin{equation}\label{eq:RI}  \mathsf{Res}^{\S_n}_{\S_{n-1}}(\S_n^\lambda)  = \bigoplus_{\mu = \lambda-\square} \S_{n-1}^\mu, 
\qquad  \mathsf{Ind}^{\S_{n+1}}_{\S_{n}}(\S_n^\lambda)  = \bigoplus_{\kappa = \lambda+\square} \S_{n+1}^\kappa, \end{equation}
where the first sum is over all partitions $\mu$ of $n-1$ obtained from $\lambda$ by removing a box from the end of a row of the diagram of
$\lambda$, and
the second sum is over all partitions $\kappa$ of $n+1$ obtained by adding a box to the end of a row of $\lambda$.

Applying these rules to the trivial one-dimensional $\S_n$-module $\S_n^{[n]}$ we see that
$$ \Ind^{\S_n}_{\S_{n-1}}\big(\Res^{\S_n}_{\S_{n-1}}(\S_n^{[n]})\big)  = 
 \Ind^{\S_n}_{\S_{n-1}}\big(\S_{n-1}^{[n-1]}\big) =
 \S_n^{[n]} \oplus \S_n^{[n-1,1]} \cong \M_n.$$ 
Thus, in the notation of  Section \ref{subsec:genresind},   $\UU^1$ is the permutation module $\M_n$,  and by \eqref{eq:tensid},  
\begin{equation}\label{eq:tensid2}
 \M_n^{\otimes k} \cong \UU^k\ \  (\text{as $\S_n$-modules)} \quad\text{and}\quad  \Res^{\S_n}_{\S_{n-1}}(\M_n^{\otimes k}) \cong \UU^{k + \frac{1}{2}}\ \  (\text{as $\S_{n-1}$-modules}).
\end{equation}
Then \eqref{eq:tensisos} implies that the centralizer algebras are given by 
\begin{align} \Z_{k,n} &:=  \End_{\S_n}(\modu^{\otimes k})\cong\End_{\S_n}(\UU^k) \\
\Z_{k+\half,n}  &:=  \End_{\S_{n-1}}(\modu^{\otimes k})\cong\End_{\S_{n-1}}(\UU^k),\end{align}
where we are writing $\Z_{k,n}$ rather than $\Z_{k,\S_n}$ and
$\Z_{k+\half,n}$  instead of  $\Z_{k+\half,\S_{n-1}}$  to simplify the notation.  (\emph{The use of $n$  
in place of the more natural choice of $n-1$ for the second one should be especially noted.}) 

 From our discussions in Sections \ref{subsec:Bratdiag} and \ref{subsec:BratSn},  we know that the following holds.  
 \medskip

\emph{If  $k,n \in \ZZ_{\ge 0}$ and $n \geq 1$, then for all $\lam \in \Lambda_{k,\S_n}$,
\begin{equation}\label{eq:pathdim}
\dim(\Z_{k,n}^\lambda)  =  \m_{k,n}^\lambda  =  
\big \vert \big\{\text{paths in $\mathcal{B}(\S_n,\S_{n-1})$ from $[n]$ at level 0 to $\lambda$ at level $k$}\big\} \big \vert.
\end{equation}
}

\Yboxdim{6pt}  
\Ylinethick{.6pt} 
\begin{figure}[h!]
$$
\begin{tikzpicture}[line width=.5pt,xscale=0.12,yscale=0.22]  
\path (-15,0)  node[anchor=west]  {$\ell = 0$};
\path (0,0)  node[anchor=west] (S5-0) {\yng(5)};
\draw (S5-0) node[below=4pt, black] {\color{blue}\nmm 1};
\path (-15,-5)  node[anchor=west]  {$\ell = \half$};
\path (5,-5)  node[anchor=west] (S4-1) {\yng(4)};
\draw (S4-1) node[below=4pt,black] {\color{blue}\nmm 1};
\path (-15,-10)  node[anchor=west]  {$\ell = 1$};
\path (0,-10)  node[anchor=west] (S5-2) {\yng(5)};
\path (15,-10)  node[anchor=west] (S41-2) {\yng(4,1)};
\draw (S5-2) node[below=4pt,black] {\color{blue}\nmm 1};
\draw (S41-2) node[below=4pt,black] {\color{blue}\nmm 1};
\path (-15,-15)  node[anchor=west]  {$\ell= 1\half$};
\path (5,-15)  node[anchor=west] (S4-3) {\yng(4)};
\path (22,-15)  node[anchor=west] (S31-3) {\yng(3,1)};
\draw (S4-3) node[below=4pt, black] {\color{blue}\nmm 2};
\draw (S31-3) node[below=4pt,black] {\color{blue}\nmm 1};
\path (-15,-20)node[anchor=west]  {$\ell=2$};
\path (0,-20)  node[anchor=west] (S5-4) {\yng(5)};
\path (15,-20)  node[anchor=west] (S41-4) {\yng(4,1)};
\path (27,-20)  node[anchor=west] (S32-4) {\yng(3,2)};
\path (38,-20)  node[anchor=west] (S311-4) {\yng(3,1,1)};
\draw (S5-4) node[below=5pt, black] {\color{blue}\nmm 2};
\draw (S41-4) node[below=5pt, black] {\color{blue}\nmm 3};
\draw (S32-4) node[below=5pt, black] {\color{blue}\nmm 1};
\draw (S311-4) node[below=5pt, black] {\quad\color{blue}\nmm 1};
\path (-15,-25) node[anchor=west]  {$\ell = 2\half$};
\path (5,-25)  node[anchor=west] (S4-5) {\yng(4)};
\path (22,-25)  node[anchor=west]  (S31-5) {\yng(3,1)};
\path (38,-25)  node[anchor=west] (S22-5) {\yng(2,2)};
\path (48,-25)  node[anchor=west] (S211-5) {\yng(2,1,1)};
\draw (S4-5) node[below=5pt,black] {\, \color{blue}\nmm 5};
\draw (S31-5) node[below=5pt,black] {\color{blue}\nmm 5};
\draw (S22-5) node[below=5pt, black]  {\color{blue}\nmm 1};
\draw (S211-5) node[below=5pt, black] {\color{blue}\quad \nmm 1};
\path (-15,-31)  node[anchor=west]  {$\ell=3$};
\path (0,-31)  node[anchor=west] (S5-6) {\yng(5)};
\path (15,-31)  node[anchor=west] (S41-6) {\yng(4,1)};
\path (27,-31)  node[anchor=west] (S32-6) {\yng(3,2)};
\path (38,-31)  node[anchor=west] (S311-6) {\yng(3,1,1)};
\path (48,-31)  node[anchor=west] (S221-6) {\yng(2,2,1)};
\path (57,-31)  node[anchor=west] (S2111-6) {\yng(2,1,1,1)};
\draw (S5-6) node[below=5pt,black] {\!\color{blue}\nmm 5};
\draw (S41-6) node[below=5pt, black] {\!\color{blue}\nmm 10};
\draw (S32-6) node[below=5pt, black] {\color{blue}\nmm 6};
\draw (S311-6) node[below=5pt,  black]  {\color{blue}\nmm 6};
\draw (S221-6) node[below=5pt, black] {\color{blue}\quad \nmm 2};
\draw (S2111-6) node[below=5pt,black] {\color{blue}\quad \nmm 1};
\path (-15,-37)  node[anchor=west]  {$\ell = 3\half$};
\path (5,  -37)  node[anchor=west] (S4-7) {\yng(4)};
\path (22,-37)  node[anchor=west] (S31-7) {\yng(3,1)};
\path (38,-37)  node[anchor=west] (S22-7) {\yng(2,2)};
\path (48,-37)  node[anchor=west] (S211-7) {\yng(2,1,1)};
\path (57,-37)  node[anchor=west] (S1111-7) {\yng(1,1,1,1)};
\draw (S4-7) node[below=5pt,black] {\,\color{blue}\,\nmm 15};
\draw (S31-7) node[below=6pt,black] {\color{blue}\nmm 22};
\draw (S22-7) node[below=5pt, black] {\color{blue}\nmm 8};
\draw (S211-7) node[below=5pt, black] {\color{blue}\quad\nmm 9};
\draw (S1111-7) node[below=6pt, black] {\color{blue}\,\quad\nmm 1};
\path (-15,-43)  node[anchor=west]  {$\ell= 4$};
\path (0,-43)  node[anchor=west] (S5-8) {\yng(5)};
\path (15,-43)  node[anchor=west] (S41-8) {\yng(4,1)};
\path (27,-43)  node[anchor=west] (S32-8) {\yng(3,2)};
\path (38,-43)  node[anchor=west] (S311-8) {\yng(3,1,1)};
\path (48,-43)  node[anchor=west] (S221-8) {\yng(2,2,1)};
\path (57,-43)  node[anchor=west] (S2111-8) {\yng(2,1,1,1)};
\path (66,-43)  node[anchor=west] (S11111-8) {\yng(1,1,1,1,1)};
\draw (S5-8) node[below=4pt, black] {\color{blue}\,\nmm 15};
\draw (S41-8) node[below=4pt, black] {\color{blue}\,\nmm 37};
\draw (S32-8) node[below=5pt, black] {\color{blue}\,\nmm 30};
\draw (S311-8) node[below=4pt, black] {\color{blue}\quad\nmm 31};
\draw (S221-8) node[below=4pt, black] {\color{blue}\qquad\nmm 17};
\draw (S2111-8) node[below=4pt, black] {\color{blue}\qquad\,\nmm 10};
\draw (S11111-8) node[below=4pt, black] {\color{blue}\qquad\,\,\nmm 1};
\path (-15,-49)  node[anchor=west]  {$\ell = 4\half$};
\path (5,  -49)  node[anchor=west] (S4-9) {\yng(4)};
\path (22,-49)  node[anchor=west] (S31-9) {\yng(3,1)};
\path (38,-49)  node[anchor=west] (S22-9) {\yng(2,2)};
\path (48,-49)  node[anchor=west] (S211-9) {\yng(2,1,1)};
\path (57,-49)  node[anchor=west] (S1111-9) {\yng(1,1,1,1)};
\draw (S4-9) node[below=5pt,black] {\,\color{blue}\,\nmm 52};
\draw (S31-9) node[below=8pt,black] {\color{blue}\nmm 98};
\draw (S22-9) node[below=5pt, black] {\color{blue}\nmm 47};
\draw (S211-9) node[below=0pt, black] {\color{blue}\qquad\nmm 58};
\draw (S1111-9) node[below=0pt, black] {\color{blue}\ \qquad\nmm 11};
\path (-15,-55)  node[anchor=west]  {$\ell= 5$};
\path (0,-55)  node[anchor=west] (S5-10) {\yng(5)};
\path (15,-55)  node[anchor=west] (S41-10) {\yng(4,1)};
\path (27,-55)  node[anchor=west] (S32-10) {\yng(3,2)};
\path (38,-55)  node[anchor=west] (S311-10) {\yng(3,1,1)};
\path (48,-55)  node[anchor=west] (S221-10) {\yng(2,2,1)};
\path (57,-55)  node[anchor=west] (S2111-10) {\yng(2,1,1,1)};
\path (66,-55)  node[anchor=west] (S11111-10) {\yng(1,1,1,1,1)};
\draw (S5-10) node[below=4pt, black] {\color{blue}\,\nmm 52};
\draw (S41-10) node[below=4pt, black] {\color{blue}\,\nmm 150};
\draw (S32-10) node[below=5pt, black] {\color{blue}\,\nmm 145};
\draw (S311-10) node[below=4pt, black] {\color{blue}\quad\nmm 156};
\draw (S221-10) node[below=4pt, black] {\color{blue}\qquad\nmm 105};
\draw (S2111-10) node[below=4pt, black] {\color{blue}\qquad\,\nmm 69};
\draw (S11111-10) node[below=4pt, black] {\color{blue}\qquad\,\,\nmm 11};
\path   (-12.5,-58.5)  node[anchor=west]  {$\vdots$};
\path   (5,-58.5)  node[anchor=west]  {$\vdots$};
\path (17,-58.5)  node[anchor=west]  {$\vdots$};
\path (29,-58.5)  node[anchor=west]  {$\vdots$};
\path (39,-58.5)  node[anchor=west]  {$\vdots$};
\path (49,-58.5)  node[anchor=west]  {$\vdots$};
\path (57.5,-58.5)  node[anchor=west]  {$\vdots$};
\path (66.5,-58.5)  node[anchor=west]  {$\vdots$};
\path  (S5-0) edge[black,thick] (S4-1);
\path  (S5-2) edge[black,thick] (S4-1);
\path  (S41-2) edge[black,thick] (S4-1);
\path  (S5-2) edge[black,thick] (S4-3);
\path  (S41-2) edge[black,thick] (S4-3);
\path  (S41-2) edge[black,thick] (S31-3);
\path  (S5-4) edge[black,thick] (S4-3);
\path  (S41-4) edge[black,thick] (S4-3);
\path  (S41-4) edge[black,thick] (S31-3);
\path  (S32-4) edge[black,thick] (S31-3);
\path  (S311-4) edge[black,thick] (S31-3);
\path  (S5-4) edge[black,thick] (S4-5);
\path  (S41-4) edge[black,thick] (S4-5);
\path  (S41-4) edge[black,thick] (S31-5);
\path  (S32-4) edge[black,thick] (S31-5);
\path  (S32-4) edge[black,thick] (S22-5);
\path  (S311-4) edge[black,thick] (S31-5);
\path  (S311-4) edge[black,thick] (S211-5);
\path  (S5-6) edge[black,thick] (S4-5);
\path  (S41-6) edge[black,thick] (S4-5);
\path  (S41-6) edge[black,thick] (S31-5);
\path  (S32-6) edge[black,thick] (S31-5);
\path  (S32-6) edge[black,thick] (S22-5);
\path  (S311-6) edge[black,thick] (S31-5);
\path  (S221-6) edge[black,thick] (S22-5);
\path  (S311-6) edge[black,thick] (S211-5);
\path  (S221-6) edge[black,thick] (S211-5);
\path  (S2111-6) edge[black,thick] (S211-5);
\path  (S5-6) edge[black,thick] (S4-7);
\path  (S41-6) edge[black,thick] (S4-7);
\path  (S41-6) edge[black,thick] (S31-7);
\path  (S311-6) edge[black,thick] (S211-7);
\path  (S32-6) edge[black,thick] (S31-7);
\path  (S32-6) edge[black,thick] (S22-7);
\path  (S221-6) edge[black,thick] (S22-7);
\path  (S311-6) edge[black,thick] (S31-7);
\path  (S221-6) edge[black,thick] (S211-7);
\path  (S221-6) edge[black,thick] (S211-7);
\path  (S2111-6) edge[black,thick] (S211-7);
\path  (S2111-6) edge[black,thick] (S1111-7);
\path  (S2111-8) edge[black,thick] (S1111-7);
\path  (S2111-8) edge[black,thick] (S211-7);
\path  (S11111-8) edge[black,thick] (S1111-7);
\path  (S221-8) edge[black,thick] (S211-7);
\path  (S221-8) edge[black,thick] (S22-7);
\path  (S311-8) edge[black,thick] (S211-7);
\path  (S311-8) edge[black,thick] (S31-7);
\path  (S32-8) edge[black,thick] (S22-7);
\path  (S32-8) edge[black,thick] (S31-7);
\path  (S5-8) edge[black,thick] (S4-7);
\path  (S41-8) edge[black,thick] (S4-7);
\path  (S41-8) edge[black,thick] (S31-7);
\path  (S5-8) edge[black,thick] (S4-9);
\path  (S41-8) edge[black,thick] (S4-9);
\path  (S41-8) edge[black,thick] (S31-9);
\path  (S311-8) edge[black,thick] (S211-9);
\path  (S32-8) edge[black,thick] (S31-9);
\path  (S32-8) edge[black,thick] (S22-9);
\path  (S221-8) edge[black,thick] (S22-9);
\path  (S311-8) edge[black,thick] (S31-9);
\path  (S221-8) edge[black,thick] (S211-9);
\path  (S221-8) edge[black,thick] (S211-9);
\path  (S2111-8) edge[black,thick] (S211-9);
\path  (S2111-8) edge[black,thick] (S1111-9);
\path  (S11111-8) edge[black,thick] (S1111-9);
\path  (S2111-10) edge[black,thick] (S1111-9);
\path  (S2111-10) edge[black,thick] (S211-9);
\path  (S11111-10) edge[black,thick] (S1111-9);
\path  (S221-10) edge[black,thick] (S211-9);
\path  (S221-10) edge[black,thick] (S22-9);
\path  (S311-10) edge[black,thick] (S211-9);
\path  (S311-10) edge[black,thick] (S31-9);
\path  (S32-10) edge[black,thick] (S22-9);
\path  (S32-10) edge[black,thick] (S31-9);
\path  (S5-10) edge[black,thick] (S4-9);
\path  (S41-10) edge[black,thick] (S4-9);
\path  (S41-10) edge[black,thick] (S31-9);
\draw (90,-1) node[anchor=east,black] {$~~{\color{black}\mathbf{1}}$};
\draw (90,-6) node[anchor=east,black] {$~~{\color{black}\mathbf{1}}$};
\draw (90,-11) node[anchor=east,black] {$~~{\color{black}\mathbf{2}}$};
\draw (90,-15.5) node[anchor=east,black] {$~~{\color{black}\mathbf{5}}$};
\draw (90,-20) node[anchor=east,black] {$~~{\color{black}\mathbf{15}}$};
\draw (90,-25) node[anchor=east,black] {$~~{\color{black}\mathbf{52}}$};
\draw (90,-31) node[anchor=east,black] {$~~{\color{black}\mathbf{202}}$};
\draw (90,-37) node[anchor=east,black] {$~~{\color{black}\mathbf{855}}$};
\draw (90,-43) node[anchor=east,black] {$~~{\color{black}\mathbf{3845}}$};
\draw (90,-49) node[anchor=east,black] {$~~{\color{black}\mathbf{18002}}$};
\draw (90,-55) node[anchor=east,black] {$~~{\color{black}\mathbf{86472}}$};
\end{tikzpicture}
$$
\caption{Levels  \ $0,\half,1,\ldots,3\half,4$ \  of the Bratteli diagram for the pair $(\S_5, \S_4)$.
All of the partitions of 5  appear on level $\ell = 4$, and the structure of the diagram stabilizes
for $\ell \ge 4$.
  \label{fig:Sbratteli}}
\end{figure}
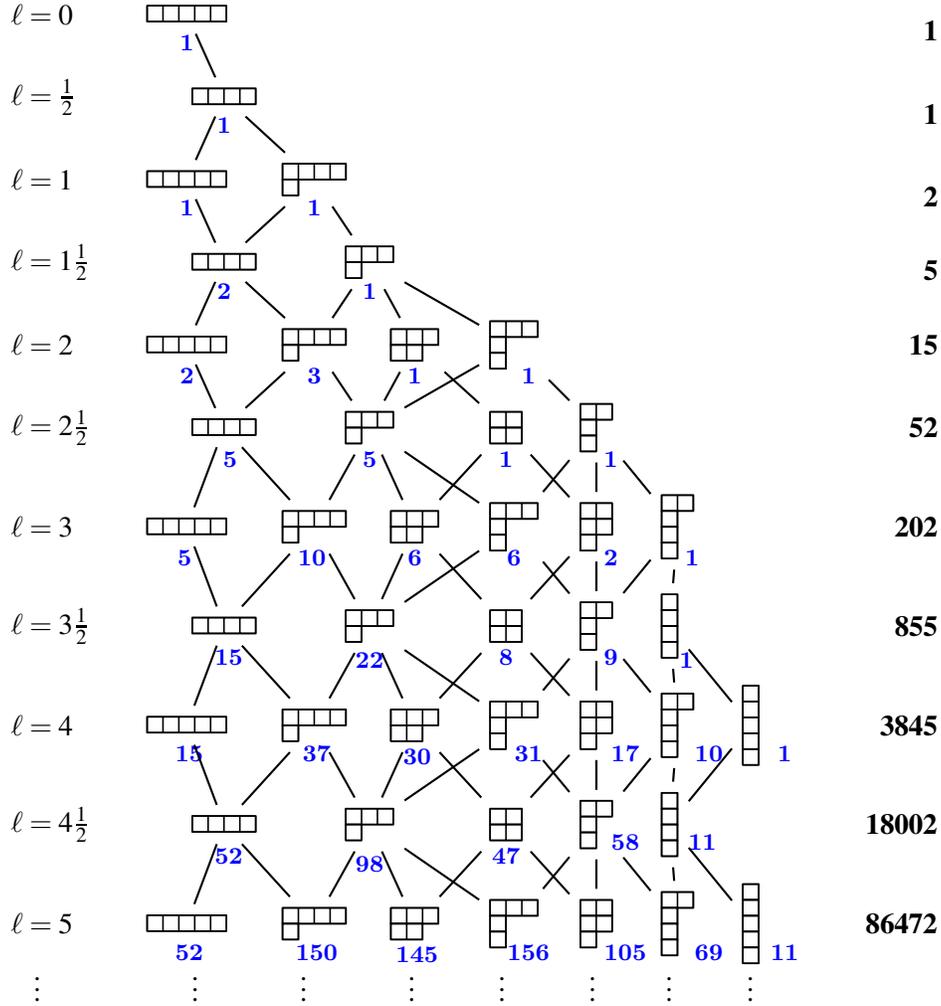

In Figure \ref{fig:Sbratteli}, we display the top portion of the restriction-induction Bratteli diagram $\mathcal{B}(\S_5,\S_4)$.  Below each partition, we record the number of paths from the top of the diagram to that partition.  For
integer values of $k$, the number beneath the partition  $\lambda\vdash n$ represents the multiplicity $\m_{k,n}^\lambda$ of the irreducible $\S_n$-module
$\S_n^\lambda$,  in $\M_n^{\ot k}$ (with $n=5$ in this particular example).  For values $k +\half$,  it indicates
the  multiplicity of $\S_{n-1}^\mu$, $\mu \vdash n-1$,  in the restriction of $\M_n^{\ot k}$ to $\S_{n-1}$.  
The number on the right of each line is $\dim(\Z_{k,n})$ (or $\dim(\Z_{k+\half,n})$), which is
the sum of the squares of the subscripts on the line. 

 The module $\M_n$ is faithful, so by Burnside's theorem every irreducible $\S_n$-module appears in $\M_n^{\otimes k}$, for some $k \ge 0$. In the example of Figure \ref{fig:Sbratteli},  all of the irreducible  $\S_5$-modules appear in $\M_5^{\otimes 4}$ (i.e., all of the partitions of 5 appear in $\mathcal{B}(\S_5,\S_4)$ at level $\ell = 4$). More generally, all of the partitions  $\lambda$ of $n$ appear in $\M_n^{\otimes n-1}$. This can be seen by taking any standard tableau $T$ of shape $\lambda$ and removing boxes in succession
starting with the one containing $n$ and proceeding down to the box with $2$.  Each removed box is placed at the end of the
first row as  it is removed. If the box is already in the first row, then remove the box at the end of the first row, and put it back at the end of the first row. Read in reverse, this sequence of partitions  obtained in this way will determine a path of length $2(n-1)$  from $[n]$ at level $1$ to $\lambda$ at level
$n-1$ in $\mathcal{B}(\S_n,\S_{n-1})$,  counting the removal of a box as one step and the adjoining of the box to the first row as another.
Furthermore, if $k \in \ZZ_{\ge 0}$ and if $\lambda\in\Lambda_{k,\S_n}$ for some integer $k$, then $\lambda\in\Lambda_{k+1,\S_n}$, since we can always remove any removable box in $\lambda$ and place it back in its same position. Thus,
\begin{equation}
\Lambda_{k,\S_n} = \Lambda_{\S_n} = \{ \lambda \mid \lambda \vdash n \} \, \ \ \text{ for all \ $k \ge n-1$},
\end{equation}
and  all partitions of $n$ occur  at level $k$, and all partitions of $n-1$ occur at level $k-\half$  in the Bratteli diagram for all $k \ge n-1$.  The first time the partition $[1^n]$ appears in $\mathcal{B}(\S_n,\S_{n-1})$ is when $k = n - 1$, since there is exactly one path from $[n]$ at level 0 to $[1^n]$ at level $k = n-1$ given by removing a box from row 1 at each  half-integer step and placing it at the bottom of the first column of the diagram
at the next integer step. Thus, $k \ge n-1$ is a necessary and sufficient condition for $\Lambda_{k,\S_n} = \Lambda_{\S_n}$.

Paths in the Bratteli diagram $\mathcal{B}(\S_n,\S_{n-1})$ give the dimension $\m_{k,n}^\lambda$ of the irreducible module 
$\Z_{k,n}^\lambda$ of the centralizer algebra $\Z_{k,n}$ in the same way that paths in Young's lattice \cite[Sec.~5.1]{Sa} give the dimension of the symmetric group module $\S_n^\lambda$. For this reason, we make the following definition (see also \cite{HL}, \cite{CDDSY}, \cite{BHH}, where this  definition is used).

\begin{definition} 
Let $k \in  \ZZ_{\ge 0}$ and let $\lambda \in \Lambda_{k,\S_n}$.  Each path of length $k$ from $[n]$ at level $0$ to $\lambda$ at level $k$ in the Bratteli diagram $\mathcal{B}(\S_n,\S_{n-1})$  determines a \emph{vacillating tableau $\underline \upsilon$ of shape $\lambda$ and length $k$}, which is  an alternating sequence
$$
\underline \upsilon =\left(\lambda^{(0)} = [n],\, \lambda^{(\half)} = [n-1],\, \lambda^{(1)},\lambda^{(1+\half)}, \ldots, \lambda^{(k-\half)},\lambda^{(k)} = \lambda\right)
$$
 of partitions starting at $\lambda^{(0)} = [n]$  and terminating at the partition $\lambda^{(k)} = \lambda$  
 such that $\lambda^{(i)} \in \Lambda_{i,n}$, \, $\lambda^{(i+\half)} \in \Lambda_{i+\half,{n-1}}$ for each
 integer $0 \le i < k$, and 
 \begin{enumerate}
\item[(a)] $\lambda^{(i+\half)} = \lambda^{(i)} - \square$,
\item[(b)] $\lambda^{(i+1)} = \lambda^{(i +\half)} + \square$.
\end{enumerate} 
\end{definition}

There is an easy bijection between paths to $\lambda$ of length $k$ in Young's lattice to standard tableaux of shape $\lambda \vdash k$. In Section \ref{subsec:biject}, we describe a bijection between vacillating tableaux of shape  $\lambda$ and length $k$ and set-partition tableaux 
of shape $\lambda$ whose nonzero entries are $1,2, \ldots,k$.
Since paths in the Bratteli diagram correspond to vacillating tableaux, \eqref{eq:pathdim} tells us the following:
\medskip

\emph{If $k,n \in \ZZ_{\ge 0}$ and $n \geq 1$, then  for all $\lam \in \Lambda_{k,\S_n}$
\begin{equation}
\dim(\Z_{k,n}^\lambda)  =  \m_{k,n}^\lambda  =  
\big \vert \big\{\text{vacillating tableaux of shape $\lambda$ and length $k$}\big\} \big \vert.
\end{equation}}

\section{Set partitions}\label{sec:setpart}

For $k \in \ZZ_{\ge 1}$, let  $[1,2k] = \{1,2,\ldots,2k\}$,  as before.
The set partitions in \begin{equation}
\begin{array}{rcl}\label{def:setparts}
\Pi_{2k} &=& \left\{ \text{\,set partitions of $[1,2k]$\,} \right\}, \\
\Pi_{2k-1} &=& \left\{ \pi \in \Pi_{2k} \mid k \text{ and } 2k \text{\,are in the same block of\,} \pi \right\}.
\end{array}
\end{equation}
 index bases of the partition algebras $\P_k(n)$ and $\P_{k+\half}(n)$, respectively. 
We let $|\pi|$ equal the number of blocks of $\pi$. 
For example, if
\begin{equation}\label{ex:setparts}
\begin{array}{rcl}
\pi &=& \big\{1,8,9,10 \,|\,  2,3 \,|\, 4,7\,|\, 5,6,11,12,14\,|\, 13\big\} \in \Pi_{14}, \\
\varrho &=& \big\{1,8,9,10\,|\,  2,3 \,|\,4 \,|\,  5,6,11,12 \,|\, 7,13, 14\big\} \in \Pi_{13} \subseteq \Pi_{14},
\end{array}
\end{equation}
then $\pi \not \in \Pi_{13}$ and  $|\pi| = |\varrho|=5$.   

For $\ell,n \in \Z_{\ge 1}$, define
\begin{equation}\label{def:restrictedsetparts}
\Pi_{\ell,n} = \{ \pi \in \Pi_\ell \mid \pi \text{ has at most $n$ blocks } \}. 
\end{equation}
Then the cardinalities of these sets are the Bell numbers: $|\Pi_{\ell,n}| = \mathsf{B}(\ell,n)$ and $|\Pi_{\ell}| = \mathsf{B}(\ell)$. We refer to the $\mathsf{B}(\ell,n) = \sum_{t = 1}^n \stirlinginline{\ell}{t}$ as  an \emph{$n$-restricted Bell number}.

\subsection{Multiplicities from a permutation module perspective}\label{subsec:multperm}
We begin by discussing a second approach to decomposing $\M_n^{\ot k}$ using permutation modules
for $\S_n$, and then describe the connections with set-partition tableaux.   

If  $\{\vf_1,\vf_2, \dots, \vf_n\}$ is the basis of $\M_n$ that $\S_n$ permutes, then for each set partition
$\pi$ of $\{1,2, \dots, k\}$,  
\begin{equation}\label{eq:PermutationCondition}
\M^\pi = \text{span}_\CC\left\{\vf_{j_1} \ot \vf_{j_2} \ot \cdots \ot \vf_{j_k} \mid  j_a = j_b\ \iff\ a, b \text{ are in the same block of } \pi
\right\}
\end{equation}
is an $\S_n$-submodule of $\M_n^{\otimes k}$.
To see this, recall that  $\sigma \in \S_{n}$ acts diagonally on simple tensors,
$\sigma.(\vf_{j_1} \ot \vf_{j_2} \ot \cdots \ot \vf_{j_k}) =  \vf_{\sigma(\vf_{j_1})} \ot \cdots \ot \vf_{\sigma(\vf_{j_k})},$ and so it preserves the condition in
\eqref{eq:PermutationCondition}.

For  $1 \leq t \leq n$,  every ordered subset of $t$ elements of $\{1,2,\dots, n\}$ determines a tabloid of
shape  $[n-t,1^t]$, where the $t$ elements are placed in order from top to bottom in the boxes corresponding to the $t$ singleton 
boxes, and the remaining
$n-t$ elements are placed in the first row (their order in the first row is unimportant, since this is a tabloid, not a tableau).  Thus, the ordered subsets of $t$ elements are in one-to-one correspondence with the  tabloids of shape $[n-t,1^t]$.  Such
tabloids form a basis for  the  permutation module $\M^{[n-t,1^t]}$, which is the $\S_n$-module obtained by inducing the trivial module
for the subgroup $\S_{n-t} \times \underbrace{\S_1 \times \cdots \times \S_1}_{t \ \text{terms}}$ to $\S_n$
(see for example \cite[Sec.~2.1]{Sa}).  

Now a set partition $\pi$  of $\{1,2, \ldots, k\}$ into $t$ blocks for $1 \le t \le n$ determines a collection of simple tensors
$\vf_{j_1} \ot \vf_{j_2} \ot \cdots \ot \vf_{j_k}$ satisfying \eqref{eq:PermutationCondition}, and the distinct  simple tensors
form a basis for $\M^\pi$.  Each such simple tensor determines an ordered set of $t$ elements, and $\M^\pi \cong
\M^{[n-t,1^t]}$. Since $\M_n = \M^{[n-1,1]}$, use of the term ``permutation module'' for both is consistent.   
For example, the simple tensor $\vf: =\v_3 \ot \v_1 \ot \v_3 \ot \v_4 \ot \v_4 \ot \v_1 \ot \v_1 \ot \v_3 \ot \v_5 \in \M_n^{\ot 9}$ corresponds to the set partition $\pi = \{1,3,8 \mid 2, 6, 7 \mid 4, 5 \mid 9\}$ having $t = 4$ blocks, and the linear span of such simple tensors gives $\M^\pi$, which is isomorphic to $\M^{[n-4,1^4]}$.   This correspondence can be made explicit by
having $\vf$ correspond to the tabloid with $2,6,7,8,9$ in its first row, then with $3,1,4,5$ in the boxes corresponding to
$1^4$ in succession from top to bottom.  Observe that the sizes of the blocks are immaterial to this isomorphism, what 
is relevant is the number $t$ of blocks.
  
The number of set partitions $\pi$ of the tensor positions $\{1, 2,\ldots, k\}$ into $1 \le t \le n$ parts is the Stirling number
$\stirlinginline{k}{t}$ of the second kind,  so by partitioning the simple tensors this way, we obtain the following decomposition
of $\M_n^{\ot k}$ into permutation modules for $\S_n$:
\begin{equation}
\M_n^{\ot k} = \bigoplus_{t=1}^n \stirling{k}{t} \M^{[n-t,1^t]}.
\end{equation} 
Note that $\stirlinginline{k}{t} = 0$ whenever $t > k$.

Young's rule (see for example, \cite[Thm.~2.11.2]{Sa}) gives the decomposition of the permutation module $\M^\gamma$  for $\gamma = [\gamma_1, \gamma_2, \ldots,\gamma_n] \vdash n$ into irreducible $\S_n$-modules $\S_n^\lambda$. It states that the multiplicity of $\S_n^\lambda$ in $\M^\gamma$  equals the Kostka number $\Ks_{\lam,\gamma}$,  which counts the number of \emph{semistandard tableaux} $T$ of shape $\lam$
and type $\gamma$,  where $T$ is a filling of the boxes of the Young diagram of $\lam$ with numbers from $\{1,2,\dots, n\}$ such that $j$ occurs $\gamma_j$ times, and the entries of $T$ weakly increase across the rows from left to right and strictly increase down the columns.   It follows that 
 \begin{equation*}
\M_n^{\ot k} = \bigoplus_{t=1}^n \stirling{k}{t} \M^{[n-t,1^t]} \cong
\bigoplus_{t=1}^n \stirling{k}{t} \left(\sum_{\lambda \vdash n} \Ks_{\lam, [n-t,1^t]} \S_n^\lambda\right)
\cong  \bigoplus_{\lambda \vdash n} \left(\sum _{t=1}^n \stirling{k}{t}  \Ks_{\lam, [n-t,1^t]}\right) \S_n^\lambda.
\end{equation*}

In this case, the Kostka number $\Ks_{\lam, [n-t, 1^t]}$ counts the number of semistandard tableaux of shape $\lambda$ filled with the entries 
$\{ 0^{n-t}, 1, 2, \ldots, t\}$.  A semistandard
tableau of shape $\lam$ whose entries are  $n-t$ zeros and the numbers $1, 2, \dots, t$  must have the $n-t$ zeros in the first row
and have a standard filling of the skew shape $\lam/[n-t]$ with the numbers $1,2,\dots,t$.  Consequently,  the number of such fillings is given by $f^{\lam/[n-t]}$  (see for example, \cite[Cor. 7.16.3]{S1}), and  $\Ks_{\lam, [n-t, 1^t]} = 0$
unless  $\lam_1 \geq n-t$, i.e.  unless $t \geq n-\lam_1 = |\lam^\#|$,  where $\lam^\#$ is the partition
obtained by removing one copy of the largest part of $\lam$.  

The next result combines  \cite[Thm.~5.5]{BHH} with \eqref{eq:dim1}-\eqref{eq:dim3}.   In the statement,  $\M_n^{\ot 0}$ should be  
interpreted as being $\FF \cong \S_n^{[n]}$ (as an $\S_n$-module).       The discussion above provides a proof of
the final equality in part (a) of the following theorem.  

 \begin{thm} \label{T:cent}   
 For   $k\in \ZZ_{\ge 0}$ and $n \in \ZZ_{\ge 1}$,  suppose $\M_n^{\ot k} = \bigoplus_{\lambda \in \Lambda_{k,\S_n}}
 \m_{k,n}^\lambda \S_n^\lambda$.     Assume $\Z_{k,n}^\lam$ is the irreducible
 module indexed by $\lam$ for $\Z_{k,n} =\End_{\S_n}(\M_n^{\ot k})$, as in (\ref{SW1})-(\ref{SW3}).  
 \begin{itemize}
\item[{\rm (a)}]  If $\lam \in \Lambda_{k,\S_n}$, then  
\begin{equation*} \frac{1}{n\,!} \sum_{\sigma \in \S_n} 
\mathsf{F}(\sigma)^{k}\chi_{\lam}{(\sigma)} \,=\, \dim(\Z_{k,n}^\lambda) \,=\, \m_{k,n}^\lambda \,=\, \sum_{t=|\lam^\#|}^n \stirling{k}{t} \, f^{\lam/[n-t]},
 \end{equation*}
where $\mathsf{F}(\sigma)$ is the number of fixed points of $\sigma$, and  $f^{\lam/[n-t]}$ is the number of standard tableaux of skew shape $\lambda/[n-t]$. 
 
\smallskip
\item[{\rm (b)}] $\displaystyle{\dim(\Z_{k,n}) = \dim\big(\Z_{2k,n}^{[n]}\big) =   \frac{1}{n\,!}  \sum_{\sigma \in \S_n}\mathsf{F}(\sigma)^{2k} =
\sum_{t =0}^{n} \stirling{2k}{t}}   = \mathsf{B}(2k,n).$ 

\smallskip

Furthermore,  $\dim(\Z_{k,n}) = \mathsf{B}(2k,n) =   \mathsf{B}(2k)$  if  $n \geq 2k.$

\medskip
\item[{\rm (c)}]   If $\mu \in \Lambda_{k+\half, \S_{n-1}}$, then
\begin{align*} \dim\big(\Z_{k+\half,n}^\mu\big) 
&= \sum_{t=|\mu^\#|}^{n-1}\stirling{k+1}{t+1} \, f^{\mu/[n-1-t]} =  \frac{1}{(n-1)!}  \sum_{\tau \in \S_{n-1}}\big(\mathsf{F}(\tau) + 1\big)^k \chi_\mu (\tau).
 \end{align*}
 
 \item[{\rm (d)}]   $\displaystyle{\dim(\Z_{k+\half,n}) \ =  \ \dim\big(\Z_{2k+1,n}^{[n-1]}\big)= \sum_{t =1}^{n} \stirling{2k+1}{t} \ = \mathsf{B}(2k+1,n). }$
 
 \smallskip
 
 Furthermore, $\dim(\Z_{k+\half,n}) \ =\mathsf{B}(2k+1,n) = \mathsf{B}(2k+1)$ if $n \geq 2k+1$.

 \end{itemize}
\end{thm}

 \begin{rem}  When $n > k$, the top limit in the summation in part (a) can be taken to be $k$ 
as the Stirling numbers $\stirlinginline{k}{t}$
are 0 for $t > k$.   When $n \leq k$, the
term $[n-t,1^t]$ for $t = n$ is assumed to be the partition $[1^n]$.  
 In that particular case,  $\mathsf{K}_{\lam, [1^n]} = f^\lam$, the
number of standard tableaux of shape $\lam$, as each entry in the tableau appears exactly once.   When $t = n-1$, the 
Kostka number is the same, $\mathsf{K}_{\lam, [1^n]} = f^\lam$. 
The only time that the term $t=0$ contributes is when $k = 0$.  In that case, $\stirlinginline{0}{0} = 1$, and the Kostka number 
$\mathsf{K}_{\lam, [n]} = 0$ if $\lam \neq [n]$ and $\mathsf{K}_{[n],[n]} = 1$.     Thus,  $\dim \Z_0^\lam(n) = \delta_{\lam,[n]}$,
as expected, since $\M_n^{\ot 0} = \S_n^{[n]}$.  
\end{rem} 

 In items (c) and (d) of  Theorem \ref{T:cent}, we have used the fact that as an $\S_{n-1}$-module, \  $\M_n = \M_{n-1} \oplus \FF \vf_n$, so the character value
$\chi_{\M_n}(\tau)$ of $\tau \in \S_{n-1}$ is the number of fixed points of $\tau$ plus 1.  Then since $\Z_{k+\half,n} = \End_{\S_{n-1}}(\M_n^{\ot k})$
and $\M_n$ is self-dual as an $\S_{n-1}$-module,  the  first line of part (d) holds.   The expressions in (a) and (c) can be related by the next result.

\begin{prop}\label{P:fixed} For all $k \in \ZZ_{\ge 0}$,  
\begin{align}\begin{split}\label{eq:fixedequals}
\frac{1}{n\,!}  \sum_{\sigma \in \S_n}
\mathsf{F}(\sigma)^{k+1} \ &= \m_{k+1,n}^{[n]} = \dim\big(\Z_{k+1,n}^{[n]}\big) \\
& =  \dim\big(\Z_{k+\half,n}^{[n-1]}\big) = \m_{k+\half,n-1}^{[n-1]}  =  \frac{1}{(n-1)!} \sum_{\tau \in \S_{n-1}}\big(\mathsf{F}(\tau) + 1\big)^{k},\end{split}
\end{align}
where $\m_{k+1,n}^{[n]}$ is the   multiplicity of $\S_n^{[n]}$ in $\M_n^{\ot k+1}$,
and $\m_{k+\half,n-1}^{[n-1]}$ is the multiplicity
of $\S_{n-1}^{[n-1]}$ in $\M_n^{\ot k}$, viewed as an $\S_{n-1}$-module. 
\end{prop}

\begin{proof}  The equalities in the first line of \eqref{eq:fixedequals} are a consequence of taking $\lambda = [n]$ in part (a) of Theorem \ref{T:cent},
and the equalities in the second line come from setting $\mu$  equal to $[n-1]$ in part (c).  
Now observe that the only way $\S_n^{[n]}$ can be obtained by inducing the $\S_{n-1}$-module $\Res_{\S_{n-1}}^{\S_n}\left(\M_n^{\ot k}\right)$  to $\S_n$
is from the summands $\S_{n-1}^{[n-1]}$, so the multiplicity of $\S_n^{[n]}$ in $\M_n^{\ot {k+1}} =\Ind_{\S_{n-1}}^{\S_n}\big(\Res_{\S_{n-1}}^{\S_n}\left(\M_n^{\ot k}\big)
\right)$, which equals $\dim\big(\Z_{k+1,n}^{[n]}\big)$,
is the same as the multiplicity of $\S_{n-1}^{[n-1]}$ in $\Res_{\S_{n-1}}^{\S_n}\left(\M_n^{\ot k}\right)$,
which equals  $\dim\big(\Z_{k+\half,n}^{[n-1]}\big)$.   This can also be seen in the Bratteli 
diagram $\mathcal{B}(\S_n,\S_{n-1})$, as there is a unique path from level $k + \half$ to the node $[n]$ at level $k$, and that path connects $[n-1]$ to $[n]$
(compare Fig. \ref{fig:Sbratteli}).    \end{proof}  

 We know from Theorem \ref{T:cent}\,(b) that the next result holds for even values of $\ell$.
Proposition \ref{P:fixed} combined with part (d) of the theorem allows us to conclude that it holds for odd values of $\ell$ as well.  Consequently,
we have the following:

\begin{cor}\label{C:fixedBell}  For all $\ell \in \ZZ_{\ge 0}$ and $n \in \ZZ_{\ge 1}$,
\begin{equation}\label{eq:Bellfixed}  \mathsf{B}(\ell,n) = \sum_{t=0}^n  \stirling{\ell}{t}  =  \frac{1}{n!} \sum_{\sigma \in \S_n}  \mathsf{F}(\sigma)^\ell.
\end{equation} \end{cor}

\begin{rem} {\rm  The identity in \eqref{eq:Bellfixed} has connections with moments of random permutations.  If the random variable $X$ denotes the number of fixed
points of a uniformly distributed random permutation of a set of size $\ell \ge 1$ into no more than $n$ parts, then the $\ell$th moment of $X$ is
$$E(X^{\ell})=\sum _{t=1}^{n} \stirling{\ell}{t} = \B(\ell,n) =  \frac{1}{n!} \sum_{\sigma \in \S_n}  \mathsf{F}(\sigma)^\ell.$$
Compare this with \cite[(18)]{FaH}, where this formula is proved for $\ell \le n$.}\end{rem}

\begin{rem} {\rm The results of this section relate Stirling numbers of the second kind to the number of fixed points of permutations.
In the case of parts (a) and (c) of Theorem \ref{T:cent}, the expressions  involve Stirling numbers and also the number of standard tableaux of 
certain skew shapes.  
 It would be interesting to have combinatorial bijections which demonstrate the identities.  
A determinantal formula for the number $f^{\lambda/\nu}$ of standard tableaux of skew shape $\lambda/\nu$ was given by Aitken (see \cite[Cor. 7.16.3]{S1}). Hook-length formulas for  $f^{\lambda/\nu}$ are studied in \cite{MPP} and the references therein.} \end{rem}

\subsection{Set-partition tableaux} \label{subsec:setparttabs}

Part (a) of the Theorem \ref{T:cent} has inspired the following definition:

\begin{definition}\label{D:setpart} For $\lambda = [\lambda_1,\lam_2, \ldots,  \lambda_n]$ a partition of $n$, 
assume $\lam^\# =[\lam_2,\dots,\lam_n]$ and $t\in \ZZ$ is such that $|\lam^\#| \le t \le n$.
A \emph{set-partition tableau $\mathsf{T}$ of shape $\lambda$  and content $\{0^{n-t},1, \ldots, k\}$}
is a filling of the boxes of $\lambda$ so that the following requirements are met:

\begin{enumerate}
\item[{\rm(i)}] the first $n-t$ boxes of the first row of $\lambda$ are filled with $0$;
\item[{\rm(ii)}]  the boxes of the skew shape $\lambda/[n-t]$ are filled with the numbers in $[1,k]$ such that 
 the entries in each box of $\lambda/[n-t]$ form a block of a set partition $\pi(\mathsf{T})$  of $[1,k]$ having $t$ blocks; 
\item[{\rm(iii)}] the boxes of $\mathsf{T}$ in the skew shape  $\lambda/[n-t]$  strictly increase across the rows and down the columns of  $\lambda/[n-t]$, where if $b_1$ and $b_2$ are two boxes of $\lambda/[n-t]$,  then  $b_1 < b_2$ holds if 
the maximum entries in these boxes satisfy $\mathsf{max}(b_1) < \mathsf{max}(b_2)$.
\end{enumerate}
\end{definition} 

\begin{examp}{\rm  Below is a  set-partition tableau $\mathsf{T}$ of shape $\lambda = [5,4,2,1] \vdash 12$ and content $\{0^4,1,2,$ $\ldots, 20\}$ with corresponding set partition $\pi(\mathsf{T})=\{1,\underline{6} \mid
4,7,9,\underline{10} \mid 
2,11,\underline{12} \mid 
8,\underline{14} \mid
15,\underline{16} \mid
5,13,\underline{18} \mid 
3,17,\underline{19} \mid 
\underline{20}  
\} \in \Pi_{20}$ consisting of $t = 8$ blocks. 
The blocks  of $\pi(\mathsf{T})$ are listed in increasing order according to their largest elements, which are  the underlined numbers.
$$
\begin{tikzpicture}[xscale=1.5,yscale=.6]
\fill[black!10!white] (0,3) rectangle (4,4);
\draw (0,0) -- (0,4) -- (5,4) -- (5,3) -- (4,3) -- (4,2) -- (2,2) -- (2,1) -- (1,1) -- (1,0) -- (0,0);
\draw (0,3) -- (4,3) -- (4,4);
\draw (0,2) -- (2,2) -- (2,4);
\draw (3,2) -- (3,4);
\draw (0,1) -- (1,1) -- (1,4);
\path (.5,3.5) node {$0$}; \path (1.5,3.5) node {$0$}; \path (2.5,3.5) node {$0$}; \path (3.5,3.5) node {$0$};
\path (.5,2.5) node {$1,\underline{6}$};
\path (1.5,2.5) node {$4,7,9,\underline{10}$};
\path (.5,1.5) node {$2, 11, \underline{12}$};
\path (1.5,1.5) node {$8,\underline{14}$};
\path (0.5,0.5) node {$15, \underline{16}$};
\path (2.5,2.5) node {$5,13,\underline{18}$};
\path (4.5,3.5) node {$3,17,\underline{19}$};
\path (3.5,2.5) node {$\underline{20}$};
\end{tikzpicture}
$$     
}
\end{examp}

\begin{rem}{\rm If $|\lambda^\#| = k$, and $\mathsf{T}$ is a set-partition tableau  of shape $\lambda$ and content $\{0^{n-t},1, \ldots,k\}$, then the $k$ boxes of $\mathsf{T}$ corresponding to $\lambda^\#$  form a set partition of $\{1, \ldots, k\}$, so they must be a standard tableau of shape $\lambda^\#$. The first row of $\mathsf{T}$ must then consist of a single row of zeros of length of $t=\lambda_1 = n - |\lambda^\#| = n-k$. For example,
if $n = 9$ and $k = 5$, then displayed below is a set partition tableau of shape $\lambda$ with $\lambda^\# = [3,2]$,
$$
\begin{array}{c}\begin{tikzpicture}[xscale=.5,yscale=.5]
\draw (0,0) rectangle (1,1); \path (.5,.5) node {$0$};
\draw (1,0) rectangle (2,1); \path (1.5,.5) node {$0$};
\draw (2,0) rectangle (3,1); \path (2.5,.5) node {$0$};
\draw (3,0) rectangle (4,1); \path (3.5,.5) node {$0$};
\draw (0,-1) rectangle (1,0);\path (.5,-.5) node {$\underline{1}$};
\draw (1,-1) rectangle (2,0);\path (1.5,-.5) node {$\underline{3}$};
\draw (2,-1) rectangle (3,0);\path (2.5,-.5) node {$\underline{5}$};
\draw (0,-2) rectangle (1,-1);\path (.5,-1.5) node {$\underline{2}$};
\draw (1,-2) rectangle (2,-1);\path (1.5,-1.5) node {$\underline{4}$};
\end{tikzpicture}\end{array}.
$$ 
}\end{rem}
 
 \noindent
The following statement is an immediate consequence of Definition \ref{D:setpart}  and  Theorem \ref{T:cent}\,(a): \\
 \medskip 
 
 \emph{
 If  $k,n \in \ZZ_{\ge 0}$ and $n \geq 1$,  then for all 
$\lam\in \Lambda_{k,\S_n}$,
\begin{equation}
\dim(\Z_{k,n}^\lambda) =  \m_{k,n}^\lambda
 =  
\left | \left \{
\begin{array}{l} 
\text{set-partition tableaux of shape $\lambda$ and content}\\
\text{$[0^{n-t},1,\ldots,k]$ for some $t$ such that $|\lam^\#| \le t \le n$}
\end{array}
\right\}
\right |. 
 \end{equation}}

\subsection{Bijections} \label{subsec:biject}

Combining the results of the previous two sections with what we know from Schur-Weyl duality \eqref{SW1}--\eqref{SW4}, we have the following:
\begin{thm}  For $\Z_{k,n} = \End_{\S_n}(\M_n^{\ot k})$ and for $\lambda \vdash n$,   the  following are equal:
\begin{itemize}
\item[{\rm (i)}] the multiplicity $\m_{k,n}^\lambda$ of $\S_n^\lambda$ in $\M_n^{\ot k}$,   
\item[{\rm(ii)}] the dimension of the irreducible $\Z_{k,n}$-module $\Z_{k,n}^\lambda$ indexed by $\lambda$,
\item[{\rm (iii)}]  the number of paths in the Bratteli diagram $\mathcal{B}(\S_n, \S_{n-1})$ from $[n]$ at level 0 to $\lambda$ at level $k$,
\item[{\rm (iv)}]  the number of vacillating tableaux  
of shape $\lambda$ and length $k$, 
\item[{\rm (v)}] the number of pairs $(\pi,\mathsf{S})$, where $\pi$ is a set partition of $\{1,2,\dots,k\}$ with $t$ blocks, and $\mathsf{S}$ is a standard tableau of shape $\lambda/[n-t]$  for
some $t$ such that  $|\lam^\#| \le t \le n$,
\item[{\rm (vi)}]  the number of set-partition tableaux of shape $\lambda$ and content $[0^{n-t},1,\dots,k]$, 
for some $t$ such that $ |\lambda^\#|\le t \le n$.
\end{itemize}\end{thm}

The fact that (iii) and (iv) have the same cardinality is immediate from the definition of vacillating tableaux. The fact that (v) and (vi) have equal cardinalities can be seen by taking a set-partition tableau as in (v), and replacing the entries in the
boxes having nonzero entries with the numbers $1,2,\dots,t$ according to their maximal entries from smallest to largest.   The reverse process
fills the boxes of the standard tableau $\mathsf{S}$ with the entries in the blocks of $\pi$ according to their maximal elements
with 1 for the block with smallest maximal entry and proceeding  to $t$ for the block with the largest entry.

We now describe an algorithm that gives a bijection between set-partition tableaux of shape $\lambda$ 
and content $\{0^{n-t},1,\dots,k\}$ for some $|\lam^\#| \le t \le n$ and vacillating tableaux of length $k$. 
The algorithm assumes familiarity with Schensted row insertion (see \cite[Sec. 7.11]{S2}). 
We use $\mathsf{T} \leftarrow b$ to mean row insertion of the box $b$ (along with its entries) into the 
set-partition tableau $\mathsf{T}$ governed by their maximum elements as in Definition \ref{D:setpart}.  
(That is, do usual Schensted insertion on the maximal elements of each box, but then also include all the entries of the box.)

\bigskip
\noindent
{\bf A.\  Set-partition Tableaux} $\Rightarrow$ {\bf Vacillating Tableaux}

\medskip\noindent
Given a set-partition tableau $\mathsf{T}$ of shape $\lambda \vdash n$ and content $[0^{n-t},1,\ldots,k]$,
 where $|\lam^\#| \le t \le n$, 
the following algorithm recursively produces a vacillating $k$-tableau
 $([n]=\lambda^{(0)}, \lambda^{(\frac{1}{2})}, \lambda^{(1)}, \ldots,$ $\lambda^{(k)}=\lambda)$ of shape $\lambda$.  Examples can be found in Figures \ref{fig:BijectionExample} and \ref{fig:BijectionDimensionExample}.

\begin{itemize}
\item[(1)] Let $\lambda^{(k)} = \lambda$,  and set $\mathsf{T}^{(k)} = \mathsf{T}$.

\item[(2)] For $j = k, k-1, \ldots, 1$ (in descending order), do the following:

\begin{itemize}
\item[(a)] Let $\mathsf{T}^{(j-\frac{1}{2})}$ be the tableau obtained from $\mathsf{T}^{(j)}$ by removing the box $b$ that contains $j$. At this stage, $j$ will be the largest entry of $\mathsf{T}$ so this box will be removable. Let $\lambda^{(j-\frac{1}{2})}$ be the shape of $\mathsf{T}^{(j-\frac{1}{2})}$.
\item[(b)] Delete the entry $j$ from $b$. If $b$ is then empty,  add 0 to it.

\item[(c)] Let $\mathsf{T}^{(j-1)} = \mathsf{T}^{(j-\frac{1}{2})} \leftarrow b$ be the Schensted row insertion of $b$ into $\mathsf{T}^{(j-\frac{1}{2})}$, and let $\lambda^{(j-1)}$ be the shape of $\mathsf{T}^{(j-1)}$.
\end{itemize}

\end{itemize}

We delete the largest number $j$ at the $j$th step, so by our construction $\mathsf{T}^{(j)}$ is a tableau containing a set partition of $\{1, 2, \ldots, j\}$ for each $k \ge j \ge 1$. Furthermore, Schensted insertion keeps the rows weakly increasing and columns strictly increasing at each step.  At the conclusion,  $\mathsf{T}^{(0)}$ is a semistandard tableau that contains only  zeros, and as such it must have shape $\lambda^{(0)} = [n]$. The sequence of underlying shapes $(\lambda^{(0)}, \lambda^{(\frac{1}{2})}, \lambda^{(1)}, \ldots, \lambda^{(k)})$, listed in reverse order from the way they are constructed, is obtained by removing and adding a box at each step, so it is a vacillating tableau of shape $\lambda$ and length $k$.

\bigskip
\noindent
{\bf B. \ Vacillating Tableaux} $\Rightarrow$ {\bf  Set-partition Tableaux}

\medskip\noindent
Algorithm A is easily seen to be invertible. Given a vacillating tableau $(\lambda^{(0)}, \lambda^{(\frac{1}{2})}, \lambda^{(1)}, \ldots,$ $\lambda^{(k)})$ of shape $\lambda$ and length $k$, the following algorithm produces a set-partition tableau $\mathsf{T}$ of shape $\lambda$. This process recursively fills the boxes of the shapes $\lambda^{(j)}$ to produce the same fillings as the algorithm above.

\begin{itemize}
\item[(1)] Let $\mathsf{T}^{(0)}$ be the semistandard tableau of shape $\lambda^{(0)}=[n]$ with each of its boxes filled with 0. 

\item[(2)] For $j = 0, 1, \ldots, k$, do the following:

\begin{itemize}
\item[(a${}'$)] Let $\mathsf{T}^{(j+\frac{1}{2})}$ be the tableau given by un-inserting the box of $\mathsf{T}^{(j)}$ at position $\lambda^{(j)}/\lambda^{(j+\frac{1}{2})}$, and let $b$ be the box that is un-inserted in this process. 
That is,  $ \mathsf{T}^{(j+\frac{1}{2})}$ and $b$ are the unique tableau of shape $\lambda^{(j+\frac{1}{2})}$ and box, respectively, such that $\mathsf{T}^{(j)} = \mathsf{T}^{(j+\frac{1}{2})}  \leftarrow b$.

\item[(b${}'$)] Add $j$ to box $b$. If $b$ contains 0, delete $0$ from it.

\item[(c${}'$)] Add the content of the box $b$ to the box in position $\lambda^{(j+1)}/\lambda^{(j+\frac{1}{2})}$, and fill the rest of $\mathsf{T}^{(j+1)}$ with the same entries as in the boxes of  $\mathsf{T}^{(j+\frac{1}{2})}$.

\end{itemize}
\end{itemize}

Algorithms A and B invert one another step-by-step,  since (a) and (c${}'$), (b) and (b${}'$), and (c) and (a${}'$) are easily seen to be the inverses of one another. The fact that
steps (a) and (c${}'$) are inverses comes from the fact that Schensted insertion is invertible.

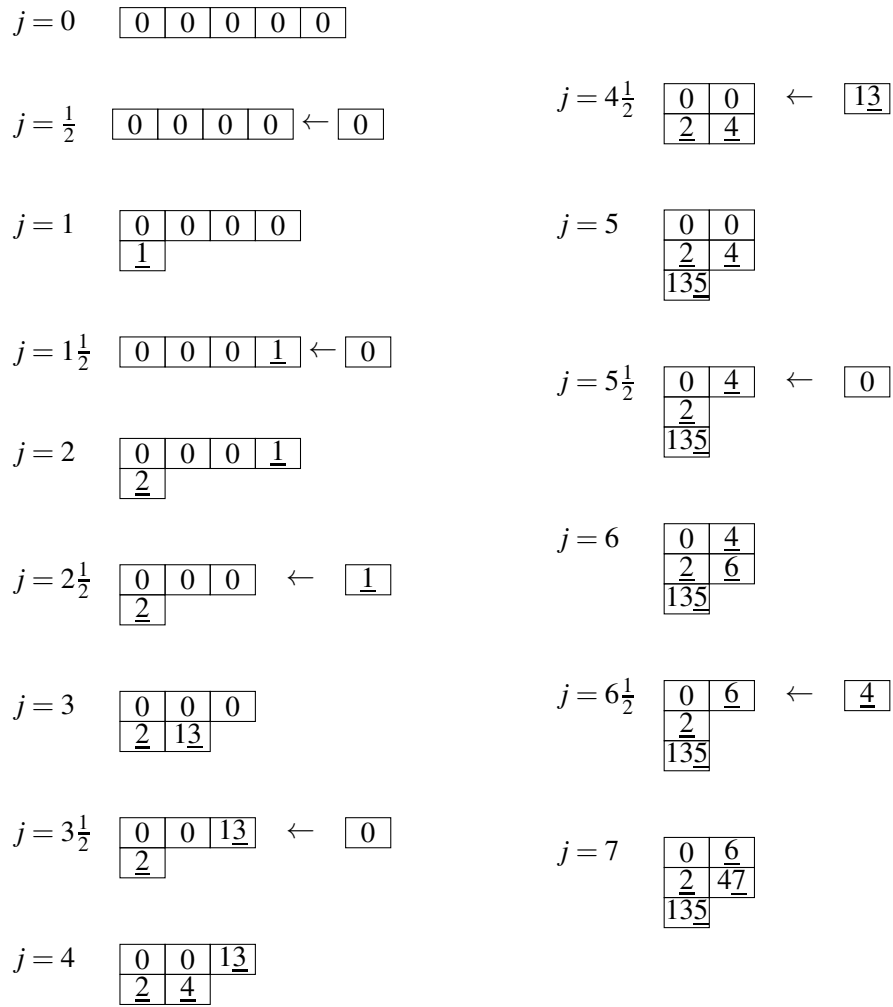
\begin{figure}
$$
\begin{array}{lcl}
\begin{array}{l}
\begin{tikzpicture}[xscale=.6,yscale=.4]
\draw (0,0) rectangle (1,1); \path (.5,.5) node {$0$};
\draw (1,0) rectangle (2,1); \path (1.5,.5) node {$0$};
\draw (2,0) rectangle (3,1);\path (2.5,.5) node {$0$};
\draw (3,0) rectangle (4,1);\path (3.5,.5) node {$0$};
\draw (4,0) rectangle (5,1);\path (4.5,.5) node {$0$};
\path (-1.5,.5) node {${j=0\phantom{\half}}$};
\end{tikzpicture} 
\\ \\
\begin{tikzpicture}[xscale=.6,yscale=.4]
\draw (0,0) rectangle (1,1); \path (.5,.5) node {$0$};
\draw (1,0) rectangle (2,1); \path (1.5,.5) node {$0$};
\draw (2,0) rectangle (3,1);\path (2.5,.5) node {$0$};
\draw (3,0) rectangle (4,1);\path (3.5,.5) node {$0$};
\path (4.5,.5) node {$\leftarrow$};
\draw (5,0) rectangle (6,1);\path (5.5,.5) node {$0$};
\path (-1.5,.5) node {${j=\half}$};
\end{tikzpicture} 
\\ \\
\begin{tikzpicture}[xscale=.6,yscale=.4]
\draw (0,0) rectangle (1,1); \path (.5,.5) node {$0$};
\draw (1,0) rectangle (2,1); \path (1.5,.5) node {$0$};
\draw (2,0) rectangle (3,1);\path (2.5,.5) node {$0$};
\draw (3,0) rectangle (4,1);\path (3.5,.5) node {$0$};
\draw (0,-1) rectangle (1,0);\path (.5,-.5) node {$\underline{1}$};
\path (-1.5,.5) node {${j=1\phantom{\half}}$};
\end{tikzpicture} 
\\ \\
\begin{tikzpicture}[xscale=.6,yscale=.4]
\draw (0,0) rectangle (1,1); \path (.5,.5) node {$0$};
\draw (1,0) rectangle (2,1); \path (1.5,.5) node {$0$};
\draw (2,0) rectangle (3,1);\path (2.5,.5) node {$0$};
\draw (3,0) rectangle (4,1);\path (3.5,.5) node {$\underline{1}$};
\path (4.5,.5) node {$\leftarrow$};
\draw (5,0) rectangle (6,1);\path (5.5,.5) node {$0$};
\path (-1.5,.5) node {${j=1\half}$};
\end{tikzpicture} 
\\ \\
\begin{tikzpicture}[xscale=.6,yscale=.4]
\draw (0,0) rectangle (1,1); \path (.5,.5) node {$0$};
\draw (1,0) rectangle (2,1); \path (1.5,.5) node {$0$};
\draw (2,0) rectangle (3,1);\path (2.5,.5) node {$0$};
\draw (3,0) rectangle (4,1);\path (3.5,.5) node {$\underline{1}$};
\draw (0,-1) rectangle (1,0);\path (.5,-.5) node {$\underline{2}$};
\path (-1.5,.5) node {${j=2\phantom{\half}}$};
\end{tikzpicture}
\\ \\
\begin{tikzpicture}[xscale=.6,yscale=.4]
\draw (0,0) rectangle (1,1); \path (.5,.5) node {$0$};
\draw (1,0) rectangle (2,1); \path (1.5,.5) node {$0$};
\draw (2,0) rectangle (3,1);\path (2.5,.5) node {$0$};
\draw (0,-1) rectangle (1,0);\path (0.5,-.5) node {$\underline{2}$};
\path (4,.5) node {$\leftarrow$};
\draw (5,0) rectangle (6,1);\path (5.5,.5) node {$\underline{1}$};
\path (-1.5,.5) node {${j=2\half}$};
\end{tikzpicture} 
\\ \\
\begin{tikzpicture}[xscale=.6,yscale=.4]
\draw (0,0) rectangle (1,1); \path (.5,.5) node {$0$};
\draw (1,0) rectangle (2,1); \path (1.5,.5) node {$0$};
\draw (2,0) rectangle (3,1);\path (2.5,.5) node {$0$};
\draw (0,-1) rectangle (1,0);\path (.5,-.5) node {$\underline{2}$};
\draw (1,-1) rectangle (2,0);\path (1.5,-.5) node {$1 \underline{3}$};
\path (-1.5,.5) node {${j=3\phantom{\half}}$};
\end{tikzpicture} 
\\ \\
\begin{tikzpicture}[xscale=.6,yscale=.4]
\draw (0,0) rectangle (1,1); \path (.5,.5) node {$0$};
\draw (1,0) rectangle (2,1); \path (1.5,.5) node {$0$};
\draw (2,0) rectangle (3,1);\path (2.5,.5) node {$1 \underline{3}$};
\draw (0,-1) rectangle (1,0);\path (.5,-.5) node {$\underline{2}$};
\path (4,.5) node {$\leftarrow$};
\draw (5,0) rectangle (6,1);\path (5.5,.5) node {$0$};
\path (-1.5,.5) node {${j=3\half}$};
\end{tikzpicture}
\\ \\
\begin{tikzpicture}[xscale=.6,yscale=.4]
\draw (0,0) rectangle (1,1); \path (.5,.5) node {$0$};
\draw (1,0) rectangle (2,1); \path (1.5,.5) node {$0$};
\draw (2,0) rectangle (3,1);\path (2.5,.5) node {$1 \underline{3}$};
\draw (0,-1) rectangle (1,0);\path (.5,-.5) node {$\underline{2}$};
\draw (1,-1) rectangle (2,0);\path (1.5,-.5) node {$\underline{4}$};
\path (-1.5,.5) node {${j=4\phantom{\half}}$};
\end{tikzpicture}
\end{array}
& \hskip.4in &
\begin{array}{l}
\begin{tikzpicture}[xscale=.6,yscale=.4]
\draw (0,0) rectangle (1,1); \path (.5,.5) node {$0$};
\draw (1,0) rectangle (2,1); \path (1.5,.5) node {$0$};
\draw (0,-1) rectangle (1,0);\path (.5,-.5) node {$\underline{2}$};
\draw (1,-1) rectangle (2,0);\path (1.5,-.5) node {$\underline{4}$};
\path (3,.5) node {$\leftarrow$};
\draw (4,0) rectangle (5,1);\path (4.5,.5) node {$1 \underline{3}$};
\path (-1.5,.5) node {${j=4\half}$};
\end{tikzpicture}  
\\ \\
\begin{tikzpicture}[xscale=.6,yscale=.4]
\draw (0,0) rectangle (1,1); \path (.5,.5) node {$0$};
\draw (1,0) rectangle (2,1); \path (1.5,.5) node {$0$};
\draw (0,-1) rectangle (1,0);\path (.5,-.5) node {$\underline{2}$};
\draw (1,-1) rectangle (2,0);\path (1.5,-.5) node {$\underline{4}$};
\draw (0,-2) rectangle (1,-1);\path (.5,-1.5) node {$1 3 \underline{5}$};
\path (-1.5,.5) node {${j=5\phantom{\half}}$};
\end{tikzpicture} 
\\ \\
\begin{tikzpicture}[xscale=.6,yscale=.4]
\draw (0,0) rectangle (1,1); \path (.5,.5) node {$0$};
\draw (1,0) rectangle (2,1); \path (1.5,.5) node {$\underline{4}$};
\draw (0,-1) rectangle (1,0);\path (.5,-.5) node {$\underline{2}$};
\draw (0,-2) rectangle (1,-1);\path (.5,-1.5) node {$1 3 \underline{5}$};
\path (3,.5) node {$\leftarrow$};
\draw (4,0) rectangle (5,1);\path (4.5,.5) node {$0$};
\path (-1.5,.5) node {${j=5\half}$};
\end{tikzpicture} 
\\ \\
\begin{tikzpicture}[xscale=.6,yscale=.4]
\draw (0,0) rectangle (1,1); \path (.5,.5) node {$0$};
\draw (1,0) rectangle (2,1); \path (1.5,.5) node {$\underline{4}$};
\draw (0,-1) rectangle (1,0);\path (.5,-.5) node {$\underline{2}$};
\draw (1,-1) rectangle (2,0);\path (1.5,-.5) node {$\underline{6}$};
\draw (0,-2) rectangle (1,-1);\path (.5,-1.5) node {$1 3 \underline{5}$};
\path (-1.5,.5) node {${j=6\phantom{\half}}$};
\end{tikzpicture}
\\ \\
\begin{tikzpicture}[xscale=.6,yscale=.4]
\draw (0,0) rectangle (1,1); \path (.5,.5) node {$0$};
\draw (1,0) rectangle (2,1); \path (1.5,.5) node {$\underline{6}$};
\draw (0,-1) rectangle (1,0);\path (.5,-.5) node {$\underline{2}$};
\draw (0,-2) rectangle (1,-1);\path (.5,-1.5) node {$1 3 \underline{5}$};
\path (3,.5) node {$\leftarrow$};
\draw (4,0) rectangle (5,1);\path (4.5,.5) node {$\underline{4}$};
\path (-1.5,.5) node {${j=6\half}$};
\end{tikzpicture}
\\ \\
\begin{tikzpicture}[xscale=.6,yscale=.4]
\draw (0,0) rectangle (1,1); \path (.5,.5) node {$0$};
\draw (1,0) rectangle (2,1); \path (1.5,.5) node {$\underline{6}$};
\draw (0,-1) rectangle (1,0);\path (.5,-.5) node {$\underline{2}$};
\draw (1,-1) rectangle (2,0);\path (1.5,-.5) node {$4 \underline{7}$};
\draw (0,-2) rectangle (1,-1);\path (.5,-1.5) node {$1 3 \underline{5}$};
\path (-1.5,.5) node {${j=7\phantom{\half}}$};
\end{tikzpicture}
\end{array}
\end{array}
$$
\caption{Bijection between a vacillating 7-tableau of shape $[2,2,1]$ and a 7-set-partition tableau of shape $[2,2,1]$.
\label{fig:BijectionExample}}
\end{figure}

\begin{figure}
$$
\begin{array}{lcl}
\begin{array}{l}
\begin{tikzpicture}[xscale=.6,yscale=.4]
\draw (0,0) rectangle (1,1); \path (.5,.5) node {$0$};
\draw (1,0) rectangle (2,1); \path (1.5,.5) node {$0$};
\draw (2,0) rectangle (3,1);\path (2.5,.5) node {$0$};
\draw (3,0) rectangle (4,1);\path (3.5,.5) node {$0$};
\draw (4,0) rectangle (5,1);\path (4.5,.5) node {$0$};
\path (-1.5,.5) node {${j=0\phantom{\half}}$};
\end{tikzpicture} 
\\ \\
\begin{tikzpicture}[xscale=.6,yscale=.4]
\draw (0,0) rectangle (1,1); \path (.5,.5) node {$0$};
\draw (1,0) rectangle (2,1); \path (1.5,.5) node {$0$};
\draw (2,0) rectangle (3,1);\path (2.5,.5) node {$0$};
\draw (3,0) rectangle (4,1);\path (3.5,.5) node {$0$};
\path (4.5,.5) node {$\leftarrow$};
\draw (5,0) rectangle (6,1);\path (5.5,.5) node {$0$};
\path (-1.5,.5) node {${j=\half\phantom{0}}$};
\end{tikzpicture} 
\\ \\
\begin{tikzpicture}[xscale=.6,yscale=.4]
\draw (0,0) rectangle (1,1); \path (.5,.5) node {$0$};
\draw (1,0) rectangle (2,1); \path (1.5,.5) node {$0$};
\draw (2,0) rectangle (3,1);\path (2.5,.5) node {$0$};
\draw (3,0) rectangle (4,1);\path (3.5,.5) node {$0$};
\draw (0,-1) rectangle (1,0);\path (.5,-.5) node {$\underline{1}$};
\path (-1.5,.5) node {${j=1\phantom{\half}}$};
\end{tikzpicture} 
\\ \\
\begin{tikzpicture}[xscale=.6,yscale=.4]
\draw (0,0) rectangle (1,1); \path (.5,.5) node {$0$};
\draw (1,0) rectangle (2,1); \path (1.5,.5) node {$0$};
\draw (2,0) rectangle (3,1);\path (2.5,.5) node {$0$};
\draw (3,0) rectangle (4,1);\path (3.5,.5) node {$\underline{1}$};
\path (4.5,.5) node {$\leftarrow$};
\draw (5,0) rectangle (6,1);\path (5.5,.5) node {$0$};
\path (-1.5,.5) node {${j=1\half}$};
\end{tikzpicture} 
\\ \\ 
\begin{tikzpicture}[xscale=.6,yscale=.4]
\draw (0,0) rectangle (1,1); \path (.5,.5) node {$0$};
\draw (1,0) rectangle (2,1); \path (1.5,.5) node {$0$};
\draw (2,0) rectangle (3,1);\path (2.5,.5) node {$0$};
\draw (3,0) rectangle (4,1);\path (3.5,.5) node {$\underline{1}$};
\draw (0,-1) rectangle (1,0);\path (.5,-.5) node {$\underline{2}$};
\path (-1.5,.5) node {${j=2\phantom{\half}}$};
\end{tikzpicture}
\\ \\
\begin{tikzpicture}[xscale=.6,yscale=.4]
\draw (0,0) rectangle (1,1); \path (.5,.5) node {$0$};
\draw (1,0) rectangle (2,1); \path (1.5,.5) node {$0$};
\draw (2,0) rectangle (3,1);\path (2.5,.5) node {$0$};
\draw (0,-1) rectangle (1,0);\path (0.5,-.5) node {$\underline{2}$};
\path (4.5,.5) node {$\leftarrow$};
\draw (5,0) rectangle (6,1);\path (5.5,.5) node {$\underline{1}$};
\path (-1.5,.5) node {${j=2\half}$};
\end{tikzpicture} 
\\ \\
\begin{tikzpicture}[xscale=.6,yscale=.4]
\draw (0,0) rectangle (1,1); \path (.5,.5) node {$0$};
\draw (1,0) rectangle (2,1); \path (1.5,.5) node {$0$};
\draw (2,0) rectangle (3,1);\path (2.5,.5) node {$0$};
\draw (0,-1) rectangle (1,0);\path (.5,-.5) node {$\underline{2}$};
\draw (1,-1) rectangle (2,0);\path (1.5,-.5) node {$1 \underline{3}$};
\path (-1.5,.5) node {${j=3\phantom{\half}}$};
\end{tikzpicture} 
\\ \\
\begin{tikzpicture}[xscale=.6,yscale=.4]
\draw (0,0) rectangle (1,1); \path (.5,.5) node {$0$};
\draw (1,0) rectangle (2,1); \path (1.5,.5) node {$0$};
\draw (2,0) rectangle (3,1);\path (2.5,.5) node {$1 \underline{3}$};
\draw (0,-1) rectangle (1,0);\path (.5,-.5) node {$\underline{2}$};
\path (4.5,.5) node {$\leftarrow$};
\draw (5,0) rectangle (6,1);\path (5.5,.5) node {$0$};
\path (-1.5,.5) node {${j=3\half}$};
\end{tikzpicture}
\\ \\
\begin{tikzpicture}[xscale=.6,yscale=.4]
\draw (0,0) rectangle (1,1); \path (.5,.5) node {$0$};
\draw (1,0) rectangle (2,1); \path (1.5,.5) node {$0$};
\draw (2,0) rectangle (3,1);\path (2.5,.5) node {$1 \underline{3}$};
\draw (0,-1) rectangle (1,0);\path (.5,-.5) node {$\underline{2}$};
\draw (0,-2) rectangle (1,-1);\path (.5,-1.5) node {$\underline{4}$};
\path (-1.5,.5) node {${j=4\phantom{\half}}$};
\end{tikzpicture}
\end{array}
& \hskip.2in &
\begin{array}{l}
\begin{tikzpicture}[xscale=.6,yscale=.4]
\draw (0,0) rectangle (1,1); \path (.5,.5) node {$0$};
\draw (1,0) rectangle (2,1); \path (1.5,.5) node {$0$};
\draw (0,-1) rectangle (1,0);\path (.5,-.5) node {$\underline{2}$};
\draw (0,-2) rectangle (1,-1);\path (.5,-1.5) node {$\underline{4}$};
\path (4.5,.5) node {$\leftarrow$};
\draw (5,0) rectangle (6,1);\path (5.5,.5) node {$1 \underline{3}$};
\path (-1.5,.5) node {${j=4\half}$};
\end{tikzpicture}  
\\ \\
\begin{tikzpicture}[xscale=.6,yscale=.4]
\draw (0,0) rectangle (1,1); \path (.5,.5) node {$0$};
\draw (0,-1) rectangle (1,0); \path (.5,-.5) node {$\underline{2}$};
\draw (0,-2) rectangle (1,-1); \path (.5,-1.5) node {$\underline{4}$};
\draw (1,0) rectangle (2,1); \path (1.5,.5) node {$0$};
\draw (2,0) rectangle (3,1);\path (2.5,.5) node {$1 3 \underline{5}$};
\path (-1.5,.5) node {${j=5\phantom{\half}}$};
\end{tikzpicture} 
\\ \\
\begin{tikzpicture}[xscale=.6,yscale=.4]
\draw (0,0) rectangle (1,1); \path (.5,.5) node {$0$};
\draw (0,-1) rectangle (1,0); \path (.5,-.5) node {$\underline{4}$};
\draw (1,0) rectangle (2,1); \path (1.5,.5) node {$\underline{2}$};
\draw (2,0) rectangle (3,1);\path (2.5,.5) node {$1 3 \underline{5}$};
\path (4.5,.5) node {$\leftarrow$};
\draw (5,0) rectangle (6,1);\path (5.5,.5) node {$0$};
\path (-1.5,.5) node {${j=5\half}$};
\end{tikzpicture} 
\\ \\
\begin{tikzpicture}[xscale=.6,yscale=.4]
\draw (0,0) rectangle (1,1); \path (.5,.5) node {$0$};
\draw (0,-1) rectangle (1,0); \path (.5,-.5) node {$\underline{4}$};
\draw (1,0) rectangle (2,1); \path (1.5,.5) node {$\underline{2}$};
\draw (2,0) rectangle (3,1);\path (2.5,.5) node {$1 3 \underline{5}$};
\draw (3,0) rectangle (4,1);\path (3.5,.5) node {$\underline{6}$};
\path (-1.5,.5) node {${j=6\phantom{\half}}$};
\end{tikzpicture}
\\ \\
\begin{tikzpicture}[xscale=.6,yscale=.4]
\draw (0,0) rectangle (1,1); \path (.5,.5) node {$0$};
\draw (0,-1) rectangle (1,0); \path (.5,-.5) node {$\underline{4}$};
\draw (1,0) rectangle (2,1); \path (1.5,.5) node {$\underline{2}$};
\draw (2,0) rectangle (3,1);\path (2.5,.5) node {$1 3 \underline{5}$};
\path (4.5,.5) node {$\leftarrow$};
\draw (5,0) rectangle (6,1);\path (5.5,.5) node {$\underline{6}$};
\path (-1.5,.5) node {${j=6\half}$};
\end{tikzpicture}
\\ \\
\begin{tikzpicture}[xscale=.6,yscale=.4]
\draw (0,0) rectangle (1,1); \path (.5,.5) node {$0$};
\draw (0,-1) rectangle (1,0); \path (.5,-.5) node {$\underline{4}$};
\draw (1,0) rectangle (2,1); \path (1.5,.5) node {$\underline{2}$};
\draw (2,0) rectangle (3,1);\path (2.5,.5) node {$1 3 \underline{5}$};
\draw (3,0) rectangle (4,1);\path (3.5,.5) node {$6 \underline{7}$};
\path (-1.5,.5) node {${j=7\phantom{\half}}$};
\end{tikzpicture}
\\ \\
\begin{tikzpicture}[xscale=.6,yscale=.4]
\draw (0,0) rectangle (1,1); \path (.5,.5) node {$0$};
\draw (1,0) rectangle (2,1); \path (1.5,.5) node {$\underline{4}$};
\draw (2,0) rectangle (3,1);\path (2.5,.5) node {$1 3 \underline{5}$};
\draw (3,0) rectangle (4,1);\path (3.5,.5) node {$6 \underline{7}$};
\path (4.5,.5) node {$\leftarrow$};
\draw (5,0) rectangle (6,1);\path (5.5,.5) node {$\underline{2}$};
\path (-1.5,.5) node {${j=7\half}$};
\end{tikzpicture}
\\ \\
\begin{tikzpicture}[xscale=.6,yscale=.4]
\draw (0,0) rectangle (1,1); \path (.5,.5) node {$0$};
\draw (1,0) rectangle (2,1); \path (1.5,.5) node {$\underline{4}$};
\draw (2,0) rectangle (3,1);\path (2.5,.5) node {$1 3 \underline{5}$};
\draw (3,0) rectangle (4,1);\path (3.5,.5) node {$6 \underline{7}$};
\draw (4,0) rectangle (5,1);\path (4.5,.5) node {$2 \underline{8}$};
\path (-1.5,.5) node {${j=8\phantom{\half}}$};
\end{tikzpicture}
\end{array}
\end{array}
$$
%
\caption{Bijection between the set partition $\pi = \{ 4 \mid 135\mid 67  \mid 2 8 \}$  and  an 8-vacillating tableau of shape $[5]$ illustrating Corollary \ref{cor:dimension}.
\label{fig:BijectionDimensionExample}}
\end{figure}
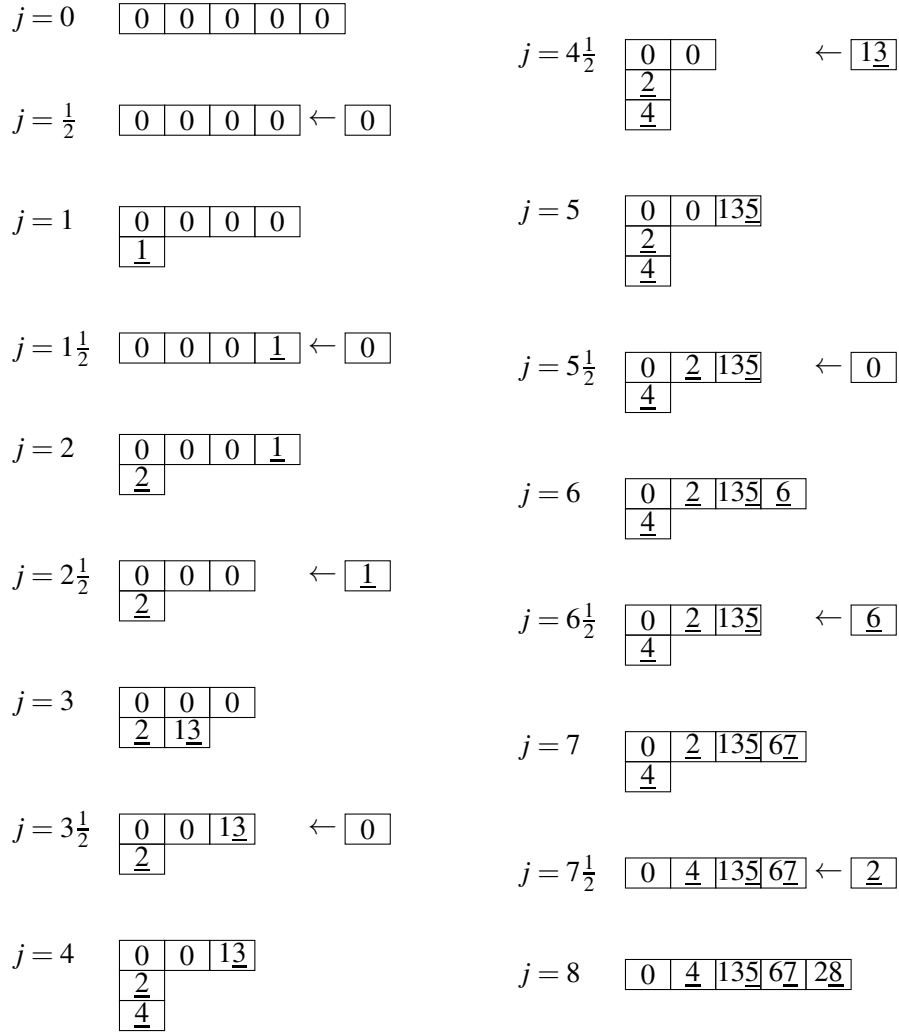

Our bijection implies the following result:

\begin{thm} For each $\lambda \in \Lambda_{k,\S_n}$, there is a bijection between the set of vacillating tableaux of shape $\lambda$ and length $k$ and the set of set-partition tableaux of shape $\lambda$ and content \\ $[0^{n-t},1,\ldots,k]$ for some $t$ such that
$|\lam^\#| \le t \le n$.
\end{thm}

Our bijection also provides a combinatorial proof of the dimension formula for $\Z_{k,n}$  (see  (b) of Theorem \ref{T:cent}),  which is illustrated with an example in Figure \ref{fig:BijectionDimensionExample}. In this case, the bijection is between set partitions of $[1,2k]$ with at most $n$ blocks and vacillating tableaux of shape $[n]$ and length $2k$.

\begin{cor} \label{cor:dimension}  For $k,n \in \ZZ_{\ge 1}$, 
 $\dim(\Z_{k,n}) = \dim(\Z_{2k,n}^{[n]}) = \m_{2k,n}^{[n]} = \mathsf{B}(2k,n) = \displaystyle{\sum_{t=1}^n} \stirling{2k}{t}$, the number of set partitions of a set of size $2k$ into at most $n$ parts. 
\end{cor}

\begin{proof}  The first equality comes from (\ref{SW4}) and the second equality from the previous corollary (or alternatively
from (\ref{SW3})).    To see the third equality, we note that $\m_{2k,n}^{[n]}$ equals the number of set-partition tableaux of shape $[n]$
having content given by a set partition $\pi \in \Pi_{2k}$ with
$t$ blocks for $t=1,2,\dots,n$ (the case $t = |[n]^\#| = 0$ not allowed).  
This is the number of set partitions of $\{1, 2, \ldots, 2k\}$ which have at most $n$ blocks, which is exactly $\mathsf{B}(2k,n)$.
\end{proof}

\begin{rem}{\rm We know that $\dim(\Z_{k,n}) = \dim(\Z_{2k,n}^{[n]}) = \m_{2k,n}^{[n]} = \sum_{\lambda \in \Lambda_{k,\S_n}} (\m_{k,n}^\lambda)^2$. We can see the last equality combinatorially as well.  Given a vacillating tableau 
$([n],[n-1],\lambda^{(1)}, \ldots,$ $\lambda^{(2k-1)}, [n-1],[n])$ 
of shape $[n]$ and length $2k$, let $\lambda = \lambda^{(k)}$. Then the first half of this tableau $([n],[n-1],\lambda^{(1)}, \ldots, $ $\lambda^{(k-\half)}, \lambda^{(k)})$, and the second half of this tableau $([n],[n-1],\lambda^{(2k-1)},  \ldots,$ $\lambda^{(k-\half)}, \lambda^{(k)})$ read in reverse, form a pair of vacillating tableaux of shape $\lambda$ and length $k$. This gives a bijection between vacillating tableaux of shape $[n]$ and length $2k$ and pairs of vacillating tableaux of shape $\lambda$
and length $k$ for some $\lambda \in \Lambda_{k,\S_n}$.}
\end{rem}

\begin{rem}{\rm  Different bijections between set partitions of $[1,2k]$ and vacillating tableaux of shape $[n]$ and length $2k$  are given in \cite{HL} and \cite{CDDSY}. However,  the bijections in those papers  require that $n \ge 2k$ holds. The bijection 
here has the advantage of working for all $k,n \in \ZZ_{\ge1}$.
}
\end{rem}

\section{The partition algebra $\P_k(n)$}\label{sec:partition}

\subsection{The diagram basis of the partition algebra $\P_k(n)$}\label{subsec:diagbasis}
Let $\pi \in \Pi_{2k}$ be a set partition of  $[1,2k]=\{1,2, \ldots, 2k\}$. The diagram $d_\pi$  of $\pi$ has two rows of $k$ vertices each,
with the bottom vertices indexed by $1,2,\dots, k$, and the top vertices indexed by $k+1,k+2, \dots, 2k$
from left to right. Edges are drawn so that the connected components of $d_\pi$ are the blocks of $\pi$.  An example
of a set partition $\pi \in \Pi_{16}$ and its corresponding diagram $d_\pi$ is displayed here: 
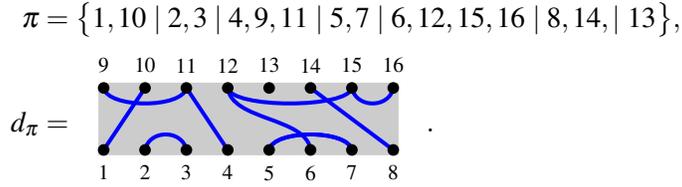
\begin{align*}
\pi & =\big\{1,10 \mid 2,3 \mid 4,9,11 \mid 5,7 \mid 6, 12, 15, 16 \mid 8,14, \mid13 \big\}, \\
d_{\pi} & =  \begin{array}{c}
\begin{tikzpicture}[xscale=.55,yscale=.55,line width=1.5pt] 
\foreach \i in {1,...,8} 
{ \path (\i,1.5) coordinate (T\i); \path (\i,0) coordinate (B\i); } 
\filldraw[fill= gray!40,draw=gray!40,line width=4pt]  (T1) -- (T8) -- (B8) -- (B1) -- (T1);
\draw[blue] (T3) -- (B4);
\draw[blue] (T2) -- (B1);
\draw[blue] (T6) -- (B8);
\draw[blue] (T4)  .. controls +(.1,-.8) and +(-.1,.8) .. (B6);
\draw[blue] (T1) .. controls +(.1,-.5) and +(-.1,-.5) .. (T3) ;
\draw[blue] (T4) .. controls +(.1,-.5) and +(-.1,-.5) .. (T7) ;
\draw[blue] (T7) .. controls +(.1,-.5) and +(-.1,-.5) .. (T8) ;
\draw[blue] (B2) .. controls +(.1,.5) and +(-.1,.5) .. (B3) ;
\draw[blue] (B5) .. controls +(.1,.5) and +(-.1,.5) .. (B7) ;
{\draw[blue] (B5) .. controls +(.1,.5) and +(-.1,.5) .. (B7) ;}
\draw  (B1)  node[below=0.05cm]{${\scriptstyle 1}$}; \draw  (T1)  node[above=0.05cm]{${\scriptstyle 9}$};
\draw  (B2)  node[below=0.05cm]{${\scriptstyle 2}$}; \draw  (T2)  node[above=0.05cm]{${\scriptstyle 10}$};
\draw  (B3)  node[below=0.05cm]{${\scriptstyle 3}$}; \draw  (T3)  node[above=0.05cm]{${\scriptstyle 11}$};
\draw  (B4)  node[below=0.05cm]{${\scriptstyle 4}$}; \draw  (T4)  node[above=0.05cm]{${\scriptstyle 12}$};
\draw  (B5)  node[below=0.05cm]{${\scriptstyle 5}$}; \draw  (T5)  node[above=0.05cm]{${\scriptstyle 13}$};
\draw  (B6)  node[below=0.05cm]{${\scriptstyle 6}$}; \draw  (T6)  node[above=0.05cm]{${\scriptstyle 14}$};
\draw  (B7)  node[below=0.05cm]{${\scriptstyle 7}$}; \draw  (T7)  node[above=0.05cm]{${\scriptstyle 15}$};
\draw  (B8)  node[below=0.05cm]{${\scriptstyle 8}$}; \draw  (T8)  node[above=0.05cm]{${\scriptstyle 16}$};
\foreach \i in {1,...,8} 
{ \fill (T\i) circle (4pt); \fill (B\i) circle (4pt); } 
\end{tikzpicture} \end{array}.
\end{align*}
The way the edges are drawn is immaterial; what matters is that the connected components of the diagram $d_\pi$ 
correspond to the blocks of the set partition $\pi$. Thus,  $d_\pi$  represents the equivalence class of all diagrams with connected components equal to the blocks of   $\pi$.

Multiplication of two diagrams $d_{\pi_1}$, $d_{\pi_2}$ is accomplished by placing $d_{\pi_1}$ above $d_{\pi_2}$,
identifying the vertices in the bottom row of $d_{\pi_1}$ with those in the top row of $d_{\pi_2}$,
concatenating  the edges,  deleting all connected components that lie entirely in the middle row of the joined diagrams,  and
multiplying by a factor of $n$ for each such middle-row component.  
For example,  if  
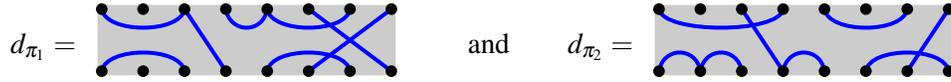
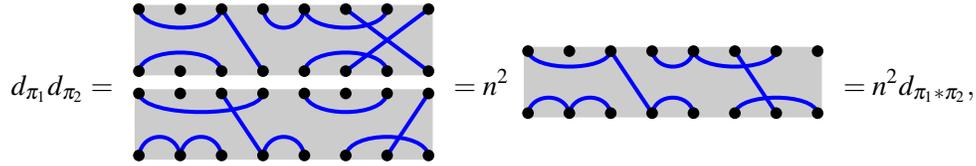
\begin{equation*} 
d_{\pi_1} =  {\begin{array}{c}
\begin{tikzpicture}[scale=.55,line width=1.5pt] 
\foreach \i in {1,...,8} 
{ \path (\i,1.5) coordinate (T\i); \path (\i,0) coordinate (B\i); } 
\filldraw[fill= gray!40,draw=gray!40,line width=3.2pt]  (T1) -- (T8) -- (B8) -- (B1) -- (T1);
\draw[blue] (T3) -- (B4);
\draw[blue] (T6) -- (B8);
\draw[blue] (T8) -- (B6);
\draw[blue] (T1) .. controls +(.1,-.6) and +(-.1,-.6) .. (T3) ;
\draw[blue] (T4) .. controls +(.1,-.6) and +(-.1,-.6) .. (T5) ;
\draw[blue] (T5) .. controls +(.1,-.6) and +(-.1,-.6) .. (T7) ;
\draw[blue] (B1) .. controls +(.1,.6) and +(-.1,.6) .. (B3) ;
\draw[blue] (B5) .. controls +(.1,.6) and +(-.1,.6) .. (B7) ;
\foreach \i in {1,...,8}  { \fill (T\i) circle (4pt); \fill (B\i) circle (4pt); } 
\end{tikzpicture}\end{array}} 
\qquad\hbox{and}\qquad
d_{\pi_2} =
{\begin{array}{c}
\begin{tikzpicture}[scale=.55,line width=1.5pt] 
\foreach \i in {1,...,8} { \path (\i,1.5) coordinate (T\i); \path (\i,0) coordinate (B\i); } 
\filldraw[fill=gray!40,draw=gray!40,line width=3.2pt]  (T1) -- (T8) -- (B8) -- (B1) -- (T1);
\draw[blue] (T8) -- (B7);
\draw[blue] (T3) -- (B4);
\draw[blue] (T1) .. controls +(.1,-.6) and +(-.1,-.6) .. (T4) ;
\draw[blue] (T5) .. controls +(.1,-.6) and +(-.1,-.6) .. (T7) ;
\draw[blue] (B1) .. controls +(.1,.6) and +(-.1,.6) .. (B2) ;
\draw[blue] (B2) .. controls +(.1,.6) and +(-.1,.6) .. (B3) ;
\draw[blue] (B4) .. controls +(.1,.6) and +(-.1,.6) .. (B5) ;
\draw[blue] (B6) .. controls +(.1,.6) and +(-.1,.6) .. (B8) ;
\foreach \i in {1,...,8}  { \fill (T\i) circle (4pt); \fill (B\i) circle (4pt); } 
\end{tikzpicture} \end{array}} 
\end{equation*}  
then
\begin{equation*}
d_{\pi_1} d_{\pi_2} = 
\begin{array}{c}
\begin{tikzpicture}[scale=.55,line width=1.5pt] 
\foreach \i in {1,...,8} 
{ \path (\i,1.5) coordinate (T\i); \path (\i,0) coordinate (B\i); } 
\filldraw[fill= gray!40,draw=gray!40,line width=3.2pt]  (T1) -- (T8) -- (B8) -- (B1) -- (T1);
\draw[blue] (T3) -- (B4);
\draw[blue] (T6) -- (B8);
\draw[blue] (T8) -- (B6);
\draw[blue] (T1) .. controls +(.1,-.6) and +(-.1,-.6) .. (T3) ;
\draw[blue] (T4) .. controls +(.1,-.6) and +(-.1,-.6) .. (T5) ;
\draw[blue] (T5) .. controls +(.1,-.6) and +(-.1,-.6) .. (T7) ;
\draw[blue] (B1) .. controls +(.1,.6) and +(-.1,.6) .. (B3) ;
\draw[blue] (B5) .. controls +(.1,.6) and +(-.1,.6) .. (B7) ;
\foreach \i in {1,...,8}  { \fill (T\i) circle (4pt); \fill (B\i) circle (4pt); } 
\end{tikzpicture}
\\
\begin{tikzpicture}[scale=.55,line width=1.5pt] 
\foreach \i in {1,...,8} { \path (\i,1.5) coordinate (T\i); \path (\i,0) coordinate (B\i); } 
\filldraw[fill=gray!40,draw=gray!40,line width=3.2pt]  (T1) -- (T8) -- (B8) -- (B1) -- (T1);
\draw[blue] (T8) -- (B7);
\draw[blue] (T3) -- (B4);
\draw[blue] (T1) .. controls +(.1,-.6) and +(-.1,-.6) .. (T4) ;
\draw[blue] (T5) .. controls +(.1,-.6) and +(-.1,-.6) .. (T7) ;
\draw[blue] (B1) .. controls +(.1,.6) and +(-.1,.6) .. (B2) ;
\draw[blue] (B2) .. controls +(.1,.6) and +(-.1,.6) .. (B3) ;
\draw[blue] (B4) .. controls +(.1,.6) and +(-.1,.6) .. (B5) ;
\draw[blue] (B6) .. controls +(.1,.6) and +(-.1,.6) .. (B8) ;
\foreach \i in {1,...,8}  { \fill (T\i) circle (4pt); \fill (B\i) circle (4pt); } 
\end{tikzpicture} \end{array} 
=n^2 
{\begin{array}{c}  
{\begin{tikzpicture}[scale=.55,line width=1.5pt] 
\foreach \i in {1,...,8} 
{ \path (\i,1.5) coordinate (T\i); \path (\i,0) coordinate (B\i); } 
\filldraw[fill= gray!40,draw=gray!40,line width=3.2pt]  (T1) -- (T8) -- (B8) -- (B1) -- (T1);
\draw[blue] (T3) -- (B4);
\draw[blue] (T6) -- (B7);
\draw[blue] (T1) .. controls +(.1,-.5) and +(-.1,-.5) .. (T3) ;
\draw[blue] (T4) .. controls +(.1,-.5) and +(-.1,-.5) .. (T5) ;
\draw[blue] (T5) .. controls +(.1,-.5) and +(-.1,-.5) .. (T7) ;
\draw[blue] (B1) .. controls +(.1,.5) and +(-.1,.5) .. (B2) ;
\draw[blue] (B2) .. controls +(.1,.5) and +(-.1,.5) .. (B3) ;
\draw[blue] (B4) .. controls +(.1,.5) and +(-.1,.5) .. (B5) ;
\draw[blue] (B6) .. controls +(.1,.5) and +(-.1,.5) .. (B8) ;
\foreach \i in {1,...,8} { \fill (T\i) circle (4pt); \fill (B\i) circle (4pt); } 
\end{tikzpicture}}
 \end{array} =  n^2 d_{\pi_1 \ast \pi_2},} 
 \end{equation*} 
where $\pi_1 \ast \pi_2$ is the set partition obtained by the concatenation of $\pi_1$ and $\pi_2$ in this process.   It is easy to confirm that the product depends only on the underlying set partition and is independent of the diagram chosen to represent $\pi$.
For any two set partitions $\pi_1,\pi_2 \in \Pi_{2k}$, we let $[\pi_1 \ast \pi_2]$ denote the number of blocks removed from the middle of the product $d_{\pi_1}d_{\pi_2}$, so that the product is given by 
\begin{equation}\label{eq:diagmult} d_{\pi_1} d_{\pi_2} = n^{[\pi_1 \ast \pi_2]} d_{\pi_1 \ast \pi_2}. \end{equation} 

For  $n \in \ZZ_{\ge 1}$ and for $k \in\ZZ_{\ge 1}$,  the partition algebra $\P_k(n)$ is the $\FF$-span of $\{d_{\pi} \mid \pi \in \Pi_{2k}\}$ under the diagram multiplication in \eqref{eq:diagmult}.  Thus, $\dim(\P_{k}(n)) = \mathsf{B}(2k)$, the $2k$-th Bell number.
We refer to  $\{d_{\pi} \mid \pi \in \Pi_{2k}\}$ as the \emph{diagram basis}.  Diagram multiplication is easily seen to be associative with identity 
element  $\mathsf{I}_k$ corresponding to the set partition, \\  $\big\{ 1,k+1 \mid 2, k+2 \mid \cdots \mid k,2k\big\}$,
where
\begin{equation}\label{id-def}
\mathsf{I}_k   =  
\begin{array}{c}\begin{tikzpicture}[scale=.5,line width=1.5pt] 
\foreach \i in {1,...,6} 
{ \path (\i,1.25) coordinate (T\i); \path (\i,0) coordinate (B\i); } 
\filldraw[fill= gray!40,draw=gray!40,line width=3.2pt]  (T1) -- (T6) -- (B6) -- (B1) -- (T1);
\draw[blue] (T1) -- (B1);
\draw[blue] (T2) -- (B2);
\draw[blue] (T3) -- (B3);
\draw[blue] (T5) -- (B5);
\draw[blue] (T6) -- (B6);
\foreach \i in {1,2,3,5,6} { \fill (T\i) circle (4pt); \fill (B\i) circle (4pt); } 
\draw (T4) node  {$\qquad \cdots \qquad $ }; \draw (B4) node  { $\qquad \cdots\qquad $ }; 
\end{tikzpicture}\end{array}.
\end{equation}

If $\pi_1, \pi_2 \in \Pi_{2k-1}$, so that $k$ and $2k$ are in the same block in both $\pi_1$ and $\pi_2$, then $k$ and $2k$ are also in the same block of $\pi_1 \ast \pi_2$. Thus, for $k \in \ZZ_{\ge 1}$, we define $\P_{k-\frac{1}{2}}(n) \subset \P_{k}(n)$ to be the  $\FF$-span of $\{d_\pi \mid \pi \in \Pi_{2k-1} \subset \Pi_{2k}\}$. There is also an embedding  $\P_{k}(n) \subset  \P_{k+\frac{1}{2}}(n)$ given by adding a top and bottom node to the right of any diagram in $\P_k(n)$ and a vertical edge connecting them. Setting $\P_0(n) = \FF$, we have a tower of embeddings
\begin{equation}
\P_0(n)  \cong  \P_{\frac{1}{2}}(n) \subset  \P_{1}(n) \subset   \P_{1\frac{1}{2}}(n) \subset  \P_{2}(n)\subset  \P_{2\frac{1}{2}}(n) \subset  \cdots
\end{equation}
with $\dim(\P_k(n)) = |\Pi_{2k}| =  \mathsf{B}(2k)$ (the $2k$-th Bell number) for each $k \in \frac{1}{2}\ZZ_{\ge 1}$.

\subsection{The orbit basis}\label{subsec:orbitbasis}

For $k \in \ZZ_{\ge 1}$, the set partitions  $\Pi_{2k}$ of $[1,2k]$ form a  lattice (a partially ordered set (poset) for which each pair has a least upper bound and a greatest lower bound) under the partial order given by  
\begin{equation}
\pi \preceq \vr \quad \hbox{ if every block of $\pi$ is contained in a block of $\vr$.}
\end{equation}
In this case, we say that $\pi$ is a {\it refinement} of $\vr$  and that $\vr$ is a {\it coarsening} of $\pi$, so that $\Pi_{2k}$ is partially ordered by refinement.  

For each $k \in \half  \ZZ_{\ge 1}$, there is a second basis $\left\{x_\pi  \mid \pi \in \Pi_{2k} \right\}$ of $\P_k(n)$, called the \emph{orbit basis}, defined by the following coarsening relation with respect to the diagram basis:
\begin{equation}\label{refinement-relation}
d_\pi = \sum_{\pi \preceq \vr} x_\vr.
\end{equation}
Thus, the diagram basis element $d_\pi$ is the sum of all orbit basis elements $x_\vr$ for which $\vr$ is  coarser than $\pi$. 
For the remainder of the paper,  we adopt the following convention used in \cite{BH}: 
\begin{equation}
\begin{array}{l} \hbox{\emph{Diagrams with white vertices indicate orbit basis elements, and}} \\
\hbox{\emph{those with black vertices indicate diagram basis elements.}} 
\end{array}
\end{equation}
For example, the expression below writes the diagram $d_{1|23|4}$ in $\P_2(n)$ in terms of the orbit basis,
\begin{equation*}
\begin{array}{c}
\begin{tikzpicture}[xscale=.5,yscale=.5,line width=1.25pt] 
\foreach \i in {1,2}  { \path (\i,1.25) coordinate (T\i); \path (\i,.25) coordinate (B\i); } 
\filldraw[fill= black!12,draw=black!12,line width=4pt]  (T1) -- (T2) -- (B2) -- (B1) -- (T1);
\draw[blue] (B2) -- (T1);
\foreach \i in {1,2}  { \fill (T\i) circle (4pt); \fill  (B\i) circle (4pt); } 
\end{tikzpicture}
\end{array}
=
\begin{array}{c}
\begin{tikzpicture}[xscale=.5,yscale=.5,line width=1.25pt] 
\foreach \i in {1,2}  { \path (\i,1.25) coordinate (T\i); \path (\i,.25) coordinate (B\i); } 
\filldraw[fill= black!12,draw=black!12,line width=3pt]  (T1) -- (T2) -- (B2) -- (B1) -- (T1);
\draw[blue] (B2) -- (T1);
\foreach \i in {1,2}  { \filldraw[fill=white,draw=black,line width = 1pt] (T\i) circle (4pt); \filldraw[fill=white,draw=black,line width = 1pt]  (B\i) circle (4pt); } 
\end{tikzpicture}
\end{array}
+\begin{array}{c}
\begin{tikzpicture}[xscale=.5,yscale=.5,line width=1.25pt] 
\foreach \i in {1,2}  { \path (\i,1.25) coordinate (T\i); \path (\i,.25) coordinate (B\i); } 
\filldraw[fill= black!12,draw=black!12,line width=3pt]  (T1) -- (T2) -- (B2) -- (B1) -- (T1);
\draw[blue] (B2) -- (T1);
\draw[blue] (B1) -- (B2);
\draw[blue] (B1) -- (T1);
\foreach \i in {1,2}  { \filldraw[fill=white,draw=black,line width = 1pt] (T\i) circle (4pt); \filldraw[fill=white,draw=black,line width = 1pt]  (B\i) circle (4pt); } 
\end{tikzpicture}
\end{array}
+\begin{array}{c}
\begin{tikzpicture}[xscale=.5,yscale=.5,line width=1.25pt] 
\foreach \i in {1,2}  { \path (\i,1.25) coordinate (T\i); \path (\i,.25) coordinate (B\i); } 
\filldraw[fill= black!12,draw=black!12,line width=3pt]  (T1) -- (T2) -- (B2) -- (B1) -- (T1);
\draw[blue] (B2) -- (T1);
\draw[blue] (T1) -- (T2);
\draw[blue] (T2) -- (B2);
\foreach \i in {1,2}  { \filldraw[fill=white,draw=black,line width = 1pt] (T\i) circle (4pt); \filldraw[fill=white,draw=black,line width = 1pt]  (B\i) circle (4pt); } 
\end{tikzpicture}
\end{array}
+\begin{array}{c}
\begin{tikzpicture}[xscale=.5,yscale=.5,line width=1.25pt] 
\foreach \i in {1,2}  { \path (\i,1.25) coordinate (T\i); \path (\i,.25) coordinate (B\i); } 
\filldraw[fill= black!12,draw=black!12,line width=3pt]  (T1) -- (T2) -- (B2) -- (B1) -- (T1);
\draw[blue] (B1) -- (T2);
\draw[blue] (T1) -- (B2);
\foreach \i in {1,2}  { \filldraw[fill=white,draw=black,line width = 1pt] (T\i) circle (4pt); \filldraw[fill=white,draw=black,line width = 1pt]  (B\i) circle (4pt); } 
\end{tikzpicture}
\end{array}
+\begin{array}{c}
\begin{tikzpicture}[xscale=.5,yscale=.5,line width=1.25pt] 
\foreach \i in {1,2}  { \path (\i,1.25) coordinate (T\i); \path (\i,.25) coordinate (B\i); } 
\filldraw[fill= black!12,draw=black!12,line width=3pt]  (T1) -- (T2) -- (B2) -- (B1) -- (T1);
\draw[blue] (B1) -- (T2);
\draw[blue] (T1) -- (T2);\draw[blue] (B1) -- (B2);
\draw[blue] (T1) -- (B1);
\draw[blue] (T2) -- (B2);
\draw[blue] (T1) -- (B2);
\foreach \i in {1,2}  { \filldraw[fill=white,draw=black,line width = 1pt] (T\i) circle (4pt); \filldraw[fill=white,draw=black,line width = 1pt]  (B\i) circle (4pt); } 
\end{tikzpicture}
\end{array}.
\end{equation*}
 
\begin{rem}{\rm
We refer to the basis $\{x_\pi  \mid \pi \in \Pi_{2k}\}$ as the orbit basis,  because the elements in this basis act on the tensor space $\modu^{\ot k}$ in a natural way that corresponds to $\S_n$-orbits on simple tensors (see \eqref{eq:Phi-coeffs-orbit}).   Jones' original definition of the partition algebra in \cite{J} introduced the orbit basis  first and defined  the diagram basis later using the refinement relation \eqref{refinement-relation}.  The multiplication rule is 
more easily stated  in the diagram basis, and for that reason the diagram basis is usually the preferred basis when working with $\P_k(n)$.
 However, when working with $\P_k(n)$ for $2k > n$ the orbit basis is especially useful. For example, in Section \ref{sec:tensorrepresentation} we are able to describe the kernel of the 
action of $\P_k(n)$ on tensor space as the two-sided ideal generated by a single orbit basis element.
}
\end{rem}

\subsection{Change of basis}\label{subsec:change}   
  
The transition matrix determined by \eqref{refinement-relation} between the diagram basis  and the orbit basis is the matrix $\zeta_{2k}$, called the \emph{zeta matrix} of the poset $\Pi_{2k}$. It is unitriangular with respect to any extension to a linear order, and thus it is invertible, confirming that indeed the elements $x_{\pi}, \pi \in \Pi_{2k}$,  form a basis of $\P_k(n)$. 
The inverse of $\zeta_{2k}$ is the matrix $\mu_{2k}$ of the M\"obius function of the set-partition lattice, and it satisfies
\begin{equation}\label{eq:mobiusa}
x_\pi = \sum_{\pi \preceq \vr} \mu_{2k}(\pi,\vr) d_\vr,
\end{equation}
where $\mu_{2k}(\pi,\vr)$ is the $(\pi,\vr)$ entry of $\mu_{2k}$. The M\"obius function of the set-partition lattice can be readily computed
using the following formula.  If  $\pi \preceq \vr$,  and $\vr$ consists of $\ell$ blocks such that the $i$th block of $\vr$ is the union of $b_i$ blocks  of $\pi$, then (see, for example, \cite[p.~30]{S1}), 
\begin{equation}\label{eq:mobiusb}
 \mu_{2k}(\pi,\vr) = \prod_{i=1}^\ell (-1)^{b_i-1}(b_i-1)!.
\end{equation}
 
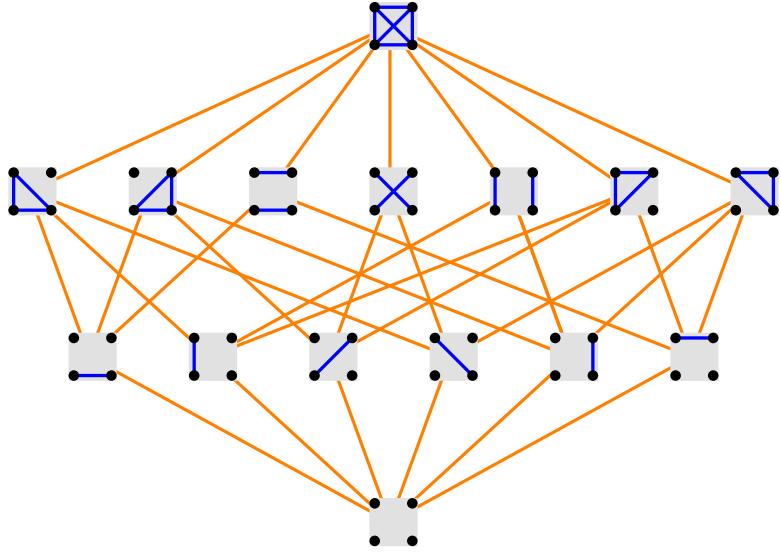
\begin{figure}[h!]
$$
  \begin{tikzpicture}[xscale=.08,yscale=.22,line width=1.25pt] 
  nodes={draw=white, 
      minimum width=.2cm, minimum height=.5cm
    }

\path (70,30) coordinate (m-1234); 

\path (10,20) coordinate (m-123-4); 
\path (30,20) coordinate (m-124-3); 
\path (50,20) coordinate (m-12-34); 
\path (70,20) coordinate (m-14-23); 
\path (90,20) coordinate (m-13-24); 
\path (110,20) coordinate (m-134-2); 
\path (130,20) coordinate (m-1-234); 

\path (20,10) coordinate (m-12-3-4); 
\path (40,10) coordinate (m-13-2-4); 
\path (60,10) coordinate (m-14-2-3); 
\path (80,10) coordinate (m-23-1-4); 
\path (100,10) coordinate (m-1-3-24); 
\path (120,10) coordinate (m-1-2-34); 

\path (70,0) coordinate (m-1-2-3-4);

\draw (m-1234) node  {\abcd}; 

\draw (m-123-4) node  {\abcld}; 
\draw (m-124-3) node  {\abdlc}; 
\draw (m-12-34) node  {\ablcd}; 
\draw (m-134-2) node  {\acdlb}; 
\draw (m-13-24) node  {\aclbd}; 
\draw (m-14-23) node  {\adlbc}; 
\draw (m-1-234) node  {\albcd}; 

\draw (m-12-3-4) node  {\ablcld}; 
\draw (m-13-2-4) node  {\aclbld}; 
\draw (m-14-2-3) node  {\adlblc}; 
\draw (m-23-1-4) node  {\albcld}; 
\draw (m-1-3-24) node  {\albdlc}; 
\draw (m-1-2-34) node  {\alblcd};

\draw (m-1-2-3-4) node  {\alblcld}; 

\begin{scope}[on background layer]
 
\path  (m-1234) edge[orange,line width=1.2pt] (m-123-4);
\path  (m-1234) edge[orange,line width=1.2pt] (m-124-3);
\path  (m-1234) edge[orange,line width=1.2pt] (m-12-34);
\path  (m-1234) edge[orange,line width=1.2pt] (m-134-2);
\path  (m-1234) edge[orange,line width=1.2pt] (m-13-24);
\path  (m-1234) edge[orange,line width=1.2pt] (m-14-23);
\path  (m-1234) edge[orange,line width=1.2pt] (m-1-234);

\path  (m-12-3-4) edge[orange,line width=1.2pt] (m-123-4);
\path  (m-13-2-4) edge[orange,line width=1.2pt] (m-123-4);
\path  (m-23-1-4) edge[orange,line width=1.2pt] (m-123-4);
\path  (m-12-3-4) edge[orange,line width=1.2pt] (m-124-3);
\path  (m-14-2-3) edge[orange,line width=1.2pt] (m-124-3);
\path  (m-1-3-24) edge[orange,line width=1.2pt] (m-124-3);
\path  (m-12-3-4) edge[orange,line width=1.2pt] (m-12-34);
\path  (m-1-2-34) edge[orange,line width=1.2pt] (m-12-34);
\path  (m-13-2-4) edge[orange,line width=1.2pt] (m-134-2);
\path  (m-14-2-3) edge[orange,line width=1.2pt] (m-134-2);
\path  (m-1-2-34) edge[orange,line width=1.2pt] (m-134-2);
\path  (m-13-2-4) edge[orange,line width=1.2pt] (m-13-24);
\path  (m-1-3-24) edge[orange,line width=1.2pt] (m-13-24);
\path  (m-1-3-24) edge[orange,line width=1.2pt] (m-13-24);
\path  (m-14-2-3) edge[orange,line width=1.2pt] (m-14-23);
\path  (m-23-1-4) edge[orange,line width=1.2pt] (m-14-23);
\path  (m-1-2-34) edge[orange,line width=1.2pt] (m-1-234);
\path  (m-1-3-24) edge[orange,line width=1.2pt] (m-1-234);
\path  (m-23-1-4) edge[orange,line width=1.2pt] (m-1-234);

\path  (m-12-3-4) edge[orange,line width=1.2pt] (m-1-2-3-4);
\path  (m-13-2-4) edge[orange,line width=1.2pt] (m-1-2-3-4);
\path  (m-14-2-3) edge[orange,line width=1.2pt] (m-1-2-3-4);
\path  (m-23-1-4) edge[orange,line width=1.2pt] (m-1-2-3-4);
\path  (m-1-3-24) edge[orange,line width=1.2pt] (m-1-2-3-4);
\path  (m-1-2-34) edge[orange,line width=1.2pt] (m-1-2-3-4);
    \end{scope}
  \end{tikzpicture}
$$  
\caption{Hasse diagram of the partition lattice $\Pi_4$ in the refinement ordering.\label{fig:Hasse}}
\end{figure}
The Hasse diagram of the partition lattice of $\Pi_4$ is shown in Figure \ref{fig:Hasse}. In the change of basis between the orbit basis and the diagram basis (in either direction), each basis element is an integer linear combination of the basis elements above or equal to it in the Hasse diagram. 
If, for example, we apply  formula \eqref{eq:mobiusa} to express $x_{1 \,|\,  2\,|\,  3\,|\,  4}$ and $x_{1\,|\,  23\,|\, 4}$ in the orbit basis in terms of the diagram basis in $\P_2(n)$, we get
\begin{align}
\begin{split}
\label{eq:orbit2diag}
\begin{array}{c}\orbitalblcld\end{array} ~=~ &  \begin{array}{c}\alblcld\end{array}  - \begin{array}{c}\ablcld\end{array} - \begin{array}{c}\aclbld\end{array} - 
 \begin{array}{c}\adlblc\end{array} - \begin{array}{c}\albcld\end{array} - \begin{array}{c}\albdlc\end{array} - \begin{array}{c}\alblcd\end{array}
 +2 \begin{array}{c}\abcld\end{array} \\ &+2 \begin{array}{c}\abdlc\end{array} + \begin{array}{c}\ablcd\end{array}+ \begin{array}{c}\adlbc\end{array}
+ \begin{array}{c}\aclbd\end{array} + 2 \begin{array}{c}\acdlb\end{array} +2 \begin{array}{c}\albcd\end{array}  - 6  \begin{array}{c}\abcd\end{array}, \\
\begin{array}{c}\orbita|bc|d \end{array} ~=~ & \begin{array}{c}\albcld\end{array} - \begin{array}{c}\abcld\end{array}
- \begin{array}{c}\albcd\end{array}  - \begin{array}{c}\adlbc\end{array}  + 2  \begin{array}{c}\abcd \end{array}. 
\end{split} \end{align}

\begin{rem}\label{R:Hasse}{ \rm 
If $\pi \in \Pi_{2k-1} \subset \Pi_{2k}$, then $k$ and $2k$ are in the same block of $\pi$. Since the expression for $x_\pi$ in  the diagram basis  and the expression for $d_\pi$ in the orbit basis are sums over coarsenings of $\pi$,  they involve only set partitions in $\Pi_{2k-1}$.  Thus, the relations in \eqref{refinement-relation} and \eqref{eq:mobiusa} apply equally well to the algebras $\P_{k-\frac{1}{2}}(n)$. The corresponding Hasse diagram  is the sublattice of the diagram for $\Pi_{2k}$ of partitions greater than or equal to $\{1\,|\,2\,| \cdots |\, k-1\,|\, k+1\,|\, k+2\,| \cdots |\, 2k-1\,|\, k, 2k \}$.  For example, the Hasse diagram for $\Pi_3$ is found inside that of $\Pi_4$ in Figure  \ref{fig:Hasse} as those partitions greater than or equal to $\begin{array}{c}\begin{tikzpicture}[scale=.4,line width=1.25pt] 
\foreach \i in {1,2}  { \path (\i,1) coordinate (T\i); \path (\i,0) coordinate (B\i); } 
\filldraw[fill= gray!40,draw=gray!40,line width=3.2pt]  (T1) -- (T2) -- (B2) -- (B1) -- (T1);
\draw[blue] (T2) -- (B2);
\foreach \i in {1,2} { \fill (T\i) circle (4pt); \fill (B\i) circle (4pt); } 
\end{tikzpicture}\end{array}$.  
}
\end{rem} 

\begin{rem}\label{R:orbid}{\rm The orbit diagram 
$\mathsf{I}_k^{\boldsymbol \circ} :=\begin{array}{c} 
\begin{tikzpicture}[xscale=.35,yscale=.35,line width=1.0pt] 
\foreach \i in {1,2,3,4}  { \path (\i,1.25) coordinate (T\i); \path (\i,.25) coordinate (B\i); } 
\filldraw[fill= black!12,draw=black!12,line width=4pt]  (T1) -- (T4) -- (B4) -- (B1) -- (T1);
\draw[blue] (T1) -- (B1);
\draw[blue] (T2) -- (B2);
\draw[blue] (T4) -- (B4);
\draw (T3) node {$\cdots$};\draw (B3) node {$\cdots$};
\foreach \i in {1,2,4}  { \filldraw[fill=white,draw=black,line width = 1pt] (T\i) circle (4pt); \filldraw[fill=white,draw=black,line width = 1pt]  (B\i) circle (4pt); } 
\end{tikzpicture} 
\end{array}$ is \emph{not} the identity element  in $\P_k(n)$. To get the identity element $\mathsf{I}_k$, we must sum all coarsenings of 
 $\mathsf{I}_k^{\boldsymbol \circ}$,
as in the first line below:
\begin{align*}
\mathsf{I}_k &= \begin{array}{c}
\begin{tikzpicture}[xscale=.45,yscale=.45,line width=1.0pt] 
\foreach \i in {1,2,3}  { \path (\i,1.25) coordinate (T\i); \path (\i,.25) coordinate (B\i); } 
\filldraw[fill= black!12,draw=black!12,line width=4pt]  (T1) -- (T3) -- (B3) -- (B1) -- (T1);
\draw[blue] (T1) -- (B1);
\draw[blue] (T2) -- (B2);
\draw[blue] (T3) -- (B3);
\foreach \i in {1,2,3}  { \filldraw[fill=black,draw=black,line width = 1pt] (T\i) circle (4pt); \filldraw[fill=black,draw=black,line width = 1pt]  (B\i) circle (4pt); } 
\end{tikzpicture}\end{array}
=
\begin{array}{c}\begin{tikzpicture}[xscale=.45,yscale=.45,line width=1.0pt] 
\foreach \i in {1,2,3}  { \path (\i,1.25) coordinate (T\i); \path (\i,.25) coordinate (B\i); } 
\filldraw[fill= black!12,draw=black!12,line width=4pt]  (T1) -- (T3) -- (B3) -- (B1) -- (T1);
\draw[blue] (T1) -- (B1);
\draw[blue] (T2) -- (B2);
\draw[blue] (T3) -- (B3);
\foreach \i in {1,2,3}  { \filldraw[fill=white,draw=black,line width = 1pt] (T\i) circle (4pt); \filldraw[fill=white,draw=black,line width = 1pt]  (B\i) circle (4pt); } 
\end{tikzpicture}\end{array}
+
\begin{array}{c}\begin{tikzpicture}[xscale=.45,yscale=.45,line width=1.0pt] 
\foreach \i in {1,2,3}  { \path (\i,1.25) coordinate (T\i); \path (\i,.25) coordinate (B\i); } 
\filldraw[fill= black!12,draw=black!12,line width=4pt]  (T1) -- (T3) -- (B3) -- (B1) -- (T1);
\draw[blue] (T1) -- (B1)--(B2);
\draw[blue] (T1)--(T2) -- (B2);
\draw[blue] (T3) -- (B3);
\foreach \i in {1,2,3}  { \filldraw[fill=white,draw=black,line width = 1pt] (T\i) circle (4pt); \filldraw[fill=white,draw=black,line width = 1pt]  (B\i) circle (4pt); } 
\end{tikzpicture}\end{array}
+
\begin{array}{c}\begin{tikzpicture}[xscale=.45,yscale=.45,line width=1.0pt] 
\foreach \i in {1,2,3}  { \path (\i,1.25) coordinate (T\i); \path (\i,.25) coordinate (B\i); } 
\filldraw[fill= black!12,draw=black!12,line width=4pt]  (T1) -- (T3) -- (B3) -- (B1) -- (T1);
\draw[blue] (T1) -- (B1);
\draw[blue] (T2) -- (B2);
\draw[blue] (T3) -- (B3);
\draw[blue] (T1) .. controls +(0,-.50) and +(0,-.50) .. (T3);
\draw[blue] (B1) .. controls +(0,.50) and +(0,.50) .. (B3);
\foreach \i in {1,2,3}  { \filldraw[fill=white,draw=black,line width = 1pt] (T\i) circle (4pt); \filldraw[fill=white,draw=black,line width = 1pt]  (B\i) circle (4pt); } 
\end{tikzpicture}\end{array}
+
\begin{array}{c}\begin{tikzpicture}[xscale=.45,yscale=.45,line width=1.0pt] 
\foreach \i in {1,2,3}  { \path (\i,1.25) coordinate (T\i); \path (\i,.25) coordinate (B\i); } 
\filldraw[fill= black!12,draw=black!12,line width=4pt]  (T1) -- (T3) -- (B3) -- (B1) -- (T1);
\draw[blue] (T1) -- (B1);
\draw[blue] (T2) -- (B2)--(B3);
\draw[blue] (T2)--(T3) -- (B3);
\foreach \i in {1,2,3}  { \filldraw[fill=white,draw=black,line width = 1pt] (T\i) circle (4pt); \filldraw[fill=white,draw=black,line width = 1pt]  (B\i) circle (4pt); } 
\end{tikzpicture}\end{array}
+
\begin{array}{c}\begin{tikzpicture}[xscale=.45,yscale=.45,line width=1.0pt] 
\foreach \i in {1,2,3}  { \path (\i,1.25) coordinate (T\i); \path (\i,.25) coordinate (B\i); } 
\filldraw[fill= black!12,draw=black!12,line width=4pt]  (T1) -- (T3) -- (B3) -- (B1) -- (T1);
\draw[blue] (T1) -- (B1) -- (B2);
\draw[blue] (T1)--(T2) -- (B2) -- (B3);
\draw[blue] (T2)--(T3) -- (B3);
\foreach \i in {1,2,3}  { \filldraw[fill=white,draw=black,line width = 1pt] (T\i) circle (4pt); \filldraw[fill=white,draw=black,line width = 1pt]  (B\i) circle (4pt); } 
\end{tikzpicture}\end{array}  \qquad \hbox{and} \\
\mathsf{I}_k^{\boldsymbol \circ}&=
\begin{array}{c}\begin{tikzpicture}[xscale=.45,yscale=.45,line width=1.0pt] 
\foreach \i in {1,2,3}  { \path (\i,1.25) coordinate (T\i); \path (\i,.25) coordinate (B\i); } 
\filldraw[fill= black!12,draw=black!12,line width=4pt]  (T1) -- (T3) -- (B3) -- (B1) -- (T1);
\draw[blue] (T1) -- (B1);
\draw[blue] (T2) -- (B2);
\draw[blue] (T3) -- (B3);
\foreach \i in {1,2,3}  { \filldraw[fill=white,draw=black,line width = 1pt] (T\i) circle (4pt); \filldraw[fill=white,draw=black,line width = 1pt]  (B\i) circle (4pt); } 
\end{tikzpicture}\end{array}
=
\begin{array}{c}\begin{tikzpicture}[xscale=.45,yscale=.45,line width=1.0pt] 
\foreach \i in {1,2,3}  { \path (\i,1.25) coordinate (T\i); \path (\i,.25) coordinate (B\i); } 
\filldraw[fill= black!12,draw=black!12,line width=4pt]  (T1) -- (T3) -- (B3) -- (B1) -- (T1);
\draw[blue] (T1) -- (B1);
\draw[blue] (T2) -- (B2);
\draw[blue] (T3) -- (B3);
\foreach \i in {1,2,3}  { \filldraw[fill=black,draw=black,line width = 1pt] (T\i) circle (4pt); \filldraw[fill=black,draw=black,line width = 1pt]  (B\i) circle (4pt); } 
\end{tikzpicture}\end{array}
-
\begin{array}{c}\begin{tikzpicture}[xscale=.45,yscale=.45,line width=1.0pt] 
\foreach \i in {1,2,3}  { \path (\i,1.25) coordinate (T\i); \path (\i,.25) coordinate (B\i); } 
\filldraw[fill= black!12,draw=black!12,line width=4pt]  (T1) -- (T3) -- (B3) -- (B1) -- (T1);
\draw[blue] (T1) -- (B1)--(B2);
\draw[blue] (T1)--(T2) -- (B2);
\draw[blue] (T3) -- (B3);
\foreach \i in {1,2,3}  { \filldraw[fill=black,draw=black,line width = 1pt] (T\i) circle (4pt); \filldraw[fill=black,draw=black,line width = 1pt]  (B\i) circle (4pt); } 
\end{tikzpicture}\end{array}
-
\begin{array}{c}\begin{tikzpicture}[xscale=.45,yscale=.45,line width=1.0pt] 
\foreach \i in {1,2,3}  { \path (\i,1.25) coordinate (T\i); \path (\i,.25) coordinate (B\i); } 
\filldraw[fill= black!12,draw=black!12,line width=4pt]  (T1) -- (T3) -- (B3) -- (B1) -- (T1);
\draw[blue] (T1) -- (B1);
\draw[blue] (T2) -- (B2);
\draw[blue] (T3) -- (B3);
\draw[blue] (T1) .. controls +(0,-.50) and +(0,-.50) .. (T3);
\draw[blue] (B1) .. controls +(0,.50) and +(0,.50) .. (B3);
\foreach \i in {1,2,3}  { \filldraw[fill=black,draw=black,line width = 1pt] (T\i) circle (4pt); \filldraw[fill=black,draw=black,line width = 1pt]  (B\i) circle (4pt); } 
\end{tikzpicture}\end{array}
-
\begin{array}{c}\begin{tikzpicture}[xscale=.45,yscale=.45,line width=1.0pt] 
\foreach \i in {1,2,3}  { \path (\i,1.25) coordinate (T\i); \path (\i,.25) coordinate (B\i); } 
\filldraw[fill= black!12,draw=black!12,line width=4pt]  (T1) -- (T3) -- (B3) -- (B1) -- (T1);
\draw[blue] (T1) -- (B1);
\draw[blue] (T2) -- (B2)--(B3);
\draw[blue] (T2)--(T3) -- (B3);
\foreach \i in {1,2,3}  { \filldraw[fill=black,draw=black,line width = 1pt] (T\i) circle (4pt); \filldraw[fill=black,draw=black,line width = 1pt]  (B\i) circle (4pt); } 
\end{tikzpicture}\end{array}
+2
\begin{array}{c}\begin{tikzpicture}[xscale=.45,yscale=.45,line width=1.0pt] 
\foreach \i in {1,2,3}  { \path (\i,1.25) coordinate (T\i); \path (\i,.25) coordinate (B\i); } 
\filldraw[fill= black!12,draw=black!12,line width=4pt]  (T1) -- (T3) -- (B3) -- (B1) -- (T1);
\draw[blue] (T1) -- (B1) -- (B2);
\draw[blue] (T1)--(T2) -- (B2) -- (B3);
\draw[blue] (T2)--(T3) -- (B3);
\foreach \i in {1,2,3}  { \filldraw[fill=black,draw=black,line width = 1pt] (T\i) circle (4pt); \filldraw[fill=black,draw=black,line width = 1pt]  (B\i) circle (4pt); } 
\end{tikzpicture}\end{array}.
\end{align*}
}
\end{rem}  

\subsection{Multiplication in the orbit basis}\label{subsec:mult}

In this section, we describe the rule for multiplying in the orbit basis of the partition algebra $\P_k(n)$
using  the following conventions:

For $\ell,m \in \ZZ_{\ge 0}$,  let
\begin{equation}
(m)_\ell = m(m-1) \cdots (m-\ell+1).
\end{equation} Thus,  $(m)_\ell = 0$ \, if $\ell >m $, \  $(m)_0 = 1$, \,and\, $(m)_\ell =  m!/(m-\ell)!$\, if $m \ge \ell$.

When $\pi \in \Pi_{2k}$,  then $\pi$ induces a set partition on the bottom row $\{1, 2, \ldots, k\}$ and a set partition on the top row $\{k+1, k+2, \ldots, 2k\}$. If 
$\pi_1, \pi_2 \in \Pi_{2k}$, then we say that $\pi_1 \ast \pi_2$ \emph{exactly matches in the middle} if the set partition that $\pi_1$ induces on its bottom row equals the set partition that $\pi_2$ induces on the top row modulo $k$.  When that happens,  $\pi_1 \ast \pi_2$ is the
concatenation of the two diagrams.   For example, if $k = 4$ then $\pi_1 = \{ 1, 4, 5\,|\, 2, 8 \,|\, 3, 6, 7\}$ induces the set partition $\{1,4 \,|\, 2 \,|\, 3\}$ on the bottom row of $\pi_1$, and $\pi_2 = \{ 1,5,8 \,|\, 2,6\,|\, 3\,|\, 4,7 \}$ induces the set partition $\{5,8\,|\, 6\,|\, 7\} \equiv \{1,4 \,|\, 2 \,|\, 3\}\, \mathsf{mod}\, 4$ on the top row of $\pi_2$. Thus, $\pi_1 \ast \pi_2$ exactly matches in the middle. This definition is easy to see in terms of the diagrams, as the examples below, particularly Example 
 \ref{Example:TwoBlocksRemoved},  demonstrate. 

In the next result, we describe the formula for multiplying two orbit basis diagrams.  This formula
was originally stated by Halverson and Ram in unpublished notes and was proven in \cite[Cor.~4.12]{BH}.   In the product expression
below,  $x_{\varrho} = 0$ whenever $\varrho$ has more than $n$ blocks.   Recall that $[\pi_1 \ast \pi_2]$ is the number of
blocks in the middle row of $\pi_1 \ast \pi_2$.  

\begin{thm}\label{T:mult}  Multiplication in $\P_k(n)$  in terms of the orbit basis $\{ x_\pi\}_{\pi \in \Pi_{2k}(n)}$ is given by
$$x_{\pi_1}x_{\pi_2} = 
\begin{cases}
\displaystyle{\sum_{\vr} (n-|\vr|)_{[\pi_1 \ast \pi_2]} \, x_\varrho,} &\quad \hbox{if $\pi_1 \ast \pi_2$ exactly matches in the middle,}\\
0 &\quad \hbox{otherwise,} 
\end{cases}
$$
where  the sum is over all coarsenings $\vr$ of $\pi_1 \ast \pi_2$ obtained by connecting blocks that lie entirely in the top row of $\pi_1$
 to blocks that lie entirely in the bottom row of $\pi_2$.     \end{thm}
\begin{examp}\label{ex:32}{\rm  Suppose $k=3$,  $n \ge 2$, and $\pi = \{1,2,3 \,|\,  4,5,6\}\in \Pi_6$. Then,
according to Theorem \ref{T:mult},  
$$\begin{array}{c}
\begin{tikzpicture}[scale=.6,line width=1.25pt] 
\foreach \i in {1,...,3} 
{ \path (\i,1) coordinate (T\i); \path (\i,0) coordinate (B\i); } 
\filldraw[fill= gray!50,draw=gray!50,line width=4pt]  (T1) -- (T3) -- (B3) -- (B1) -- (T1);
\draw[blue] (B1) .. controls +(0,+.3) and +(0,.3) .. (B2).. controls +(0,+.3) and +(0,.3) .. (B3);
\draw[blue] (T1) .. controls +(0,-.3) and +(0,-.3) .. (T2) .. controls +(0,-.3) and +(0,-.3) .. (T3);
\foreach \i in {1,...,3} 
{  \filldraw[fill=white,draw=black,line width = 1pt]  (T\i) circle (4pt);  \filldraw[fill=white,draw=black,line width = 1pt] (B\i) circle (4pt); } 
\end{tikzpicture} \\
\begin{tikzpicture}[scale=.6,line width=1.25pt] 
\foreach \i in {1,...,3} 
{ \path (\i,1) coordinate (T\i); \path (\i,0) coordinate (B\i); } 
\filldraw[fill= gray!50,draw=gray!50,line width=4pt]  (T1) -- (T3) -- (B3) -- (B1) -- (T1);
\draw[blue] (B1) .. controls +(0,+.3) and +(0,.3) .. (B2).. controls +(0,+.3) and +(0,.3) .. (B3);
\draw[blue] (T1) .. controls +(0,-.3) and +(0,-.3) .. (T2) .. controls +(0,-.3) and +(0,-.3) .. (T3);
\foreach \i in {1,...,3} 
{ \filldraw[fill=white,draw=black,line width = 1pt] (T\i) circle (4pt); \filldraw[fill=white,draw=black,line width = 1pt] (B\i) circle (4pt); } 
\end{tikzpicture}
\end{array} = (n-2)
\begin{array}{c}
\begin{tikzpicture}[scale=.6,line width=1.25pt] 
\foreach \i in {1,...,3} 
{ \path (\i,1) coordinate (T\i); \path (\i,0) coordinate (B\i); } 
\filldraw[fill= gray!50,draw=gray!50,line width=4pt]  (T1) -- (T3) -- (B3) -- (B1) -- (T1);
\draw[blue] (B1) .. controls +(0,+.3) and +(0,.3) .. (B2).. controls +(0,+.3) and +(0,.3) .. (B3);
\draw[blue] (T1) .. controls +(0,-.3) and +(0,-.3) .. (T2) .. controls +(0,-.3) and +(0,-.3) .. (T3);
\foreach \i in {1,...,3} 
{ \filldraw[fill=white,draw=black,line width = 1pt] (T\i) circle (4pt);  \filldraw[fill=white,draw=black,line width = 1pt]  (B\i) circle (4pt); } 
\end{tikzpicture}\end{array} + (n-1)
\begin{array}{c}
\begin{tikzpicture}[scale=.6,line width=1.25pt] 
\foreach \i in {1,...,3} 
{ \path (\i,1) coordinate (T\i); \path (\i,0) coordinate (B\i); } 
\filldraw[fill= gray!50,draw=gray!50,line width=4pt]  (T1) -- (T3) -- (B3) -- (B1) -- (T1);
\draw[blue] (B1) .. controls +(0,+.30) and +(0,.30) .. (B2).. controls +(0,+.30) and +(0,.30) .. (B3);
\draw[blue] (T1) .. controls +(0,-.30) and +(0,-.30) .. (T2) .. controls +(0,-.30) and +(0,-.30) .. (T3);
\draw[orange,line width=1.5pt] (T1) -- (B1);
\foreach \i in {1,...,3} 
{ \filldraw[fill=white,draw=black,line width = 1pt]  (T\i) circle (4pt); \filldraw[fill=white,draw=black,line width = 1pt]  (B\i) circle (4pt); } 
\end{tikzpicture}\end{array},
\qquad\hbox{$n \ge 2$}.
$$
}
\end{examp} 

 \begin{examp}\label{Example:TwoBlocksRemoved}{\rm Here $k = 4$, $n \ge 5$, and $[\pi_1 \ast \pi_2] = 2$ (two blocks are removed upon concatenation
of  ${\pi_1}$ and ${\pi_2}$).  
$$
\begin{array}{l}
\begin{array}{c}
\begin{tikzpicture}[scale=.6,line width=1.25pt] 
\foreach \i in {1,...,4} 
{ \path (\i,1) coordinate (T\i); \path (\i,0) coordinate (B\i); } 
\filldraw[fill= gray!50,draw=gray!50,line width=4pt]  (T1) -- (T4) -- (B4) -- (B1) -- (T1);
\draw[blue] (T2) .. controls +(0,-.50) and +(0,-.50) .. (T4);
\draw[blue] (B2) .. controls +(0,+.50) and +(0,+.50) .. (B4);
\draw[blue] (T1) -- (B3);
\foreach \i in {1,...,4} 
{\filldraw[fill=white,draw=black,line width = 1pt] (T\i) circle (4pt); \filldraw[fill=white,draw=black,line width = 1pt](B\i) circle (4pt); } 
\end{tikzpicture} \\
\begin{tikzpicture}[scale=.6,line width=1.25pt] 
\foreach \i in {1,...,4} 
{ \path (\i,1) coordinate (T\i); \path (\i,0) coordinate (B\i); } 
\filldraw[fill= gray!50,draw=gray!50,line width=4pt]  (T1) -- (T4) -- (B4) -- (B1) -- (T1);
\draw[blue] (T3) -- (B2);
\draw[blue] (T2) .. controls +(0,-.50) and +(0,-.50) .. (T4);
\draw[blue] (B1) .. controls +(0,+.40) and +(0,+.40) .. (B3);
\foreach \i in {1,...,4} 
{\filldraw[fill=white,draw=black,line width = 1pt](T\i) circle (4pt);\filldraw[fill=white,draw=black,line width = 1pt](B\i) circle (4pt); } 
\end{tikzpicture}
\end{array} =  (n-5)(n-6)
\begin{array}{c}
\begin{tikzpicture}[scale=.6,line width=1.25pt] 
\foreach \i in {1,...,4} 
{ \path (\i,1) coordinate (T\i); \path (\i,0) coordinate (B\i); } 
\filldraw[fill= gray!50,draw=gray!50,line width=4pt]  (T1) -- (T4) -- (B4) -- (B1) -- (T1);
\draw[blue] (T2) .. controls +(0,-.50) and +(0,-.50) .. (T4);
\draw[blue] (B1) .. controls +(0,+.40) and +(0,+.40) .. (B3);
\draw[blue] (T1) -- (B2);
\foreach \i in {1,...,4} 
{\filldraw[fill=white,draw=black,line width = 1pt] (T\i) circle (4pt); \filldraw[fill=white,draw=black,line width = 1pt] (B\i) circle (4pt); } 
\end{tikzpicture}
\end{array}  \\
\hskip.10in+ (n-4)(n-5) \Big(
\begin{array}{c}
\begin{tikzpicture}[scale=.6,line width=1.25pt] 
\foreach \i in {1,...,4} 
{ \path (\i,1) coordinate (T\i); \path (\i,0) coordinate (B\i); } 
\filldraw[fill= gray!50,draw=gray!50,line width=4pt]  (T1) -- (T4) -- (B4) -- (B1) -- (T1);
\draw[blue] (T2) .. controls +(0,-.50) and +(0,-.50) .. (T4);
\draw[blue] (B1) .. controls +(0,+.40) and +(0,+.40) .. (B3);
\draw[blue] (T1) -- (B2);
\draw[orange,line width=1.5pt] (T2)--(B1);
\foreach \i in {1,...,4} 
{\filldraw[fill=white,draw=black,line width = 1pt] (T\i) circle (4pt); \filldraw[fill=white,draw=black,line width = 1pt] (B\i) circle (4pt); } 
\end{tikzpicture}
\end{array}+
\begin{array}{c}
\begin{tikzpicture}[scale=.6,line width=1.25pt] 
\foreach \i in {1,...,4} 
{ \path (\i,1) coordinate (T\i); \path (\i,0) coordinate (B\i); } 
\filldraw[fill= gray!50,draw=gray!50,line width=4pt]  (T1) -- (T4) -- (B4) -- (B1) -- (T1);
\draw[blue] (T2) .. controls +(0,-.50) and +(0,-.50) .. (T4);
\draw[blue] (B1) .. controls +(0,+.40) and +(0,+.40) .. (B3);
\draw[blue] (T1) -- (B2);
\draw[orange,line width=1.5pt] (T4)--(B4);
\foreach \i in {1,...,4} 
{\filldraw[fill=white,draw=black,line width = 1pt](T\i) circle (4pt); \filldraw[fill=white,draw=black,line width = 1pt](B\i) circle (4pt); } 
\end{tikzpicture}
\end{array}
+
\begin{array}{c}
\begin{tikzpicture}[scale=.6,line width=1.25pt] 
\foreach \i in {1,...,4} 
{ \path (\i,1) coordinate (T\i); \path (\i,0) coordinate (B\i); } 
\filldraw[fill= gray!50,draw=gray!50,line width=4pt]  (T1) -- (T4) -- (B4) -- (B1) -- (T1);
\draw[blue] (T2) .. controls +(0,-.50) and +(0,-.50) .. (T4);
\draw[blue] (B1) .. controls +(0,+.40) and +(0,+.40) .. (B3);
\draw[blue] (T1) -- (B2);
\draw[orange,line width=1.5pt] (T3).. controls +(0,-.30) and +(0,+.30) ..(B3); 
\foreach \i in {1,...,4} 
{ \filldraw[fill=white,draw=black,line width = 1pt] (T\i) circle (4pt); \filldraw[fill=white,draw=black,line width = 1pt](B\i) circle (4pt); } 
\end{tikzpicture}
\end{array}+
\begin{array}{c}
\begin{tikzpicture}[scale=.6,line width=1.25pt] 
\foreach \i in {1,...,4} 
{ \path (\i,1) coordinate (T\i); \path (\i,0) coordinate (B\i); } 
\filldraw[fill= gray!50,draw=gray!50,line width=4pt]  (T1) -- (T4) -- (B4) -- (B1) -- (T1);
\draw[blue] (T2) .. controls +(0,-.50) and +(0,-.50) .. (T4);
\draw[blue] (B1) .. controls +(0,+.40) and +(0,+.40) .. (B3);
\draw[blue] (T1) -- (B2);
\draw[orange,line width=1.5pt] (T3).. controls +(0,-.30) and +(0,+.30) ..(B4);
\foreach \i in {1,...,4} 
{\filldraw[fill=white,draw=black,line width = 1pt](T\i) circle (4pt); \filldraw[fill=white,draw=black,line width = 1pt](B\i) circle (4pt); } 
\end{tikzpicture}
\end{array}\Big) \\
\hskip.10in+(n-3)(n-4) \left(
\begin{array}{c}
\begin{tikzpicture}[scale=.6,line width=1.25pt] 
\foreach \i in {1,...,4} 
{ \path (\i,1) coordinate (T\i); \path (\i,0) coordinate (B\i); } 
\filldraw[fill= gray!50,draw=gray!50,line width=4pt]  (T1) -- (T4) -- (B4) -- (B1) -- (T1);
\draw[blue] (T2) .. controls +(0,-.50) and +(0,-.50) .. (T4);
\draw[blue] (B1) .. controls +(0,+.40) and +(0,+.40) .. (B3);
\draw[blue] (T1) -- (B2);
\draw[orange,line width=1.5pt] (T2)--(B1);
\draw[orange,line width=1.5pt] (T3).. controls +(0,-.30) and +(0,+.30) ..(B4);
\foreach \i in {1,...,4} 
{\filldraw[fill=white,draw=black,line width = 1pt](T\i) circle (4pt); \filldraw[fill=white,draw=black,line width = 1pt] (B\i) circle (4pt); } 
\end{tikzpicture}
\end{array}+
\begin{array}{c}
\begin{tikzpicture}[scale=.6,line width=1.25pt] 
\foreach \i in {1,...,4} 
{ \path (\i,1) coordinate (T\i); \path (\i,0) coordinate (B\i); } 
\filldraw[fill= gray!50,draw=gray!50,line width=4pt]  (T1) -- (T4) -- (B4) -- (B1) -- (T1);
\draw[blue] (T2) .. controls +(0,-.50) and +(0,-.50) .. (T4);
\draw[blue] (B1) .. controls +(0,+.40) and +(0,+.40) .. (B3);
\draw[blue] (T1) -- (B2);
\draw[orange,line width=1.5pt] (T4)--(B4);
\draw[orange,line width=1.5pt] (T3).. controls +(0,-.30) and +(0,+.30) ..(B3);
\foreach \i in {1,...,4} 
{\filldraw[fill=white,draw=black,line width = 1pt] (T\i) circle (4pt); \filldraw[fill=white,draw=black,line width = 1pt] (B\i) circle (4pt); } 
\end{tikzpicture}
\end{array}\right).
\end{array}
$$
}
\end{examp}   

\subsection{Characters  of partition algebras}\label{subsec:chars} 

Let $\gamma_r$ be the $r$-cycle $(1,2,\dots, r)$ in $\S_r \subseteq \P_r(n)$. For a partition $\mu  = [\mu_1, \mu_2, \ldots, \mu_t] \vdash \ell$, with $\mu_t > 0$ and $0 \le \ell \le k$, define 
$\gamma_\mu = \gamma_{\mu_1} \otimes \gamma_{\mu_2} \otimes \cdots \otimes \gamma_{\mu_t} \otimes (\emptydiagram)^{\otimes (k-\ell)} \in \P_k(n)$,  where ``$\otimes$'' here stands for juxtaposing diagrams from smaller partition algebras.  For example, if $\mu = [5,3,2,2]$ and $k = 15$, then 
$$
\gamma_{[5,3,2,2]} = 
\begin{array}{c}
\begin{tikzpicture}[scale=.6,line width=1.25pt] 
\foreach \i in {1,...,15} 
{ \path (\i,1) coordinate (T\i); \path (\i,0) coordinate (B\i); } 
\filldraw[fill= gray!50,draw=gray!50,line width=4pt]  (T1) -- (T15) -- (B15) -- (B1) -- (T1);
\draw[blue] (B1) -- (T2);
\draw[blue] (B2) -- (T3);
\draw[blue] (B3) -- (T4);
\draw[blue] (B4) -- (T5);
\draw[blue] (T1).. controls +(0,-.50) and +(0,+.50) ..(B5);
\draw[blue] (B6) -- (T7);
\draw[blue] (B7) -- (T8);
\draw[blue] (T6).. controls +(0,-.50) and +(0,+.50) ..(B8);
\draw[blue] (B9) -- (T10);
\draw[blue] (T9) -- (B10);
\draw[blue] (B11) -- (T12);
\draw[blue] (T11) -- (B12);
\foreach \i in {1,...,15} 
{\filldraw[fill=black,draw=black,line width = 1pt] (T\i) circle (3pt); \filldraw[fill=black,draw=black,line width = 1pt] (B\i) circle (3pt); } 
\end{tikzpicture}
\end{array}.
$$
In \cite[Sec.~2.2]{H},  it is shown that characters of the partition algebra $\P_k(n)$ are completely determined by their values on the elements of the set $\{\gamma_{\mu} \mid \mu \vdash \ell, \ 0 \le \ell \le k\}$, and thus,  the $\gamma_\mu$ are analogous to conjugacy class representatives in a group. 

For any integer $m \in \ZZ_{\ge 1}$,  let $\mathsf{F}_m(\sigma) = \mathsf{F}(\sigma^m)$, the fixed points of $\sigma^m$ for  $\sigma \in \S_n$.
Then for the partition $\mu=[\mu_1, \mu_2, \ldots, \mu_t] \vdash \ell$ above,  set 
$$\mathsf{F}_\mu(\sigma) = \mathsf{F}_{\mu_1}(\sigma) \cdots \mathsf{F}_{\mu_t}(\sigma).$$
Applying the duality between $\S_n$ and $\P_k(n)$, Halverson \cite[Thm.~3.22]{H} showed that the character value of
$\sigma \times \gamma_\mu$ on $\M_n^{\ot k}$ is $n^{k-\ell} \mathsf{F}_\mu(\sigma)$.    

The conjugacy classes of $\S_n$ are indexed by the
 partitions  $\delta \vdash n$ which correspond to the cycle type of a  permutation.  The number of permutations of cycle type $\delta$ is
$n!/z_\delta$,  where  $z_\delta = 1^{d_1} 2^{d_1}\, \cdots\, n^{d_n} \,d_1! \cdots d_n!\,$  when $\delta$ has $d_i$ parts equal to $i$.  The fixed points of a permutation
depend only on its cycle type,  so $\mathsf{F}_\mu$ is a class function, and we let $\mathsf{F}_\mu(\delta)$  be the value of
$\mathsf{F}_\mu$ on the conjugacy class labeled by $\delta$. Similarly,  we let $\chi_\lambda(\delta)$ denote the value of the irreducible 
 $\S_n$-character $\chi_\lambda$, $\lambda \vdash n$, on the class labeled by $\delta$.Then, applying \eqref{eq:SWchar},  we have the following (compare \cite[Cor.~3.25]{H}):   

\begin{thm}\label{T:partitionchars} Assume $n \ge 2k$.  For $\lambda \in \Lambda_{k,\S_n}$, and $\mu = [\mu_1,\mu_2,\dots,\mu_t] \vdash \ell$ with $0 \leq \ell \leq k$,  
the value of the  irreducible character $\xi_\lambda$ for $\P_k(n)$ on $\gamma_\mu$ is given by 
\begin{align*} 
\xi_{\lambda}(\gamma_\mu) =  \frac{n^{k-\ell}}{n!} \sum_{\sigma \in \S_n} \mathsf{F}_\mu(\sigma)\chi_\lambda(\sigma^{-1})
& =  \frac{n^{k-\ell}}{n!} \sum_{\sigma \in \S_n} \mathsf{F}_\mu(\sigma)\chi_\lambda(\sigma) 
=   n^{k-\ell}  \sum_{\delta \vdash n} \ \frac{1}{z_\delta}\, \mathsf{F}_\mu(\delta)\chi_\lambda(\delta). 
\end{align*}
\end{thm}    

\begin{rem}{\rm In the special case that  $\mu = [1^k]$  (the partition of $k$ with all parts equal to 1), then 
$\mathsf{F}_\mu(\sigma) = \mathsf{F}_1(\sigma)^k
= \mathsf{F}(\sigma)^k$, and 
$\gamma_\mu = \mathsf{I}_k$, the identity element of $\P_k(n)$.  The result in Theorem \ref{T:partitionchars}  reduces to

$$\dim \Z_{k,n}^\lambda = \xi_\lambda(\mathsf{I}_k) = \frac{1}{n!} \sum_{\sigma \in \S_n} \mathsf{F}(\sigma)^k \chi_\lambda(\sigma),$$
which is exactly the expression in Theorem \ref{T:cent}\,(a)
for the dimension of  the irreducible $\Z_{k,n}$-module $\Z_{k,n}^\lambda$ indexed by $\lambda$, since we can identify $\P_k(n)$ with
the centralizer algebra $\Z_{k,n} = \End_{\S_n}(\M_n^{\ot k})$ when
$n \ge 2k$. } \end{rem}

\section{The Representation $\Phi_{k,n}: \P_k(n) \rightarrow \End_{\S_n}(\M_n^{\ot k})$ and its Kernel}
\label{sec:tensorrepresentation}

Jones \cite{J} defined an action of the partition algebra $\P_k(n)$ on  the tensor space $\M_n^{\otimes k}$  that commutes with the diagonal action of the symmetric group $\S_n$ on that same space  and  showed that this action affords a representation of $\P_k(n)$ onto the centralizer algebra $\End_{\S_n}(\M_n^{\otimes k})$. In this section, we describe the action of each diagram basis element $d_\pi$ and each orbit basis element $x_\pi$  of $\P_k(n)$ on $\M_n^{\otimes k}$, and we use the orbit basis to describe the image and the kernel of this action.  
 
\subsection{The orbit basis of $\End_{\G}(\M_n^{\ot k})$  for $\G$ a subgroup of $\S_n$}\label{subsec:orbitbasisaction} 

Assume $k,n \in \ZZ_{\ge 1}$, and let $\{\vf_j \mid  1\le j \le n\}$ be the basis for the permutation module $\modu$ of $\S_n$.   The elements $\vf_\rs = \vf_{r_1} \ot \cdots \ot \vf_{r_k}$ for $\rs = (r_1,\dots, r_k) \in [1,n]^k = \{1,2,\dots, n\}^k$ form a basis for the $\S_n$-module $\modu^{\ot k}$ with $\S_n$ acting diagonally,  
 $\sigma. \vf_\rs = \vf_{\sigma(\rs)} := \vf_{\sigma(r_1)} \ot \cdots \ot \vf_{\sigma(r_n)},$
as in the Introduction. 

Suppose $\varphi =  \sum_{\rs, \sff \in [1,n]^k} \varphi_{\rs}^{\sff}\, \EE_{\rs}^\sff  \in \End(\modu^{\ot k})$ 
where $\{\EE_{\rs}^\sff\}$ is a basis for $\End(\modu^{\ot k})$ of matrix units, and the 
coefficients $\varphi_{\rs}^{\sff}$ belong to $\FF$. 
Then  $\EE_{\rs}^\sff \vf_{\tf} = \delta_{\rs,\tf} \vf_{\sff}$,  with $\delta_{\rs,\tf}$ being the Kronecker delta, 
and
the action of $\varphi$ on the basis of simple tensors is given by
\begin{equation}
\varphi (\v_\rs) = \sum_{\sff \in [1,n]^k}  \varphi_{\rs}^{\sff} \v_{\sff}.
\end{equation}

For  any subgroup $\GG \subseteq \S_n$ (in particular,  for $\S_n$ itself) and for the centralizer algebra 
$\End_{\GG}(\modu^{\ot k}) = \{ \varphi \in \End(\modu^{\ot k}) \mid \varphi \sigma = \sigma \varphi$ for all $\sigma \in \GG$\}, we have
\begin{align*}
\begin{split}
 \varphi \in \End_{\GG}(\modu^{\ot k}) & \iff   \sigma \varphi  = \varphi \sigma  \ \ \hbox{\rm for all} \ \ \sigma \in \GG   \\
& \iff  \sum_{\sff \in [1,n]^k}  \varphi_{\rs}^{\sff} \, \vf_{\sigma(\sff)} = \sum_{\sff \in [1,n]^k}  \varphi_{\sigma(\rs)}^{\sff}\, \vf_{\sff}  \quad\hbox{\rm for all} \ \ \rs \in [1,n]^k  
\end{split}\end{align*}
and so
\begin{equation} \label{eq:centcond}
 \varphi \in \End_{\GG}(\modu^{\ot k})  \iff  \varphi_{\rs}^{\sff}  = \varphi_{\sigma(\rs)}^{\sigma(\sff)} \quad \hbox{\rm for all} \ \ \rs, \sff  \in [1,n]^k,  \sigma \in \GG. \hskip.7in
\end{equation} 
It is convenient to view the pair of $k$-tuples $\rs, \sff \in [1,n]^k$ in \eqref{eq:centcond} as a single $2k$-tuple $(\rs,\sff) \in [1,n]^{2k}$.
In this notation, condition \eqref{eq:centcond} tells us that the elements of $\End_\G(\modu^{\ot k})$ are in one-to-one correspondence with the $\G$-orbits on $[1,n]^{2k}$, 
where $\sigma \in \G$ acts on $(r_1, \ldots, r_{2k}) \in [1,n]^{2k}$ by $\sigma(r_1, \ldots, r_{2k}) = (\sigma(r_1), \ldots, \sigma(r_{2k}))$.

We adopt the shorthand notation  $(\rs|\rs') = (r_1,\dots, r_{2k}) \in [1,n]^{2k}$ when $\rs = (r_1,\dots,r_k) \in [1,n]^k$ and $\rs' = (r_{k+1},\dots, r_{2k}) \in [1,n]^k$.
Let the $\G$-orbit of $(\rs|\rs') \in [1,n]^{2k}$ be denoted by $\G(\rs|\rs') = \{ \sigma(\rs|\rs') \mid \sigma \in \G\}$ and define
\begin{equation}\label{eq:orbitindicator}
\mathsf{X}_{(\rs|\rs')}  = \sum_{(\sff|\sff') \in \G (\rs|\rs')} \mathsf{E}_{\sff}^{\sff'},
\end{equation}
where the sum is over the distinct elements in the orbit.   This is the indicator  function of the orbit $\G(\rs|\rs')$, and it satisfies \eqref{eq:centcond},  so  $\mathsf{X}_{(\rs|\rs')} \in \End_{\G}(\M_n^{\ot k})$. 
Let $[1,n]^{2k}/\G$ be the set consisting of one $2k$-tuple  $(\rs|\rs')$ for each   $\G$-orbit.
Since \eqref{eq:centcond} is a necessary and sufficient condition for a transformation to belong to $\End_{\G}(\M_n^{\ot k})$, and the
elements $\mathsf{X}_{(\rs|\rs')}$ for $(\rs|\rs') \in [1,n]^{2k}/\G$ are linearly independent,  we have the following result.

\begin{thm}\label{thm:Generalorbitbasis} For $\G \subseteq \S_n$ and $n, k \in \ZZ_{\ge 1}$,   the centralizer algebra $\End_\G(\M_n^{\ot k})$ has a basis 
$\{\mathsf{X}_{(\rs|\rs')} \mid (\rs|\rs') \in [1,n]^{2k}/\G\}$. 
In particular, $\dim(\End_\G(\M_n^{\ot k}))$ equals the number of $\G$-orbits on $[1,n]^{2k}$. 
\end{thm} 

When $\G = \S_n$, then since $\S_n$ acts transitively on $[1,n]$, the  $\S_n$-orbits on $[1,n]^{2k}$ correspond to set partitions of $[1,2k]$ into \emph{at most $n$ blocks}. In particular, if $\pi \in \Pi_{2k}$ is a set partition of $[1,2k]$, then 
\begin{equation}\label{eq:OrbitPermutationCondition}
\left\{(r_1,r_2, \ldots, r_{2k}) \mid r_a = r_b\ \iff\ a, b \text{ are in the same block of } \pi\right\}
\end{equation}
is the $\S_n$-orbit corresponding to $\pi$ (compare this with \eqref{eq:PermutationCondition}).  The condition \eqref{eq:OrbitPermutationCondition} requires there to be $n$ or fewer blocks in $\pi$; otherwise there are not enough distinct values $r_a \in [1,n]$ to assign to each of the blocks.

If $\G \subset \S_n$ is a proper subgroup, then the $\S_n$-orbits may split into smaller  $\G$-orbits. For example, if $k=2$ and $\G = \mathsf{A}_4 \subseteq \S_4$, the alternating subgroup, then the  $\S_4$-orbit corresponding to $\pi = \{1 \mid 2 \mid 3, 4\}$ contains 
$(1,2,3,3)$ and $(1,2,4,4)$, but no element of $\mathsf{A}_4$ sends $(1,2,3,3)$ to $(1,2,4,4)$, since it would have to fix 1 and 2 and swap 3 and 4.

When $\G = \S_n$, the orbit basis has an especially nice form, since the  $\S_n$-orbits of $[1,n]^{2k}$ correspond to set partitions of $[1,2k]$ having at most $n$ parts.  For $\pi \in \Pi_{2k}$, we designate a special labeling associated to $\pi$ as follows. 

\begin{definition}\label{D:std}   Let $\mathsf{B}_1$ be the block of $\pi$ containing 1, and for $1< j \le |\pi|$, let $\mathsf{B}_j$ be the block of $\pi$ containing the smallest number not in $\mathsf{B}_1 \cup \mathsf{B}_2 \cup \cdots \cup \mathsf{B}_{j-1}$.   
The \emph{standard labeling of $\pi$}  is  $(\bs_\pi | \bs_{\pi}')$,  where 
$\bs_\pi = (\textrm{b}_1,\dots, \textrm{b}_k)$ and $\bs_\pi' = (\textrm{b}_{k+1}, \dots,\textrm{b}_{2k})$ in $[1,n]^k$,  and 
\begin{align}\begin{split}\label{eq:beta}  
&\textrm{b}_\ell = j \ \ \hbox{\rm if} \ \  \ell \in \mathsf{B}_j\ \ \hbox{\rm for} \ \  \ell \in [1,2k].  \end{split} \end{align}  
\end{definition}
Rather than writing $\mathsf{X}_{(\bs_\pi | \bs_{\pi}')}$ for the $\S_n$-orbit element  determined by  $(\bs_\pi | \bs_{\pi}') \in [1,n]^{2k}$, we denote this simply
as $\mathsf{X}_\pi$.     Then it follows that 

\begin{equation} \mathsf{X}_\pi  =  \sum_{(\sff|\sff')\, \in \, \S_n(\bs_\pi | \bs_{\pi}')} \mathsf{E}_{\sff}^{\sff'} =
 \sum_{(\rs|\rs') \in [1,n]^{2k}} (\mathsf{X}_\pi)_\rs^{\rs'} \EE_{\rs}^{\rs'}, \end{equation}
where
 \begin{equation}
(\mathsf{X}_\pi)_{\rs}^{\rs'} = 
  \begin{cases}  1  & \quad \text{if $r_a = r_b$ \emph{if and only if} $a$ and $b$ are in the same block of $\pi$,}  \\
 0  & \quad \text{otherwise.}
  \end{cases} 
  \end{equation} 
 In this case, Theorem \ref{thm:Generalorbitbasis} specializes to the following.

\begin{thm}\label{thm:orbitbasis} For  $k,n \in \ZZ_{\ge 1}$, $\End_{\S_n}(\M_n^{\ot k})$ has a basis $\{ \mathsf{X}_\pi \mid \pi \in \Pi_{2k,n}\}$,  and therefore $$\dim(\End_{\S_n}(\M_n^{\ot k})) = \B(2k,n).$$
\end{thm}

\medskip

\begin{rem}{\rm   If $\G$ is any  finite group and  $\M$ is any permutation module for $\G$ (that is, $g \in \G$ acts as a permutation on a distinguished basis of $n$ elements of $\M$), then $\G$ can be viewed as a subgroup of $\S_n$  and $\M$ can be regarded as the module $\M_n$. Thus, the method of this section applies to tensor powers of any permutation module for any group $\G$. For example,  an action of  a group $\G$ on a finite set $\{ x_1, \ldots, x_n\}$, can be viewed as an action of the subgroup $\G$ of $\S_n$ on the permutation module $\M_n = \text{span}_\FF\{x_1,x_2, \ldots, x_n\}$.
}
\end{rem}

\subsection{The definition of $\Phi_{k,n}$}\label{subsec:repn}

For $k,n \in \ZZ_{\ge 1}$, define  $\Phi_{k,n}: \P_k(n) \rightarrow \End(\modu^{\ot k})$ by 
\begin{equation} \label{eq:Phi-orbit}
\Phi_{k,n}(x_\pi) = 
\begin{cases}
\mathsf{X}_\pi & \text{if $\pi$ has $n$ or fewer blocks, and}\\
0 & \text{if $\pi$ has more than $n$ blocks.}
\end{cases} 
\end{equation}
As $\{x_\pi \mid \pi \in \Pi_{2k}\}$ is a basis for $\P_k(n)$, we can extend $\Phi_{k,n}$ linearly to get a transformation,  $\Phi_{k,n}: \P_k(n) \rightarrow \End(\modu^{\ot k})$. 
It follows from Theorem \ref{thm:orbitbasis} that $\Phi_{k,n}$ maps $\P_k(n)$ surjectively onto $\End_{\S_n}(\modu^{\ot k})$ for all $k,n \in \ZZ_{\ge 1}$.  When $n \ge 2k$, we have $\dimm(\P_k(n)) = \mathsf{B}(2k) = \dim(\End_{\S_n}(\M_n^{\ot k}))$, and thus $\Phi_{k,n}$ is an isomorphism.  

From \eqref{eq:Phi-orbit} we see that
 \begin{equation}\label{eq:Phi-coeffs-orbit}
  \Phi_{k,n}(x_\pi)_{\rs}^{\rs'} = 
  \begin{cases}  1  & \quad \text{if $r_a = r_b$ \emph{if and only if} $a$ and $b$ are in the same block of $\pi$,}  \\
 0  & \quad \text{otherwise}.
  \end{cases} 
  \end{equation} 
  Since the diagram basis  $\{d_\pi \mid \pi \in \Pi_{2k}\}$ is related to the orbit basis $\{x_\pi \mid \pi \in \Pi_{2k}\}$ by 
 the refinement relation \eqref{refinement-relation}, we have as an immediate consequence, 
 \begin{equation}\label{eq:Phi-coeffs-diagram}
  \Phi_{k,n}(d_\pi)_{\rs}^{\rs'} = 
  \begin{cases}  1  & \quad \text{if $r_a = r_b$ \emph{when} $a$ and $b$ are in the same block of $\pi$,}  \\
 0  & \quad \text{otherwise}.
  \end{cases} 
  \end{equation} 
  \smallskip
   The map $\Phi_{k,n}$ can be shown to be an algebra homomorphism (the diagram basis works especially
   well for doing that,  see  \cite[Prop.~3.6]{BH}), and so 
  $\Phi_{k,n}$  affords a representation of the partition algebra $\P_k(n)$.   
 
\begin{examp}{\rm (Example \ref{ex:32} revisited) \    Recall that in this example  $\pi = \{123\,|\, 456\}$
and $\varrho = \{1,2,3,4,5,6\mid\}$ are elements of $\Pi_{6}$.   Then, by definition,  
$\Phi_{3,n}(x_\pi) = \sum_{i\neq j \in [1,n]}  \EE_{iii}^{jjj}$ for all $n \ge 2$ so that  
\begin{align*} \Phi_{3,n}(x_\pi^2) =\left(\Phi_{3,n}(x_\pi)\right)^2  &
=  (n-2)\sum_{i \neq j \in [1,n]}  \EE_{iii}^{jjj} +
  (n-1)\sum_{i \in [1,n]}  \EE_{iii}^{iii} \\
& = (n-2)\Phi_{3,n}(x_\pi)+ (n-1)\Phi_{3,n}(x_\varrho) 
 = \Phi_{3,n}\big((n-2)x_\pi + (n-1)x_\varrho\big).\
 \end{align*} 
Such examples inspired the product rule 
in Theorem \ref{T:mult}.
}   \end{examp} 
   
We identify  $\S_{n-1}$ with the subgroup of $\S_n$ of permutations
that fix $n$ and make the identification $\modu^{\ot k} \cong \modu^{\ot k} \ot \vf_n  \subseteq \modu^{\ot (k+1)}$, so that $\modu^{\ot k}$ is a submodule for both $\S_{n-1}$ and $\P_{k+\frac{1}{2}}(n) \subset \P_{k+1}(n)$. 
Then for tuples $\tilde \rs, \tilde \sff \in [1,n]^{k+1}$  having  $r_{k+1} = n = s_{k+1}$,  condition \eqref{eq:centcond} for $\GG= \S_{n-1}$ becomes
$$
\varphi_{\tilde \rs}^{\tilde \sff} = \varphi_{\sigma(\tilde \rs)}^{\sigma(\tilde \sff)} \quad 
\hbox{\rm for all} \ \ \tilde \rs, \tilde \sff  \in [1,n]^{k+1},  \sigma \in \S_{n-1}.
$$
Thus, the matrix units for $\GG=\S_{n-1}$ in  \eqref{eq:orbitindicator} correspond to set partitions in  $\Pi_{2k+1}$; that is, set partitions of $\{1,2,\ldots,2(k+1)\}$ having $k+1$ and $2(k+1)$ in the same block. 

 Let $\Phi_{k+\half,n}: \P_{k+\half}(n)  \rightarrow \End(\modu^{\ot k} \ot \mathsf{v}_n)$ be defined by
$$\Phi_{k+\half,n}(x_\pi) = \sum_{(\tilde\rs|\tilde\rs') \in [1,n]^{2(k+1)}} (X_\pi)_{\tilde \rs}^{\tilde \rs'} \EE_{\tilde \rs}^{\tilde \rs'}$$
where the sum is over tuples of the form $\tilde \rs = (r_1,\dots,r_k,n), \tilde \rs' = (r_{k+1},\dots,r_{2k},n)$ in $[1,n]^{k+1}$, and   
 \begin{equation}\label{eq:Phi2-coeffs-orbit}
(\mathsf{X}_\pi)_{\tilde \rs}^{\tilde \rs'} = 
  \begin{cases}  1  & \quad \text{if $r_a = r_b$ \emph{if and only if} $a$ and $b$ are in the same block of $\pi$,}  \\
 0  & \quad \text{otherwise.}
  \end{cases} 
  \end{equation}  
 Then the argument
proving  \cite[Prop.~3.6]{BH} can be easily adapted to show that 
$\Phi_{k+\half,n}$  is 
a representation of $\P_{k+\half}(n)$.
  
  The next theorem  describes a basis for the image and the kernel of $\Phi_{k,n}$. Part (a) follows from our work in Section \ref{subsec:orbitbasisaction} and is originally due to Jones \cite{J}. The extension
  to $\Phi_{k+\half,n}$ can be found in \cite[Thm.~3.6]{HR}.
 \medskip
 
\begin{thm}\label{T:Phi}  Assume  $n \in \mathbb Z_{\ge 1}$  and   $\{x_\pi \mid \pi \in \Pi_{2k}\}$ is the orbit basis
for $\P_k(n)$.   
\begin{itemize}
\item[{\rm (a)}]  For $k \in \ZZ_{\ge 1}$, the representation  $\Phi_{k,n}: \P_k(n) \rightarrow \End(\modu^{\ot k})$ has
\begin{align*} \im \Phi_{k,n} &= \End_{\S_n}(\modu^{\ot k}) = \spann_{\CC}\{\Phi_{k,n}(x_\pi) \mid \pi \in \Pi_{2k} \ \text{has $\le n$ blocks}\}\ \ \hbox{\rm and}  \\ 
\ker \Phi_{k,n} &= \spann_{\CC}\{x_\pi \mid \pi \in \Pi_{2k} \ \text{has more than $n$ blocks}\}.\end{align*} 
Consequently, $\End_{\S_n}(\modu^{\ot k})$ is isomorphic to $\P_k(n)$ for $n \geq 2k$.  
\item[{\rm (b)}] For $k \in \ZZ_{\ge 0}$,  the representation   $\Phi_{k+\half}: \P_{k+\half}(n) \rightarrow \End(\modu^{\ot k})$ has
\begin{align*} \im \Phi_{k+\half,n} &= \End_{\S_{n-1}}(\modu^{\ot k}) = \spann_{\CC}\{\Phi_{k+\half,n}(x_\pi) \mid \pi \in \Pi_{2k+1} \ \text{has $\le n$ blocks}\} \ \ \hbox{\rm and} \\
\ker \Phi_{k+\half,n} &= \spann_{\CC}\{x_\pi \mid \pi \in \Pi_{2k+1} \ \text{has more than $n$ blocks}\}.\end{align*}
Consequently, $\End_{\S_{n-1}}(\modu^{\ot k})$ is isomorphic to $ \P_{k+\half}(n)$ for $n \geq 2k+1$.   \end{itemize}
\end{thm}

\begin{rem}{\rm The assertion that the map $\Phi_{k,n}$ (resp. $\Phi_{k+\half,n}$) is an isomorphism when $n \geq 2k$ (resp. when $n \geq 2k+1$) holds because set partitions  $\pi \in \Pi_{2k}$ (resp. $\pi \in \Pi_{2k+1}$) have no more
than $n$ blocks under those assumptions. 
}
\end{rem}       

\begin{examp} {\rm When $k = 2$ and $n = 2$, the image of $\Phi_{2,2}: \P_2(2) \to \End_{\S_2}(\M_2^{\otimes 2})$ is spanned by the  images of the following 8 diagrams,
$$
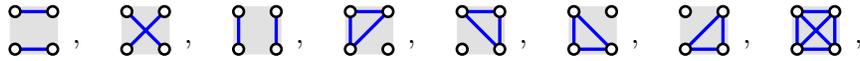
\begin{array}{cccccccc}
\begin{array}{c}\orbitablcd\end{array}, &
\begin{array}{c}\orbitadlbc\end{array}, & \begin{array}{c}\orbitaclbd\end{array}, & \begin{array}{c}\orbitacdlb\end{array}, & \begin{array}{c}\orbitalbcd\end{array},  & 
\begin{array}{c}\orbitabcld\end{array}, & \begin{array}{c}\orbitabdlc\end{array}, & \begin{array}{c}\orbitabcd\end{array} ,
\end{array}
$$ 
and the kernel is spanned by the following 7 diagrams,
$$
\begin{array}{ccccccc}
\begin{array}{c}\orbitalblcld\end{array},  &  \begin{array}{c}\orbitablcld\end{array},  & \begin{array}{c}\orbitaclbld\end{array}, &  \begin{array}{c}\orbitadlblc\end{array}, & \begin{array}{c}\orbitalbcld\end{array},  & \begin{array}{c}\orbitalbdlc\end{array}, & \begin{array}{c}\orbitalblcd\end{array}.
\end{array}
$$ 
}
\end{examp}

\begin{figure}
$$
\begin{array}{|l|rrrrrrr|r|}
\hline
k & \B(2k,2) & \B(2k,3) & \B(2k,4) & \B(2k,5) & \B(2k,6) & \B(2k,7) & \B(2k,8) &  \B(2k) \\
\hline
\half &\phantom{\Big\vert} 1 &  1 &  1 &  1 &  1 &  1 &  1 &  1 \\
1 &\phantom{\big\vert} 2 &  2 &  2 &  2 &  2 &  2 &  2 &  2 \\
1\half &\phantom{\Big\vert} 4 &  5 &  5 &  5 &  5 &  5 &  5 &  5 \\
2 &\phantom{\big\vert} 8 &  14 &  15 &  15 &  15 &  15 &  15 &  15 \\
2\half &\phantom{\Big\vert} 16 &  41 &  51 &  52 &  52 &  52 &  52 &  52 \\
3 &\phantom{\big\vert} 32 &  122 &  187 &  202 &  203 &  203 &  203 &  203 \\
3\half &\phantom{\Big\vert} 64 &  365 &  715 &  855 &  876 &  877 &  877 &  877 \\
4 &\phantom{\big\vert} 128 &  1094 &  2795 &  3845 &  4111 &  4139 &  4140 &  4140 \\
4\half &\phantom{\Big\vert} 256 &  3281 &  11051 &  18002 &  20648 &  21110 &  21146 &  21147 \\
5 &\phantom{\big\vert} 512 &  9842 &  43947 &  86472 &  109299 &  115179 &  115929 &  115975 \\
5\half &\phantom{\Big\vert} 1024 &  29525 &  175275 &  422005 &  601492 &  665479 &  677359 &  678570 \\
6 &\phantom{\big\vert} 2048 &  88574 &  700075 &  2079475 &  3403127 &  4030523 &  4189550 &  4213597 \\
\hline
\end{array}
$$
\caption{Table of values of the restricted Bell number $\B(2k,n)$, which equals the dimension of the image of the surjection $\Phi_{k,n}: \P_k(n) \to \End_{\S_n}(\M_n^{\ot k})$. The rightmost column gives $\dimm(\P_k(n)) = \B(2k)$, the $2k$-th (unrestricted) Bell number. Compare column $\B(2k,5)$ to the rightmost column of Figure \ref{fig:Sbratteli}.
\label{table:BellNumbers}}
\end{figure} 

 \begin{rem} {\rm  Recall from Corollary \ref{C:fixedBell}, that $\B(\ell,n) = (n\,!)^{-1} \sum_{\sigma \in \S_n} \mathsf{F}(\sigma)^{\ell}$ for all $\ell \in \ZZ_{\ge 0}$.  When $n = 2$, only the
identity element of $\S_2$ has fixed points, and we see that $\B(2k,2) = \half (2^{2k}) = 2^{2k-1}$ for all $k \in \half \ZZ_{\ge 0}$,  in agreement with the values in the first column of Figure \ref{table:BellNumbers}. When $n =3$,  there are three transpositions in $\S_3$, each having one fixed point, and two cycles of length 3 that have no fixed points.   Thus,
$$\B(2k,3) = \frac{1}{6} \Big( 3^{2k} + 3 \cdot 1^{2k} + 2 \cdot 0^{2k}\Big)= \frac{3^{2k-1}+1}{2}  \quad \text{for} \ \  k \in \half \ZZ_{\ge 0}.$$
These values correspond to the numbers in the second column of the table.} \end{rem}

\subsection{The  kernel of  the surjection $\Phi_{k,n}: \P_k(n) \to \End_{\S_n}(\modu^{\ot k})$ when $2k > n$}
\label{subsec:kernels}

This section is devoted to a description of the kernel of the map $\Phi_{k,n}$ (and also of $\Phi_{k-\half,n}$) when $2k > n$.  Towards this purpose, the 
following orbit basis elements $\ef_{k,n}$ for $k,n \in \ZZ_{\ge 1}$ and  $2k > n$ were introduced
in \cite[Sec.~5.3]{BH}: 
\vspace{-.1cm}
\begin{align}\label{eq:ef}
\ef_{k,n}  &= \begin{cases}
\!\!\begin{array}{c} 
\begin{tikzpicture}[xscale=.5,yscale=.5,line width=1.25pt] 
\foreach \i in {1,2,3,4,5,6,7}  { \path (\i,1.25) coordinate (T\i); \path (\i,.25) coordinate (B\i); } 
\filldraw[fill= black!12,draw=black!12,line width=4pt]  (T1) -- (T7) -- (B7) -- (B1) -- (T1);
\draw[blue] (T4) -- (B4);
\draw[blue] (T5) -- (B5);
\draw[blue] (T7) -- (B7);
\draw (T2) node {$\cdots$};\draw (B2) node {$\cdots$};
\draw (T6) node {$\cdots$};\draw (B6) node {$\cdots$};
\draw (2,-.5) node {$\underbrace{\phantom{iiiiiiiiiiii}}_{n+1-k}$};
\draw (5.5,-.5) node {$\underbrace{\phantom{iiiiiiiiiiiiiiiii}}_{2k-n-1}$};
\foreach \i in {1,3,4,5,7}  { \filldraw[fill=white,draw=black,line width = 1pt] (T\i) circle (4pt); \filldraw[fill=white,draw=black,line width = 1pt]  (B\i) circle (4pt); } 
\end{tikzpicture}
\end{array} & \quad \text{if \ $n \ge k > n/2$,}\\ 
\begin{array}{c} 
\begin{tikzpicture}[xscale=.5,yscale=.5,line width=1.25pt] 
\foreach \i in {1,2,3,4,5,6,7}  { \path (\i,1.25) coordinate (T\i); \path (\i,.25) coordinate (B\i); } 
\filldraw[fill= black!12,draw=black!12,line width=4pt]  (T1) -- (T7) -- (B7) -- (B1) -- (T1);
\draw[blue] (T1) -- (B1);
\draw[blue] (T2) -- (B2);
\draw[blue] (T3) -- (B3);
\draw[blue] (T4) -- (B4);
\draw (T5) node {$\cdots$};\draw (B5) node {$\cdots$};
\draw (T6) node {$\cdots$};\draw (B6) node {$\cdots$};
\draw[blue] (T7) -- (B7);
\draw (4,-.5) node {$\underbrace{\phantom{iiiiiiiiiiiiiiiiiiiiiiiiiiiiiiiii}}_{k}$};
\foreach \i in {1,2,3,4,7}  { \filldraw[fill=white,draw=black,line width = 1pt] (T\i) circle (4pt); \filldraw[fill=white,draw=black,line width = 1pt]  (B\i) circle (4pt); } 
\end{tikzpicture} 
\end{array}&\quad  \text{if \ $k >n$.}
\end{cases}\\
\ef_{k-\half,n}  &= \ef_{k,n},  \hskip.2in  \text{if \  $ 2k - 1 >  n$}. \label{eq:efhalf} 
\end{align}
Observe that if $n \ge k > n/2$, then the number of
blocks in $\ef_{k,n}$ is  $|\ef_{k,n}| =  2(n+1-k) + 2k-n-1 = n+1$,  so $\ef_{k,n}$ is in the kernel of $\Phi_{k,n}$.
For example, 
\vspace{-.17cm} 
\begin{center}{$
\ef_{4\half,6} = \ef_{5,6} =   \begin{array}{c} 
\begin{tikzpicture}[xscale=.5,yscale=.5,line width=1.25pt] 
\foreach \i in {1,2,3,4,5}  { \path (\i,1.25) coordinate (T\i); \path (\i,.25) coordinate (B\i); } 
\filldraw[fill= black!12,draw=black!12,line width=4pt]  (T1) -- (T5) -- (B5) -- (B1) -- (T1);
\draw[blue] (T3) -- (B3);
\draw[blue] (T4) -- (B4);
\draw[blue] (T5) -- (B5);
\foreach \i in {1,2,3,4,5}  { \filldraw[fill=white,draw=black,line width = 1pt] (T\i) circle (4pt); \filldraw[fill=white,draw=black,line width = 1pt]  (B\i) circle (4pt); } 
\end{tikzpicture} 
\end{array}
$} \end{center}
 has $|\ef_{5,6}|  = 7$ blocks. The elements $\ef_{k,n}$  for $k \le 5$ and $n \le 9$ are displayed
in Figure \ref{fig:etable}.

\smallskip
\begin{thm} \label{thm:generator}{\rm (\cite[Thms.~5.6 and 5.9]{BH})}  Assume $k,n\in \ZZ_{\ge 1}$  and $2k > n$. 
\begin{itemize} \item[{\rm (a)}]  The orbit basis element $\ef_{k,n}$ in \eqref{eq:ef} 
is an essential idempotent such that  $(\ef_{k,n})^2 = c_{k,n} \ef_{k,n}$, where 
\begin{equation}\label{eq:ckn}  c_{k,n} = \begin{cases} 
(-1)^{n+1-k}\,  (n+1-k)! & \quad \text{if} \ \  n \ge k> n/2,  \\
1 &  \quad \text{if} \ \ k > n.\end{cases} \end{equation}   
\item[{\rm(b)}]  The kernel of the representation $ \Phi_{k,n}$ is the ideal of $\P_k(n)$ generated
by $\ef_{k,n}$ when $2k>n$.
\item [{\rm (c)}]  The kernel of the representation $\Phi_{k-\half,n}$ is the ideal of $\P_{k-\half}(n)$
generated by $\ef_{k-\half,n} = \ef_{k,n}$ when $2k > n+1$.  \end{itemize}
\end{thm}

\smallskip
\begin{rem}{\rm  For a fixed value of $n$,  the first time the kernel is nonzero is when $k = \half(n+1)$ (i.e. when $n=2k-1$).  This is 
the first entry in each row in  the table in Figure \ref{fig:etable}.  For that particular value of $n$,
$(\ef_{k,2k-1})^2  = (-1)^k/k! \ \ef_{k,2k-1}$ by Theorem \ref{thm:generator}\,(a).
}\end{rem}

The expression for $\ef_{k,n}$ in the diagram basis is given by 
\begin{equation}  
\ef_{k,n}  =  \sum_{\vr \in \Pi_{2k},\ \pi_{k,n} \preceq \vr}  \mu_{2k}(\pi_{k,n},\vr)  d_{\vr},
\end{equation}
where $\pi_{k,n}$ is the set partition of $[1,2k]$ corresponding to $\ef_{k,n}$.  
When $k = \half(n+1)$ and $n$ is odd,  $\ef_{k,n} =   \begin{array}{c}
\begin{tikzpicture}[xscale=.3,yscale=.3,line width=1.25pt] 
\foreach \i in {1,2,3} 
{ \path (\i,.4) coordinate (T\i); \path (\i,-.4) coordinate (B\i); } 
\filldraw[fill= black!12,draw=black!12,line width=4pt]  (T1) -- (T3) -- (B3) -- (B1) -- (T1);
\foreach \i in {1,3} 
{ \filldraw[fill=white,draw=black,line width = 1pt] (T\i) circle (5pt); \filldraw[fill=white,draw=black,line width = 1pt]  (B\i) circle (5pt); } 
\draw (T2) node {$\cdots$};
\draw (B2) node {$\cdots$};
\end{tikzpicture}
\end{array},$ 
and all 
 $\vr$ in  $\Pi_{2k}$ occur in the expression for $\ef_{k,n}$.  
 When $k = \half(n+1)$ and $n$ is even,  $\ef_{k,n} =   \begin{array}{c}
\begin{tikzpicture}[xscale=.3,yscale=.3,line width=1.25pt] 
\foreach \i in {1,2,3,4} 
{ \path (\i,.4) coordinate (T\i); \path (\i,-.4) coordinate (B\i); } 
\filldraw[fill= black!12,draw=black!12,line width=4pt]  (T1) -- (T4) -- (B4) -- (B1) -- (T1);
\draw[blue] (T4) -- (B4);
\foreach \i in {1,3,4} 
{ \filldraw[fill=white,draw=black,line width = 1pt] (T\i) circle (5pt); \filldraw[fill=white,draw=black,line width = 1pt]  (B\i) circle (5pt); } 
\draw (T2) node {$\cdots$};
\draw (B2) node {$\cdots$};
\end{tikzpicture}
\end{array}$, 
and all 
 $\vr$ in  $\Pi_{2k-1}$ occur in the expression for $\ef_{k,n}$.   
If $k > n$,  then $\ef_{k,n} =   \begin{array}{c}
\begin{tikzpicture}[xscale=.3,yscale=.3,line width=1.25pt] 
\foreach \i in {1,2,3} 
{ \path (\i,.4) coordinate (T\i); \path (\i,-.4) coordinate (B\i); } 
\filldraw[fill= black!12,draw=black!12,line width=4pt]  (T1) -- (T3) -- (B3) -- (B1) -- (T1);
\draw[blue] (T1) -- (B1);\draw[blue] (T3) -- (B3);
\foreach \i in {1,3} 
{ \filldraw[fill=white,draw=black,line width = 1pt] (T\i) circle (5pt); \filldraw[fill=white,draw=black,line width = 1pt]  (B\i) circle (5pt); } 
\draw (T2) node {$\cdots$};
\draw (B2) node {$\cdots$};
\end{tikzpicture}\end{array},$ 
and the 
 set partitions $\vr \in\Pi_{2k}$  that occur in the expression for $\ef_{k,n}$ correspond to the coarsenings of the $k$ columns of the 
 diagram  $\pi_{k,n}$. 
 There are Bell number $\B(k)$ such terms.
In each case, the integer coefficients $\mu_{2k}(\pi_{k,n},\vr)$ 
can be computed using \eqref{eq:mobiusb}.

\begin{figure}
\noindent $$\begin{array}{c|ccccccccc}
\ef_{k,n}& k=1&k = 1 \frac{1}{2}&k=2&k = 2 \frac{1}{2}&k = 3&k = 3\frac{1}{2}&k = 4& k = 4\frac{1}{2}& k = 5\\
\hline
n = 1 & 
\vphantom{\bigg\vert}
\begin{array}{c} 
\begin{tikzpicture}[xscale=.3,yscale=.3,line width=1.25pt] 
\foreach \i in {1}  { \path (\i,1.25) coordinate (T\i); \path (\i,.25) coordinate (B\i); } 
\foreach \i in {1}  { \filldraw[fill=white,draw=black,line width = 1pt] (T\i) circle (4pt); \filldraw[fill=white,draw=black,line width = 1pt]  (B\i) circle (4pt); } 
\end{tikzpicture} 
\end{array}
&
\begin{array}{c} 
\begin{tikzpicture}[xscale=.3,yscale=.3,line width=1.25pt] 
\foreach \i in {1,2}  { \path (\i,1.25) coordinate (T\i); \path (\i,.25) coordinate (B\i); } 
\filldraw[fill= black!12,draw=black!12,line width=4pt]  (T1) -- (T2) -- (B2) -- (B1) -- (T1);
\draw[blue] (T1) -- (B1);
\draw[blue] (T2) -- (B2);
\foreach \i in {1,2}  { \filldraw[fill=white,draw=black,line width = 1pt] (T\i) circle (4pt); \filldraw[fill=white,draw=black,line width = 1pt]  (B\i) circle (4pt); } 
\end{tikzpicture} 
\end{array}
&
\begin{array}{c} 
\begin{tikzpicture}[xscale=.3,yscale=.3,line width=1.25pt] 
\foreach \i in {1,2}  { \path (\i,1.25) coordinate (T\i); \path (\i,.25) coordinate (B\i); } 
\filldraw[fill= black!12,draw=black!12,line width=4pt]  (T1) -- (T2) -- (B2) -- (B1) -- (T1);
\draw[blue] (T1) -- (B1);
\draw[blue] (T2) -- (B2);
\foreach \i in {1,2}  { \filldraw[fill=white,draw=black,line width = 1pt] (T\i) circle (4pt); \filldraw[fill=white,draw=black,line width = 1pt]  (B\i) circle (4pt); } 
\end{tikzpicture} 
\end{array}
&
\begin{array}{c} 
\begin{tikzpicture}[xscale=.3,yscale=.3,line width=1.25pt] 
\foreach \i in {1,2,3}  { \path (\i,1.25) coordinate (T\i); \path (\i,.25) coordinate (B\i); } 
\filldraw[fill= black!12,draw=black!12,line width=4pt]  (T1) -- (T3) -- (B3) -- (B1) -- (T1);
\draw[blue] (T1) -- (B1);
\draw[blue] (T2) -- (B2);
\draw[blue] (T3) -- (B3);
\foreach \i in {1,2,3}  { \filldraw[fill=white,draw=black,line width = 1pt] (T\i) circle (4pt); \filldraw[fill=white,draw=black,line width = 1pt]  (B\i) circle (4pt); } 
\end{tikzpicture} 
\end{array}
&
\begin{array}{c} 
\begin{tikzpicture}[xscale=.3,yscale=.3,line width=1.25pt] 
\foreach \i in {1,2,3}  { \path (\i,1.25) coordinate (T\i); \path (\i,.25) coordinate (B\i); } 
\filldraw[fill= black!12,draw=black!12,line width=4pt]  (T1) -- (T3) -- (B3) -- (B1) -- (T1);
\draw[blue] (T1) -- (B1);
\draw[blue] (T2) -- (B2);
\draw[blue] (T3) -- (B3);
\foreach \i in {1,2,3}  { \filldraw[fill=white,draw=black,line width = 1pt] (T\i) circle (4pt); \filldraw[fill=white,draw=black,line width = 1pt]  (B\i) circle (4pt); } 
\end{tikzpicture} 
\end{array}
&
\begin{array}{c} 
\begin{tikzpicture}[xscale=.3,yscale=.3,line width=1.25pt] 
\foreach \i in {1,2,3,4}  { \path (\i,1.25) coordinate (T\i); \path (\i,.25) coordinate (B\i); } 
\filldraw[fill= black!12,draw=black!12,line width=4pt]  (T1) -- (T4) -- (B4) -- (B1) -- (T1);
\draw[blue] (T1) -- (B1);
\draw[blue] (T2) -- (B2);
\draw[blue] (T3) -- (B3);
\draw[blue] (T4) -- (B4);
\foreach \i in {1,2,3,4}  { \filldraw[fill=white,draw=black,line width = 1pt] (T\i) circle (4pt); \filldraw[fill=white,draw=black,line width = 1pt]  (B\i) circle (4pt); } 
\end{tikzpicture} 
\end{array}
&
\begin{array}{c} 
\begin{tikzpicture}[xscale=.3,yscale=.3,line width=1.25pt] 
\foreach \i in {1,2,3,4}  { \path (\i,1.25) coordinate (T\i); \path (\i,.25) coordinate (B\i); } 
\filldraw[fill= black!12,draw=black!12,line width=4pt]  (T1) -- (T4) -- (B4) -- (B1) -- (T1);
\draw[blue] (T1) -- (B1);
\draw[blue] (T2) -- (B2);
\draw[blue] (T3) -- (B3);
\draw[blue] (T4) -- (B4);
\foreach \i in {1,2,3,4}  { \filldraw[fill=white,draw=black,line width = 1pt] (T\i) circle (4pt); \filldraw[fill=white,draw=black,line width = 1pt]  (B\i) circle (4pt); } 
\end{tikzpicture} 
\end{array}
&
\begin{array}{c} 
\begin{tikzpicture}[xscale=.3,yscale=.3,line width=1.25pt] 
\foreach \i in {1,2,3,4,5}  { \path (\i,1.25) coordinate (T\i); \path (\i,.25) coordinate (B\i); } 
\filldraw[fill= black!12,draw=black!12,line width=4pt]  (T1) -- (T5) -- (B5) -- (B1) -- (T1);
\draw[blue] (T1) -- (B1);
\draw[blue] (T2) -- (B2);
\draw[blue] (T3) -- (B3);
\draw[blue] (T4) -- (B4);
\draw[blue] (T5) -- (B5);
\foreach \i in {1,2,3,4,5}  { \filldraw[fill=white,draw=black,line width = 1pt] (T\i) circle (4pt); \filldraw[fill=white,draw=black,line width = 1pt]  (B\i) circle (4pt); } 
\end{tikzpicture} 
\end{array}
&
\begin{array}{c} 
\begin{tikzpicture}[xscale=.3,yscale=.3,line width=1.25pt] 
\foreach \i in {1,2,3,4,5}  { \path (\i,1.25) coordinate (T\i); \path (\i,.25) coordinate (B\i); } 
\filldraw[fill= black!12,draw=black!12,line width=4pt]  (T1) -- (T5) -- (B5) -- (B1) -- (T1);
\draw[blue] (T1) -- (B1);
\draw[blue] (T2) -- (B2);
\draw[blue] (T3) -- (B3);
\draw[blue] (T4) -- (B4);
\draw[blue] (T5) -- (B5);
\foreach \i in {1,2,3,4,5}  { \filldraw[fill=white,draw=black,line width = 1pt] (T\i) circle (4pt); \filldraw[fill=white,draw=black,line width = 1pt]  (B\i) circle (4pt); } 
\end{tikzpicture} 
\end{array}
\\
n = 2 
& 
\vphantom{\bigg\vert}
&
\begin{array}{c} 
\begin{tikzpicture}[xscale=.3,yscale=.3,line width=1.25pt] 
\foreach \i in {1,2}  { \path (\i,1.25) coordinate (T\i); \path (\i,.25) coordinate (B\i); } 
\filldraw[fill= black!12,draw=black!12,line width=4pt]  (T1) -- (T2) -- (B2) -- (B1) -- (T1);
\draw[blue] (T2) -- (B2);
\foreach \i in {1,2}  { \filldraw[fill=white,draw=black,line width = 1pt] (T\i) circle (4pt); \filldraw[fill=white,draw=black,line width = 1pt]  (B\i) circle (4pt); } 
\end{tikzpicture} 
\end{array}
&
\begin{array}{c} 
\begin{tikzpicture}[xscale=.3,yscale=.3,line width=1.25pt] 
\foreach \i in {1,2}  { \path (\i,1.25) coordinate (T\i); \path (\i,.25) coordinate (B\i); } 
\filldraw[fill= black!12,draw=black!12,line width=4pt]  (T1) -- (T2) -- (B2) -- (B1) -- (T1);
\draw[blue] (T2) -- (B2);
\foreach \i in {1,2}  { \filldraw[fill=white,draw=black,line width = 1pt] (T\i) circle (4pt); \filldraw[fill=white,draw=black,line width = 1pt]  (B\i) circle (4pt); } 
\end{tikzpicture} 
\end{array}
&
\begin{array}{c} 
\begin{tikzpicture}[xscale=.3,yscale=.3,line width=1.25pt] 
\foreach \i in {1,2,3}  { \path (\i,1.25) coordinate (T\i); \path (\i,.25) coordinate (B\i); } 
\filldraw[fill= black!12,draw=black!12,line width=4pt]  (T1) -- (T3) -- (B3) -- (B1) -- (T1);
\draw[blue] (T1) -- (B1);
\draw[blue] (T2) -- (B2);
\draw[blue] (T3) -- (B3);
\foreach \i in {1,2,3}  { \filldraw[fill=white,draw=black,line width = 1pt] (T\i) circle (4pt); \filldraw[fill=white,draw=black,line width = 1pt]  (B\i) circle (4pt); } 
\end{tikzpicture} 
\end{array}
&
\begin{array}{c} 
\begin{tikzpicture}[xscale=.3,yscale=.3,line width=1.25pt] 
\foreach \i in {1,2,3}  { \path (\i,1.25) coordinate (T\i); \path (\i,.25) coordinate (B\i); } 
\filldraw[fill= black!12,draw=black!12,line width=4pt]  (T1) -- (T3) -- (B3) -- (B1) -- (T1);
\draw[blue] (T1) -- (B1);
\draw[blue] (T2) -- (B2);
\draw[blue] (T3) -- (B3);
\foreach \i in {1,2,3}  { \filldraw[fill=white,draw=black,line width = 1pt] (T\i) circle (4pt); \filldraw[fill=white,draw=black,line width = 1pt]  (B\i) circle (4pt); } 
\end{tikzpicture} 
\end{array}
&
\begin{array}{c} 
\begin{tikzpicture}[xscale=.3,yscale=.3,line width=1.25pt] 
\foreach \i in {1,2,3,4}  { \path (\i,1.25) coordinate (T\i); \path (\i,.25) coordinate (B\i); } 
\filldraw[fill= black!12,draw=black!12,line width=4pt]  (T1) -- (T4) -- (B4) -- (B1) -- (T1);
\draw[blue] (T1) -- (B1);
\draw[blue] (T2) -- (B2);
\draw[blue] (T3) -- (B3);
\draw[blue] (T4) -- (B4);
\foreach \i in {1,2,3,4}  { \filldraw[fill=white,draw=black,line width = 1pt] (T\i) circle (4pt); \filldraw[fill=white,draw=black,line width = 1pt]  (B\i) circle (4pt); } 
\end{tikzpicture} 
\end{array}
&
\begin{array}{c} 
\begin{tikzpicture}[xscale=.3,yscale=.3,line width=1.25pt] 
\foreach \i in {1,2,3,4}  { \path (\i,1.25) coordinate (T\i); \path (\i,.25) coordinate (B\i); } 
\filldraw[fill= black!12,draw=black!12,line width=4pt]  (T1) -- (T4) -- (B4) -- (B1) -- (T1);
\draw[blue] (T1) -- (B1);
\draw[blue] (T2) -- (B2);
\draw[blue] (T3) -- (B3);
\draw[blue] (T4) -- (B4);
\foreach \i in {1,2,3,4}  { \filldraw[fill=white,draw=black,line width = 1pt] (T\i) circle (4pt); \filldraw[fill=white,draw=black,line width = 1pt]  (B\i) circle (4pt); } 
\end{tikzpicture} 
\end{array}
&
\begin{array}{c} 
\begin{tikzpicture}[xscale=.3,yscale=.3,line width=1.25pt] 
\foreach \i in {1,2,3,4,5}  { \path (\i,1.25) coordinate (T\i); \path (\i,.25) coordinate (B\i); } 
\filldraw[fill= black!12,draw=black!12,line width=4pt]  (T1) -- (T5) -- (B5) -- (B1) -- (T1);
\draw[blue] (T1) -- (B1);
\draw[blue] (T2) -- (B2);
\draw[blue] (T3) -- (B3);
\draw[blue] (T4) -- (B4);
\draw[blue] (T5) -- (B5);
\foreach \i in {1,2,3,4,5}  { \filldraw[fill=white,draw=black,line width = 1pt] (T\i) circle (4pt); \filldraw[fill=white,draw=black,line width = 1pt]  (B\i) circle (4pt); } 
\end{tikzpicture} 
\end{array}
&
\begin{array}{c} 
\begin{tikzpicture}[xscale=.3,yscale=.3,line width=1.25pt] 
\foreach \i in {1,2,3,4,5}  { \path (\i,1.25) coordinate (T\i); \path (\i,.25) coordinate (B\i); } 
\filldraw[fill= black!12,draw=black!12,line width=4pt]  (T1) -- (T5) -- (B5) -- (B1) -- (T1);
\draw[blue] (T1) -- (B1);
\draw[blue] (T2) -- (B2);
\draw[blue] (T3) -- (B3);
\draw[blue] (T4) -- (B4);
\draw[blue] (T5) -- (B5);
\foreach \i in {1,2,3,4,5}  { \filldraw[fill=white,draw=black,line width = 1pt] (T\i) circle (4pt); \filldraw[fill=white,draw=black,line width = 1pt]  (B\i) circle (4pt); } 
\end{tikzpicture} 
\end{array}
\\
n = 3
& 
\vphantom{\bigg\vert}
&
&
\begin{array}{c} 
\begin{tikzpicture}[xscale=.3,yscale=.3,line width=1.25pt] 
\foreach \i in {1,2}  { \path (\i,1.25) coordinate (T\i); \path (\i,.25) coordinate (B\i); } 
\filldraw[fill= black!12,draw=black!12,line width=4pt]  (T1) -- (T2) -- (B2) -- (B1) -- (T1);
\foreach \i in {1,2}  { \filldraw[fill=white,draw=black,line width = 1pt] (T\i) circle (4pt); \filldraw[fill=white,draw=black,line width = 1pt]  (B\i) circle (4pt); } 
\end{tikzpicture} 
\end{array}
&
\begin{array}{c} 
\begin{tikzpicture}[xscale=.3,yscale=.3,line width=1.25pt] 
\foreach \i in {1,2,3}  { \path (\i,1.25) coordinate (T\i); \path (\i,.25) coordinate (B\i); } 
\filldraw[fill= black!12,draw=black!12,line width=4pt]  (T1) -- (T3) -- (B3) -- (B1) -- (T1);
\draw[blue] (T2) -- (B2);
\draw[blue] (T3) -- (B3);
\foreach \i in {1,2,3}  { \filldraw[fill=white,draw=black,line width = 1pt] (T\i) circle (4pt); \filldraw[fill=white,draw=black,line width = 1pt]  (B\i) circle (4pt); } 
\end{tikzpicture} 
\end{array}
&
\begin{array}{c} 
\begin{tikzpicture}[xscale=.3,yscale=.3,line width=1.25pt] 
\foreach \i in {1,2,3}  { \path (\i,1.25) coordinate (T\i); \path (\i,.25) coordinate (B\i); } 
\filldraw[fill= black!12,draw=black!12,line width=4pt]  (T1) -- (T3) -- (B3) -- (B1) -- (T1);
\draw[blue] (T2) -- (B2);
\draw[blue] (T3) -- (B3);
\foreach \i in {1,2,3}  { \filldraw[fill=white,draw=black,line width = 1pt] (T\i) circle (4pt); \filldraw[fill=white,draw=black,line width = 1pt]  (B\i) circle (4pt); } 
\end{tikzpicture} 
\end{array}
&
\begin{array}{c} 
\begin{tikzpicture}[xscale=.3,yscale=.3,line width=1.25pt] 
\foreach \i in {1,2,3,4}  { \path (\i,1.25) coordinate (T\i); \path (\i,.25) coordinate (B\i); } 
\filldraw[fill= black!12,draw=black!12,line width=4pt]  (T1) -- (T4) -- (B4) -- (B1) -- (T1);
\draw[blue] (T1) -- (B1);
\draw[blue] (T2) -- (B2);
\draw[blue] (T3) -- (B3);
\draw[blue] (T4) -- (B4);
\foreach \i in {1,2,3,4}  { \filldraw[fill=white,draw=black,line width = 1pt] (T\i) circle (4pt); \filldraw[fill=white,draw=black,line width = 1pt]  (B\i) circle (4pt); } 
\end{tikzpicture} 
\end{array}
&
\begin{array}{c} 
\begin{tikzpicture}[xscale=.3,yscale=.3,line width=1.25pt] 
\foreach \i in {1,2,3,4}  { \path (\i,1.25) coordinate (T\i); \path (\i,.25) coordinate (B\i); } 
\filldraw[fill= black!12,draw=black!12,line width=4pt]  (T1) -- (T4) -- (B4) -- (B1) -- (T1);
\draw[blue] (T1) -- (B1);
\draw[blue] (T2) -- (B2);
\draw[blue] (T3) -- (B3);
\draw[blue] (T4) -- (B4);
\foreach \i in {1,2,3,4}  { \filldraw[fill=white,draw=black,line width = 1pt] (T\i) circle (4pt); \filldraw[fill=white,draw=black,line width = 1pt]  (B\i) circle (4pt); } 
\end{tikzpicture} 
\end{array}
&
\begin{array}{c} 
\begin{tikzpicture}[xscale=.3,yscale=.3,line width=1.25pt] 
\foreach \i in {1,2,3,4,5}  { \path (\i,1.25) coordinate (T\i); \path (\i,.25) coordinate (B\i); } 
\filldraw[fill= black!12,draw=black!12,line width=4pt]  (T1) -- (T5) -- (B5) -- (B1) -- (T1);
\draw[blue] (T1) -- (B1);
\draw[blue] (T2) -- (B2);
\draw[blue] (T3) -- (B3);
\draw[blue] (T4) -- (B4);
\draw[blue] (T5) -- (B5);
\foreach \i in {1,2,3,4,5}  { \filldraw[fill=white,draw=black,line width = 1pt] (T\i) circle (4pt); \filldraw[fill=white,draw=black,line width = 1pt]  (B\i) circle (4pt); } 
\end{tikzpicture} 
\end{array}
&
\begin{array}{c} 
\begin{tikzpicture}[xscale=.3,yscale=.3,line width=1.25pt] 
\foreach \i in {1,2,3,4,5}  { \path (\i,1.25) coordinate (T\i); \path (\i,.25) coordinate (B\i); } 
\filldraw[fill= black!12,draw=black!12,line width=4pt]  (T1) -- (T5) -- (B5) -- (B1) -- (T1);
\draw[blue] (T1) -- (B1);
\draw[blue] (T2) -- (B2);
\draw[blue] (T3) -- (B3);
\draw[blue] (T4) -- (B4);
\draw[blue] (T5) -- (B5);
\foreach \i in {1,2,3,4,5}  { \filldraw[fill=white,draw=black,line width = 1pt] (T\i) circle (4pt); \filldraw[fill=white,draw=black,line width = 1pt]  (B\i) circle (4pt); } 
\end{tikzpicture} 
\end{array}
\\
n = 4
& 
\vphantom{\bigg\vert}
&
&
&
\begin{array}{c} 
\begin{tikzpicture}[xscale=.3,yscale=.3,line width=1.25pt] 
\foreach \i in {1,2,3}  { \path (\i,1.25) coordinate (T\i); \path (\i,.25) coordinate (B\i); } 
\filldraw[fill= black!12,draw=black!12,line width=4pt]  (T1) -- (T3) -- (B3) -- (B1) -- (T1);
\draw[blue] (T3) -- (B3);
\foreach \i in {1,2,3}  { \filldraw[fill=white,draw=black,line width = 1pt] (T\i) circle (4pt); \filldraw[fill=white,draw=black,line width = 1pt]  (B\i) circle (4pt); } 
\end{tikzpicture} 
\end{array}
&
\begin{array}{c} 
\begin{tikzpicture}[xscale=.3,yscale=.3,line width=1.25pt] 
\foreach \i in {1,2,3}  { \path (\i,1.25) coordinate (T\i); \path (\i,.25) coordinate (B\i); } 
\filldraw[fill= black!12,draw=black!12,line width=4pt]  (T1) -- (T3) -- (B3) -- (B1) -- (T1);
\draw[blue] (T3) -- (B3);
\foreach \i in {1,2,3}  { \filldraw[fill=white,draw=black,line width = 1pt] (T\i) circle (4pt); \filldraw[fill=white,draw=black,line width = 1pt]  (B\i) circle (4pt); } 
\end{tikzpicture} 
\end{array}
&
\begin{array}{c} 
\begin{tikzpicture}[xscale=.3,yscale=.3,line width=1.25pt] 
\foreach \i in {1,2,3,4}  { \path (\i,1.25) coordinate (T\i); \path (\i,.25) coordinate (B\i); } 
\filldraw[fill= black!12,draw=black!12,line width=4pt]  (T1) -- (T4) -- (B4) -- (B1) -- (T1);
\draw[blue] (T2) -- (B2);
\draw[blue] (T3) -- (B3);
\draw[blue] (T4) -- (B4);
\foreach \i in {1,2,3,4}  { \filldraw[fill=white,draw=black,line width = 1pt] (T\i) circle (4pt); \filldraw[fill=white,draw=black,line width = 1pt]  (B\i) circle (4pt); } 
\end{tikzpicture} 
\end{array}
&
\begin{array}{c} 
\begin{tikzpicture}[xscale=.3,yscale=.3,line width=1.25pt] 
\foreach \i in {1,2,3,4}  { \path (\i,1.25) coordinate (T\i); \path (\i,.25) coordinate (B\i); } 
\filldraw[fill= black!12,draw=black!12,line width=4pt]  (T1) -- (T4) -- (B4) -- (B1) -- (T1);
\draw[blue] (T2) -- (B2);
\draw[blue] (T3) -- (B3);
\draw[blue] (T4) -- (B4);
\foreach \i in {1,2,3,4}  { \filldraw[fill=white,draw=black,line width = 1pt] (T\i) circle (4pt); \filldraw[fill=white,draw=black,line width = 1pt]  (B\i) circle (4pt); } 
\end{tikzpicture} 
\end{array}
&
\begin{array}{c} 
\begin{tikzpicture}[xscale=.3,yscale=.3,line width=1.25pt] 
\foreach \i in {1,2,3,4,5}  { \path (\i,1.25) coordinate (T\i); \path (\i,.25) coordinate (B\i); } 
\filldraw[fill= black!12,draw=black!12,line width=4pt]  (T1) -- (T5) -- (B5) -- (B1) -- (T1);
\draw[blue] (T1) -- (B1);
\draw[blue] (T2) -- (B2);
\draw[blue] (T3) -- (B3);
\draw[blue] (T4) -- (B4);
\draw[blue] (T5) -- (B5);
\foreach \i in {1,2,3,4,5}  { \filldraw[fill=white,draw=black,line width = 1pt] (T\i) circle (4pt); \filldraw[fill=white,draw=black,line width = 1pt]  (B\i) circle (4pt); } 
\end{tikzpicture} 
\end{array}
&
\begin{array}{c} 
\begin{tikzpicture}[xscale=.3,yscale=.3,line width=1.25pt] 
\foreach \i in {1,2,3,4,5}  { \path (\i,1.25) coordinate (T\i); \path (\i,.25) coordinate (B\i); } 
\filldraw[fill= black!12,draw=black!12,line width=4pt]  (T1) -- (T5) -- (B5) -- (B1) -- (T1);
\draw[blue] (T1) -- (B1);
\draw[blue] (T2) -- (B2);
\draw[blue] (T3) -- (B3);
\draw[blue] (T4) -- (B4);
\draw[blue] (T5) -- (B5);
\foreach \i in {1,2,3,4,5}  { \filldraw[fill=white,draw=black,line width = 1pt] (T\i) circle (4pt); \filldraw[fill=white,draw=black,line width = 1pt]  (B\i) circle (4pt); } 
\end{tikzpicture} 
\end{array}
\\
n = 5
& 
\vphantom{\bigg\vert}
&
&
&
&
\begin{array}{c} 
\begin{tikzpicture}[xscale=.3,yscale=.3,line width=1.25pt] 
\foreach \i in {1,2,3}  { \path (\i,1.25) coordinate (T\i); \path (\i,.25) coordinate (B\i); } 
\filldraw[fill= black!12,draw=black!12,line width=4pt]  (T1) -- (T3) -- (B3) -- (B1) -- (T1);
\foreach \i in {1,2,3}  { \filldraw[fill=white,draw=black,line width = 1pt] (T\i) circle (4pt); \filldraw[fill=white,draw=black,line width = 1pt]  (B\i) circle (4pt); } 
\end{tikzpicture} 
\end{array}
&
\begin{array}{c} 
\begin{tikzpicture}[xscale=.3,yscale=.3,line width=1.25pt] 
\foreach \i in {1,2,3,4}  { \path (\i,1.25) coordinate (T\i); \path (\i,.25) coordinate (B\i); } 
\filldraw[fill= black!12,draw=black!12,line width=4pt]  (T1) -- (T4) -- (B4) -- (B1) -- (T1);
\draw[blue] (T3) -- (B3);
\draw[blue] (T4) -- (B4);
\foreach \i in {1,2,3,4}  { \filldraw[fill=white,draw=black,line width = 1pt] (T\i) circle (4pt); \filldraw[fill=white,draw=black,line width = 1pt]  (B\i) circle (4pt); } 
\end{tikzpicture} 
\end{array}
&
\begin{array}{c} 
\begin{tikzpicture}[xscale=.3,yscale=.3,line width=1.25pt] 
\foreach \i in {1,2,3,4}  { \path (\i,1.25) coordinate (T\i); \path (\i,.25) coordinate (B\i); } 
\filldraw[fill= black!12,draw=black!12,line width=4pt]  (T1) -- (T4) -- (B4) -- (B1) -- (T1);
\draw[blue] (T3) -- (B3);
\draw[blue] (T4) -- (B4);
\foreach \i in {1,2,3,4}  { \filldraw[fill=white,draw=black,line width = 1pt] (T\i) circle (4pt); \filldraw[fill=white,draw=black,line width = 1pt]  (B\i) circle (4pt); } 
\end{tikzpicture} 
\end{array}
&
\begin{array}{c} 
\begin{tikzpicture}[xscale=.3,yscale=.3,line width=1.25pt] 
\foreach \i in {1,2,3,4,5}  { \path (\i,1.25) coordinate (T\i); \path (\i,.25) coordinate (B\i); } 
\filldraw[fill= black!12,draw=black!12,line width=4pt]  (T1) -- (T5) -- (B5) -- (B1) -- (T1);
\draw[blue] (T2) -- (B2);
\draw[blue] (T3) -- (B3);
\draw[blue] (T4) -- (B4);
\draw[blue] (T5) -- (B5);
\foreach \i in {1,2,3,4,5}  { \filldraw[fill=white,draw=black,line width = 1pt] (T\i) circle (4pt); \filldraw[fill=white,draw=black,line width = 1pt]  (B\i) circle (4pt); } 
\end{tikzpicture} 
\end{array}
&
\begin{array}{c} 
\begin{tikzpicture}[xscale=.3,yscale=.3,line width=1.25pt] 
\foreach \i in {1,2,3,4,5}  { \path (\i,1.25) coordinate (T\i); \path (\i,.25) coordinate (B\i); } 
\filldraw[fill= black!12,draw=black!12,line width=4pt]  (T1) -- (T5) -- (B5) -- (B1) -- (T1);
\draw[blue] (T2) -- (B2);
\draw[blue] (T3) -- (B3);
\draw[blue] (T4) -- (B4);
\draw[blue] (T5) -- (B5);
\foreach \i in {1,2,3,4,5}  { \filldraw[fill=white,draw=black,line width = 1pt] (T\i) circle (4pt); \filldraw[fill=white,draw=black,line width = 1pt]  (B\i) circle (4pt); } 
\end{tikzpicture} 
\end{array}
\\
n = 6
& 
\vphantom{\bigg\vert}
&
&
&
&
&
\begin{array}{c} 
\begin{tikzpicture}[xscale=.3,yscale=.3,line width=1.25pt] 
\foreach \i in {1,2,3,4}  { \path (\i,1.25) coordinate (T\i); \path (\i,.25) coordinate (B\i); } 
\filldraw[fill= black!12,draw=black!12,line width=4pt]  (T1) -- (T4) -- (B4) -- (B1) -- (T1);
\draw[blue] (T4) -- (B4);
\foreach \i in {1,2,3,4}  { \filldraw[fill=white,draw=black,line width = 1pt] (T\i) circle (4pt); \filldraw[fill=white,draw=black,line width = 1pt]  (B\i) circle (4pt); } 
\end{tikzpicture} 
\end{array}
&
\begin{array}{c} 
\begin{tikzpicture}[xscale=.3,yscale=.3,line width=1.25pt] 
\foreach \i in {1,2,3,4}  { \path (\i,1.25) coordinate (T\i); \path (\i,.25) coordinate (B\i); } 
\filldraw[fill= black!12,draw=black!12,line width=4pt]  (T1) -- (T4) -- (B4) -- (B1) -- (T1);
\draw[blue] (T4) -- (B4);
\foreach \i in {1,2,3,4}  { \filldraw[fill=white,draw=black,line width = 1pt] (T\i) circle (4pt); \filldraw[fill=white,draw=black,line width = 1pt]  (B\i) circle (4pt); } 
\end{tikzpicture} 
\end{array}
&
\begin{array}{c} 
\begin{tikzpicture}[xscale=.3,yscale=.3,line width=1.25pt] 
\foreach \i in {1,2,3,4,5}  { \path (\i,1.25) coordinate (T\i); \path (\i,.25) coordinate (B\i); } 
\filldraw[fill= black!12,draw=black!12,line width=4pt]  (T1) -- (T5) -- (B5) -- (B1) -- (T1);
\draw[blue] (T3) -- (B3);
\draw[blue] (T4) -- (B4);
\draw[blue] (T5) -- (B5);
\foreach \i in {1,2,3,4,5}  { \filldraw[fill=white,draw=black,line width = 1pt] (T\i) circle (4pt); \filldraw[fill=white,draw=black,line width = 1pt]  (B\i) circle (4pt); } 
\end{tikzpicture} 
\end{array}
&
\begin{array}{c} 
\begin{tikzpicture}[xscale=.3,yscale=.3,line width=1.25pt] 
\foreach \i in {1,2,3,4,5}  { \path (\i,1.25) coordinate (T\i); \path (\i,.25) coordinate (B\i); } 
\filldraw[fill= black!12,draw=black!12,line width=4pt]  (T1) -- (T5) -- (B5) -- (B1) -- (T1);
\draw[blue] (T3) -- (B3);
\draw[blue] (T4) -- (B4);
\draw[blue] (T5) -- (B5);
\foreach \i in {1,2,3,4,5}  { \filldraw[fill=white,draw=black,line width = 1pt] (T\i) circle (4pt); \filldraw[fill=white,draw=black,line width = 1pt]  (B\i) circle (4pt); } 
\end{tikzpicture} 
\end{array}
\\
n = 7
& 
\vphantom{\bigg\vert}
&
&
&
&
&
&
\begin{array}{c} 
\begin{tikzpicture}[xscale=.3,yscale=.3,line width=1.25pt] 
\foreach \i in {1,2,3,4}  { \path (\i,1.25) coordinate (T\i); \path (\i,.25) coordinate (B\i); } 
\filldraw[fill= black!12,draw=black!12,line width=4pt]  (T1) -- (T4) -- (B4) -- (B1) -- (T1);
\foreach \i in {1,2,3,4}  { \filldraw[fill=white,draw=black,line width = 1pt] (T\i) circle (4pt); \filldraw[fill=white,draw=black,line width = 1pt]  (B\i) circle (4pt); } 
\end{tikzpicture} 
\end{array}
&
\begin{array}{c} 
\begin{tikzpicture}[xscale=.3,yscale=.3,line width=1.25pt] 
\foreach \i in {1,2,3,4,5}  { \path (\i,1.25) coordinate (T\i); \path (\i,.25) coordinate (B\i); } 
\filldraw[fill= black!12,draw=black!12,line width=4pt]  (T1) -- (T5) -- (B5) -- (B1) -- (T1);
\draw[blue] (T4) -- (B4);
\draw[blue] (T5) -- (B5);
\foreach \i in {1,2,3,4,5}  { \filldraw[fill=white,draw=black,line width = 1pt] (T\i) circle (4pt); \filldraw[fill=white,draw=black,line width = 1pt]  (B\i) circle (4pt); } 
\end{tikzpicture} 
\end{array}
&
\begin{array}{c} 
\begin{tikzpicture}[xscale=.3,yscale=.3,line width=1.25pt] 
\foreach \i in {1,2,3,4,5}  { \path (\i,1.25) coordinate (T\i); \path (\i,.25) coordinate (B\i); } 
\filldraw[fill= black!12,draw=black!12,line width=4pt]  (T1) -- (T5) -- (B5) -- (B1) -- (T1);
\draw[blue] (T4) -- (B4);
\draw[blue] (T5) -- (B5);
\foreach \i in {1,2,3,4,5}  { \filldraw[fill=white,draw=black,line width = 1pt] (T\i) circle (4pt); \filldraw[fill=white,draw=black,line width = 1pt]  (B\i) circle (4pt); } 
\end{tikzpicture} 
\end{array}
\\
n = 8
& 
\vphantom{\bigg\vert}
&
&
&
&
&
&
&
\begin{array}{c} 
\begin{tikzpicture}[xscale=.3,yscale=.3,line width=1.25pt] 
\foreach \i in {1,2,3,4,5}  { \path (\i,1.25) coordinate (T\i); \path (\i,.25) coordinate (B\i); } 
\filldraw[fill= black!12,draw=black!12,line width=4pt]  (T1) -- (T5) -- (B5) -- (B1) -- (T1);
\draw[blue] (T5) -- (B5);
\foreach \i in {1,2,3,4,5}  { \filldraw[fill=white,draw=black,line width = 1pt] (T\i) circle (4pt); \filldraw[fill=white,draw=black,line width = 1pt]  (B\i) circle (4pt); } 
\end{tikzpicture} 
\end{array}
&
\begin{array}{c} 
\begin{tikzpicture}[xscale=.3,yscale=.3,line width=1.25pt] 
\foreach \i in {1,2,3,4,5}  { \path (\i,1.25) coordinate (T\i); \path (\i,.25) coordinate (B\i); } 
\filldraw[fill= black!12,draw=black!12,line width=4pt]  (T1) -- (T5) -- (B5) -- (B1) -- (T1);
\draw[blue] (T5) -- (B5);
\foreach \i in {1,2,3,4,5}  { \filldraw[fill=white,draw=black,line width = 1pt] (T\i) circle (4pt); \filldraw[fill=white,draw=black,line width = 1pt]  (B\i) circle (4pt); } 
\end{tikzpicture} 
\end{array}
\\
n = 9
& 
\vphantom{\bigg\vert}
&
&
&
&
&
&
&
&
\begin{array}{c} 
\begin{tikzpicture}[xscale=.3,yscale=.3,line width=1.25pt] 
\foreach \i in {1,2,3,4,5}  { \path (\i,1.25) coordinate (T\i); \path (\i,.25) coordinate (B\i); } 
\filldraw[fill= black!12,draw=black!12,line width=4pt]  (T1) -- (T5) -- (B5) -- (B1) -- (T1);
\foreach \i in {1,2,3,4,5}  { \filldraw[fill=white,draw=black,line width = 1pt] (T\i) circle (4pt); \filldraw[fill=white,draw=black,line width = 1pt]  (B\i) circle (4pt); } 
\end{tikzpicture} 
\end{array}
\end{array}
$$ 
\caption{The essential idempotent $\ef_{k,n}$ for $k \le 5$ and $n \le 9$. When $n < 2k$ the kernel of $\Phi_{k,n}$ equals the principal ideal $\langle \ef_{k,n} \rangle$. \label{fig:etable}}
\end{figure}

\begin{rem}{\rm The element $\ef_{4,3} =  \begin{array}{c}
\begin{tikzpicture}[xscale=.30,yscale=.30,line width=1.25pt] 
\foreach \i in {1,2,3,4} 
{ \path (\i,.5) coordinate (T\i); \path (\i,-.5) coordinate (B\i); } 
\filldraw[fill= black!12,draw=black!12,line width=4pt]  (T1) -- (T4) -- (B4) -- (B1) -- (T1);
\draw[blue] (T1) -- (B1);
\draw[blue] (T2) -- (B2);
\draw[blue] (T3) -- (B3);
\draw[blue] (T4) -- (B4);
\foreach \i in {1,2,3,4} 
{ \filldraw[fill=white,draw=black,line width = 1pt] (T\i) circle (5pt); \filldraw[fill=white,draw=black,line width = 1pt]  (B\i) circle (5pt); } 
\end{tikzpicture}
\end{array}$  generates the 
kernel of  $\Phi_{4,3}: \P_4(3) \to \End_{\S_3}(\M_3^{\otimes 4})$. 
Here is the diagram basis expansion for $\ef_{4,3}$  using the M\"obius formulas in \eqref{eq:mobiusa} and \eqref{eq:mobiusb}.    The diagram basis expansion for  $\ef_{2,3}$  can be found
in \eqref{eq:orbit2diag}.

 \begin{align*}
\ef_{4,3} = \begin{array}{c}
\begin{tikzpicture}[xscale=.40,yscale=.40,line width=1.25pt] 
\foreach \i in {1,2,3,4} 
{ \path (\i,.5) coordinate (T\i); \path (\i,-.5) coordinate (B\i); } 
\filldraw[fill= black!12,draw=black!12,line width=4pt]  (T1) -- (T4) -- (B4) -- (B1) -- (T1);
\draw[blue] (T1) -- (B1);
\draw[blue] (T2) -- (B2);
\draw[blue] (T3) -- (B3);
\draw[blue] (T4) -- (B4);
\foreach \i in {1,2,3,4} 
{ \filldraw[fill=white,draw=black,line width = 1pt] (T\i) circle (5pt); \filldraw[fill=white,draw=black,line width = 1pt]  (B\i) circle (5pt); } 
\end{tikzpicture}
\end{array} 
& =
\begin{array}{c}
\begin{tikzpicture}[xscale=.40,yscale=.40,line width=1.25pt] 
\foreach \i in {1,2,3,4} 
{ \path (\i,.5) coordinate (T\i); \path (\i,-.5) coordinate (B\i); } 
\filldraw[fill= black!12,draw=black!12,line width=4pt]  (T1) -- (T4) -- (B4) -- (B1) -- (T1);
\draw[blue] (T1) -- (B1);
\draw[blue] (T2) -- (B2);
\draw[blue] (T3) -- (B3);
\draw[blue] (T4) -- (B4);
\foreach \i in {1,2,3,4} 
{ \fill (T\i) circle (5pt); \fill (B\i) circle (5pt); } 
\end{tikzpicture}
\end{array}
-
\begin{array}{c}
\begin{tikzpicture}[xscale=.40,yscale=.40,line width=1.25pt] 
\foreach \i in {1,2,3,4} 
{ \path (\i,1.25) coordinate (T\i); \path (\i,.25) coordinate (B\i); } 
\filldraw[fill= black!12,draw=black!12,line width=4pt]  (T1) -- (T4) -- (B4) -- (B1) -- (T1);
\draw[blue] (T1) -- (B1);
\draw[blue] (T2) -- (B2);
\draw[blue] (T3) -- (B3);
\draw[blue] (T4) -- (B4);
\draw[blue] (T1) .. controls +(0,-.40) and +(0,-.40) .. (T2);
\foreach \i in {1,2,3,4} 
{ \fill (T\i) circle (5pt); \fill (B\i) circle (5pt); } 
\end{tikzpicture}
\end{array}
-
\begin{array}{c}
\begin{tikzpicture}[xscale=.40,yscale=.40,line width=1.25pt] 
\foreach \i in {1,2,3,4} 
{ \path (\i,1.25) coordinate (T\i); \path (\i,.25) coordinate (B\i); } 
\filldraw[fill= black!12,draw=black!12,line width=4pt]  (T1) -- (T4) -- (B4) -- (B1) -- (T1);
\draw[blue] (T1) -- (B1);
\draw[blue] (T2) -- (B2);
\draw[blue] (T3) -- (B3);
\draw[blue] (T4) -- (B4);
\draw[blue] (T1) .. controls +(0,-.60) and +(0,-.60) .. (T3);
\foreach \i in {1,2,3,4} 
{ \fill (T\i) circle (5pt); \fill (B\i) circle (5pt); } 
\end{tikzpicture}
\end{array}
-
\begin{array}{c}
\begin{tikzpicture}[xscale=.40,yscale=.40,line width=1.25pt] 
\foreach \i in {1,2,3,4} 
{ \path (\i,1.25) coordinate (T\i); \path (\i,.25) coordinate (B\i); } 
\filldraw[fill= black!12,draw=black!12,line width=4pt]  (T1) -- (T4) -- (B4) -- (B1) -- (T1);
\draw[blue] (T1) -- (B1);
\draw[blue] (T2) -- (B2);
\draw[blue] (T3) -- (B3);
\draw[blue] (T4) -- (B4);
\draw[blue] (T1) .. controls +(0,-.80) and +(0,-.80) .. (T4);
\foreach \i in {1,2,3,4} 
{ \fill (T\i) circle (5pt); \fill (B\i) circle (5pt); } 
\end{tikzpicture}
\end{array}
-
\begin{array}{c}
\begin{tikzpicture}[xscale=.40,yscale=.40,line width=1.25pt] 
\foreach \i in {1,2,3,4} 
{ \path (\i,1.25) coordinate (T\i); \path (\i,.25) coordinate (B\i); } 
\filldraw[fill= black!12,draw=black!12,line width=4pt]  (T1) -- (T4) -- (B4) -- (B1) -- (T1);
\draw[blue] (T1) -- (B1);
\draw[blue] (T2) -- (B2);
\draw[blue] (T3) -- (B3);
\draw[blue] (T4) -- (B4);
\draw[blue] (T2) .. controls +(0,-.40) and +(0,-.40) .. (T3);
\foreach \i in {1,2,3,4} 
{ \fill (T\i) circle (5pt); \fill (B\i) circle (5pt); } 
\end{tikzpicture}
\end{array}
\\
&
-
\begin{array}{c}
\begin{tikzpicture}[xscale=.40,yscale=.40,line width=1.25pt] 
\foreach \i in {1,2,3,4} 
{ \path (\i,1.25) coordinate (T\i); \path (\i,.25) coordinate (B\i); } 
\filldraw[fill= black!12,draw=black!12,line width=4pt]  (T1) -- (T4) -- (B4) -- (B1) -- (T1);
\draw[blue] (T1) -- (B1);
\draw[blue] (T2) -- (B2);
\draw[blue] (T3) -- (B3);
\draw[blue] (T4) -- (B4);
\draw[blue] (T2) .. controls +(0,-.60) and +(0,-.60) .. (T4);
\foreach \i in {1,2,3,4} 
{ \fill (T\i) circle (5pt); \fill (B\i) circle (5pt); } 
\end{tikzpicture}
\end{array}
-
\begin{array}{c}
\begin{tikzpicture}[xscale=.40,yscale=.40,line width=1.25pt] 
\foreach \i in {1,2,3,4} 
{ \path (\i,1.25) coordinate (T\i); \path (\i,.25) coordinate (B\i); } 
\filldraw[fill= black!12,draw=black!12,line width=4pt]  (T1) -- (T4) -- (B4) -- (B1) -- (T1);
\draw[blue] (T1) -- (B1);
\draw[blue] (T2) -- (B2);
\draw[blue] (T3) -- (B3);
\draw[blue] (T4) -- (B4);
\draw[blue] (T3) .. controls +(0,-.40) and +(0,-.40) .. (T4);
\foreach \i in {1,2,3,4} 
{ \fill (T\i) circle (5pt); \fill (B\i) circle (5pt); } 
\end{tikzpicture}
\end{array}
 + 2 
\begin{array}{c}
\begin{tikzpicture}[xscale=.40,yscale=.40,line width=1.25pt] 
\foreach \i in {1,2,3,4} 
{ \path (\i,1.25) coordinate (T\i); \path (\i,.25) coordinate (B\i); } 
\filldraw[fill= black!12,draw=black!12,line width=4pt]  (T1) -- (T4) -- (B4) -- (B1) -- (T1);
\draw[blue] (T1) -- (B1);
\draw[blue] (T2) -- (B2);
\draw[blue] (T3) -- (B3);
\draw[blue] (T4) -- (B4);
\draw[blue] (T1) .. controls +(0,-.40) and +(0,-.40) .. (T2) .. controls +(0,-.40) and +(0,-.40) .. (T3);
\foreach \i in {1,2,3,4} 
{ \fill (T\i) circle (5pt); \fill (B\i) circle (5pt); } 
\end{tikzpicture}
\end{array}
+2
\begin{array}{c}
\begin{tikzpicture}[xscale=.40,yscale=.40,line width=1.25pt] 
\foreach \i in {1,2,3,4} 
{ \path (\i,1.25) coordinate (T\i); \path (\i,.25) coordinate (B\i); } 
\filldraw[fill= black!12,draw=black!12,line width=4pt]  (T1) -- (T4) -- (B4) -- (B1) -- (T1);
\draw[blue] (T1) -- (B1);
\draw[blue] (T2) -- (B2);
\draw[blue] (T3) -- (B3);
\draw[blue] (T4) -- (B4);
\draw[blue] (T1) .. controls +(0,-.40) and +(0,-.40) .. (T2) .. controls +(0,-.60) and +(0,-.60) .. (T4);
\foreach \i in {1,2,3,4} 
{ \fill (T\i) circle (5pt); \fill (B\i) circle (5pt); } 
\end{tikzpicture}
\end{array} 
+2
\begin{array}{c}
\begin{tikzpicture}[xscale=.40,yscale=.40,line width=1.25pt] 
\foreach \i in {1,2,3,4} 
{ \path (\i,1.25) coordinate (T\i); \path (\i,.25) coordinate (B\i); } 
\filldraw[fill= black!12,draw=black!12,line width=4pt]  (T1) -- (T4) -- (B4) -- (B1) -- (T1);
\draw[blue] (T1) -- (B1);
\draw[blue] (T2) -- (B2);
\draw[blue] (T3) -- (B3);
\draw[blue] (T4) -- (B4);
\draw[blue] (T1) .. controls +(0,-.60) and +(0,-.60) .. (T3) .. controls +(0,-.40) and +(0,-.40) .. (T4);
\foreach \i in {1,2,3,4} 
{ \fill (T\i) circle (5pt); \fill (B\i) circle (5pt); } 
\end{tikzpicture}
\end{array}\\
&
+2
\begin{array}{c}
\begin{tikzpicture}[xscale=.40,yscale=.40,line width=1.25pt] 
\foreach \i in {1,2,3,4} 
{ \path (\i,1.25) coordinate (T\i); \path (\i,.25) coordinate (B\i); } 
\filldraw[fill= black!12,draw=black!12,line width=4pt]  (T1) -- (T4) -- (B4) -- (B1) -- (T1);
\draw[blue] (T1) -- (B1);
\draw[blue] (T2) -- (B2);
\draw[blue] (T3) -- (B3);
\draw[blue] (T4) -- (B4);
\draw[blue] (T2) .. controls +(0,-.40) and +(0,-.40) .. (T3) .. controls  +(0,-.40) and +(0,-.40) .. (T4);
\foreach \i in {1,2,3,4} 
{ \fill (T\i) circle (5pt); \fill (B\i) circle (5pt); } 
\end{tikzpicture}
\end{array}
+
\begin{array}{c}
\begin{tikzpicture}[xscale=.40,yscale=.40,line width=1.25pt] 
\foreach \i in {1,2,3,4} 
{ \path (\i,1.25) coordinate (T\i); \path (\i,.25) coordinate (B\i); } 
\filldraw[fill= black!12,draw=black!12,line width=4pt]  (T1) -- (T4) -- (B4) -- (B1) -- (T1);
\draw[blue] (T1) -- (B1);
\draw[blue] (T2) -- (B2);
\draw[blue] (T3) -- (B3);
\draw[blue] (T4) -- (B4);
\draw[blue] (T1) .. controls +(0,-.40) and +(0,-.40) .. (T2);
\draw[blue] (T3) .. controls +(0,-.40) and +(0,-.40) .. (T4);
\foreach \i in {1,2,3,4} 
{ \fill (T\i) circle (5pt); \fill (B\i) circle (5pt); } 
\end{tikzpicture}
\end{array}
+
\begin{array}{c}
\begin{tikzpicture}[xscale=.40,yscale=.40,line width=1.25pt] 
\foreach \i in {1,2,3,4} 
{ \path (\i,1.25) coordinate (T\i); \path (\i,.25) coordinate (B\i); } 
\filldraw[fill= black!12,draw=black!12,line width=4pt]  (T1) -- (T4) -- (B4) -- (B1) -- (T1);
\draw[blue] (T1) -- (B1);
\draw[blue] (T2) -- (B2);
\draw[blue] (T3) -- (B3);
\draw[blue] (T4) -- (B4);
\draw[blue] (T1) .. controls +(0,-.60) and +(0,-.60) .. (T3);
\draw[blue] (T2) .. controls +(0,-.60) and +(0,-.60) .. (T4);
\foreach \i in {1,2,3,4} 
{ \fill (T\i) circle (5pt); \fill (B\i) circle (5pt); } 
\end{tikzpicture}
\end{array}
+
\begin{array}{c}
\begin{tikzpicture}[xscale=.40,yscale=.40,line width=1.25pt] 
\foreach \i in {1,2,3,4} 
{ \path (\i,1.25) coordinate (T\i); \path (\i,.25) coordinate (B\i); } 
\filldraw[fill= black!12,draw=black!12,line width=4pt]  (T1) -- (T4) -- (B4) -- (B1) -- (T1);
\draw[blue] (T1) -- (B1);
\draw[blue] (T2) -- (B2);
\draw[blue] (T3) -- (B3);
\draw[blue] (T4) -- (B4);
\draw[blue] (T1) .. controls +(0,-.80) and +(0,-.80) .. (T4);
\draw[blue] (T2) .. controls +(0,-.40) and +(0,-.40) .. (T3);
\foreach \i in {1,2,3,4} 
{ \fill (T\i) circle (5pt); \fill (B\i) circle (5pt); } 
\end{tikzpicture}
\end{array}
- 6
\begin{array}{c}
\begin{tikzpicture}[xscale=.40,yscale=.40,line width=1.25pt] 
\foreach \i in {1,2,3,4} 
{ \path (\i,1.25) coordinate (T\i); \path (\i,.25) coordinate (B\i); } 
\filldraw[fill= black!12,draw=black!12,line width=4pt]  (T1) -- (T4) -- (B4) -- (B1) -- (T1);
\draw[blue] (T1) -- (B1);
\draw[blue] (T2) -- (B2);
\draw[blue] (T3) -- (B3);
\draw[blue] (T4) -- (B4);
\draw[blue] (T1) .. controls +(0,-.40) and +(0,-.40) .. (T2);
\draw[blue] (T2) .. controls +(0,-.40) and +(0,-.40) .. (T3);
\draw[blue] (T3) .. controls +(0,-.40) and +(0,-.40) .. (T4);
\foreach \i in {1,2,3,4} 
{ \fill (T\i) circle (5pt); \fill (B\i) circle (5pt); } 
\end{tikzpicture}
\end{array}.
\end{align*}
There are $15 = \mathsf{B}(4)$ diagram basis elements in the expression for orbit element $\ef_{k,n}$.}
\end{rem}

\section{The Fundamental Theorems of Invariant Theory for $\S_n$}  \label{sec:fun}

 Section \ref{subsec:repn} gives the explicit construction of the algebra homomorphism $\Phi_{k,n}: \P_k(n) \rightarrow \End_{\S_n}(\M_n^{\ot k})$  and shows that
the partition algebra generates the tensor invariants of the symmetric group $\S_n$. The  First Fundamental Theorem of Invariant Theory for $\S_n$ says that the partition algebra generates \emph{all} tensor invariants of $\S_n$ on   $\End_{\S_n}(\M_n^{\ot k}) \cong \big(\M_n^{\ot {2k}}\big)^{\S_n}$, as $\Phi_{k,n}$ is a surjection for all $k,n$. 
 A more precise statement is the following: \medskip

\begin{thm}\label{T:1stfund}{\rm (\cite{J})}\, {\rm (First Fundamental Theorem of Invariant Theory for $\S_n$)} \, 
For all $k,n\in \ZZ_{\ge 1}$, $\Phi_{k,n}: \P_k(n) \to  \End_{\S_n}(\modu^{\ot k})$ is a surjective algebra homomorphism,
and when $n \ge 2k$, $\Phi_{k,n}$ is an isomorphism, so  $\P_k(n) \cong  \End_{\S_n}(\modu^{\ot k})$ when $n \ge 2k$.
\end{thm} 

 For $k \in \ZZ_{\ge 1}$,  the partition algebra $\P_k(n)$  has a presentation by the generators
\begin{align}
\label{s-gen}
\mathfrak{s}_i &=  
\begin{array}{c}\begin{tikzpicture}[scale=.6,line width=1.25pt] 
\foreach \i in {1,...,8} 
{ \path (\i,1) coordinate (T\i); \path (\i,0) coordinate (B\i); } 
\filldraw[fill= gray!40,draw=gray!40,line width=3.2pt]  (T1) -- (T8) -- (B8) -- (B1) -- (T1);
\draw[blue] (T1) -- (B1);
\draw[blue] (T3) -- (B3);
\draw[blue] (T4) -- (B5);
\draw[blue] (T5) -- (B4);
\draw[blue] (T6) -- (B6);
\draw[blue] (T8) -- (B8);
\foreach \i in {1,3,4,5,6,8} { \fill (T\i) circle (4pt); \fill (B\i) circle (4pt); } 
\draw (T2) node  {$\cdots$}; \draw (B2) node  {$\cdots$}; \draw (T7) node  {$\cdots$}; \draw (B7) node  {$\cdots$}; 
\draw  (T4)  node[black,above=0.05cm]{$\scriptstyle{i}$};
\draw  (T5)  node[black,above=0.0cm]{$\scriptstyle{i+1}$};
\end{tikzpicture}\end{array} \qquad 1 \le i \le k-1,
\\
\label{p-gen}
\mathfrak{p}_i &=  \frac{1}{n}
\begin{array}{c}\begin{tikzpicture}[scale=.6,line width=1.25pt] 
\foreach \i in {1,...,8} 
{ \path (\i,1) coordinate (T\i); \path (\i,0) coordinate (B\i); } 
\filldraw[fill= gray!40,draw=gray!40,line width=3.2pt]  (T1) -- (T8) -- (B8) -- (B1) -- (T1);
\draw[blue] (T1) -- (B1);
\draw[blue] (T3) -- (B3);
\draw[blue] (T5) -- (B5);
\draw[blue] (T6) -- (B6);
\draw[blue] (T8) -- (B8);
\foreach \i in {1,3,4,5,6,8} { \fill (T\i) circle (4pt); \fill (B\i) circle (4pt); } 
\draw (T2) node  {$\cdots$}; \draw (B2) node  {$\cdots$}; \draw (T7) node  {$\cdots$}; \draw (B7) node  {$\cdots$}; 
\draw  (T4)  node[black,above=0.05cm]{$\scriptstyle{i}$};
\end{tikzpicture}\end{array} \qquad 1 \le i \le k, \\
\label{b-gen}
\mathfrak{b}_i &=  
\begin{array}{c}\begin{tikzpicture}[scale=.6,line width=1.25pt] 
\foreach \i in {1,...,8} 
{ \path (\i,1) coordinate (T\i); \path (\i,0) coordinate (B\i); } 
\filldraw[fill= gray!40,draw=gray!40,line width=3.2pt]  (T1) -- (T8) -- (B8) -- (B1) -- (T1);
\draw[blue] (T1) -- (B1);
\draw[blue] (T3) -- (B3);
\draw[blue] (T4) .. controls +(.1,-.30) and +(-.1,-.30) .. (T5);
\draw[blue] (B4) .. controls +(.1,+.30) and +(-.1,+.30) .. (B5);
\draw[blue] (T4) -- (B4);
\draw[blue] (T5) -- (B5);
\draw[blue] (T6) -- (B6);
\draw[blue] (T8) -- (B8);
\foreach \i in {1,3,4,5,6,8} { \fill (T\i) circle (4pt); \fill (B\i) circle (4pt); } 
\draw (T2) node  {$\cdots$}; \draw (B2) node  {$\cdots$}; \draw (T7) node  {$\cdots$}; \draw (B7) node  {$\cdots$}; 
\draw  (T4)  node[black,above=0.05cm]{$\scriptstyle{i}$};
\draw  (T5)  node[black,above=0.0cm]{$\scriptstyle{i+1}$};
\end{tikzpicture}\end{array} \qquad 1 \le i \le k-1,
\end{align}
and the relations in the next result.    

\begin{thm}{\rm \cite[Thm.~1.11]{HR}}\label{T:present}  Assume $k \in \ZZ_{\ge 1}$,  and set $\mathfrak{p}_{i+\half} = \mathfrak{b}_i \ \  (1 \le i \le k-1)$.     Then $\P_k( n)$ has a presentation as a unital associative algebra by  
generators $\mathfrak{s}_i \ \ (1 \leq i \le k-1)$,  $\mathfrak{p}_\ell \ \  (\ell \in \half \ZZ_{\ge 1},  \  1 \le \ell \le k)$,  and the following relations:
{\rm
\begin{itemize}
\item[{\rm(a)}]  $ \mathfrak{s}_i^2 = \mathsf{I}_k, \quad   \mathfrak{s}_i \mathfrak{s}_j = \mathfrak{s}_j\mathfrak{s}_i \quad  (|i-j| > 1)$, \quad    $\mathfrak{s}_i \mathfrak{s}_{i+1} \mathfrak{s}_i =  \mathfrak{s}_{i+1}
 \mathfrak{s}_i \mathfrak{s}_{i+1} \quad (1 \leq i \leq k-2)$;
\item[{\rm(b)}]  $\mathfrak{p}_\ell^2 = \mathfrak{p}_\ell,  \quad   \mathfrak{p}_\ell \mathfrak{p}_m = 
 \mathfrak{p}_m \mathfrak{p}_\ell \quad (m \ne \ell\pm \half), \quad   \mathfrak{p}_\ell \mathfrak{p}_{\ell\pm \half}  \mathfrak{p}_\ell =\mathfrak{p}_\ell  \quad   ( \mathfrak{p}_\half := \mathsf{I}_k =: \mathfrak{p}_{k+\half})$;
 \item[{\rm(c)}]  $\mathfrak{s}_i \mathfrak{p}_i \mathfrak{p}_{i +1} =  \mathfrak{p}_i \mathfrak{p}_{i +1},  \   \quad  
  \mathfrak{s}_i \mathfrak{p}_i \mathfrak{s}_i =  \mathfrak{p}_{i+1},   \ \quad   \mathfrak{s}_i \mathfrak{p}_{i+\half}= 
 \mathfrak{p}_{i+\half} \mathfrak{s}_i = \mathfrak{p}_{i+\half}  \ \ \ \, (1 \leq i \leq k-1), \newline 
 \mathfrak{s}_i  \mathfrak{s}_{i+1}  \mathfrak{p}_{i+\half} \mathfrak{s}_{i+1}\mathfrak{s}_i = \mathfrak{p}_{i+\frac{3}{2}} \ \ \ (1 \le i \le k-2), \quad  \mathfrak{s}_i \mathfrak{p}_\ell = \mathfrak{p}_\ell \mathfrak{s}_i \quad  (\ell \neq i-\half, i,i+\half, i+1, i+\frac{3}{2})$.
 \end{itemize}
 }
\end{thm} 
  
\begin{rem}{\rm  It is easily seen from the relations that $\P_k(n)$ is generated by $\mathfrak{s}_i\ (1 \le i \le k-1), \mathfrak{p}_1$,  and $\mathfrak{b}_1 = \mathfrak{p}_{1+\half}.$   
}
\end{rem} 

 Theorems 5.6 and 5.8 of \cite{BH} 
prove that $\ef_{k,n}$ is an essential idempotent that generates the kernel of $\Phi_{k,n}$ as a two-sided ideal.  Moreover,  Theorem 5.15 of \cite{BH} shows that the kernel of $\Phi_{k,n}$ is generated as a two-sided ideal by the embedded image $\ef_{n,n} \ot (\blvertedge)^{\ot k-n}$ (the diagram of $\ef_{n,n}$ with $k-n$ vertical edges $\blvertedge$ juxtaposed to its right)
 of the essential idempotent $\ef_{n,n}$ in $\P_k(n)$  for all $k \ge n$.  By \cite[Remark 5.20]{BH},  $\ker \Phi_{k,n}$
 cannot be generated by $\ef_{\ell,n} \ot (\blvertedge)^{\ot k-\ell}$ for any $\ell$ such that $k \ge n > \ell \ge \half(n+1)$.  
Identifying $\ef_{n,n}$ with its image in $\P_k(n)$ for $k \ge n$, we have  \smallskip
 
\begin{thm}\label{T:SecondFund} {\rm \cite[Thm.~5.19]{BH}} \ (Second Fundamental Theorem of Invariant Theory for $\S_n$)  \ For all $k,n\in \ZZ_{\ge 1}$,   $\im \Phi_{k,n}
= \End_{\S_n}(\modu^{\ot k})$ is generated by the partition algebra generators and relations in Theorem \ref{T:present} (a)-(c) together with the one additional relation $\ef_{k,n} = 0$ in the case that $2k > n$.   When $k \ge n$,  the relation $\ef_{k,n} = 0$ can be replaced with  $\ef_{n,n} = 0$.
\end{thm}  

\begin{examp} {\rm The kernel of $\Phi_{3,3}: \P_3(3) \to \End_{\S_3}(\M_3^{\otimes 3})$ is generated by $\ef_{3,3} =  \begin{array}{c}
\begin{tikzpicture}[xscale=.30,yscale=.30,line width=1.25pt] 
\foreach \i in {1,2,3} 
{ \path (\i,.5) coordinate (T\i); \path (\i,-.5) coordinate (B\i); } 
\filldraw[fill= black!12,draw=black!12,line width=4pt]  (T1) -- (T3) -- (B3) -- (B1) -- (T1);
\draw[blue] (T2) -- (B2);
\draw[blue] (T3) -- (B3);
\foreach \i in {1,2,3} 
{ \filldraw[fill=white,draw=black,line width = 1pt] (T\i) circle (5pt); \filldraw[fill=white,draw=black,line width = 1pt]  (B\i) circle (5pt); } 
\end{tikzpicture}
\end{array}$, and the embedded element $\ef_{3,3} \ot (\blvertedge)^{\ot k-3}$ principally generates the kernel for $\P_k(3)$ for all $k \ge 3$.
The image of $\Phi_{3,3}$ is generated by the partition algebra $\P_3(3)$ with the additional dependence relation

\begin{align*}
 0 = & \begin{array}{c}\begin{tikzpicture}[xscale=.45,yscale=.45,line width=1.0pt] 
\foreach \i in {1,2,3}  { \path (\i,1.25) coordinate (T\i); \path (\i,.25) coordinate (B\i); } 
\filldraw[fill= black!12,draw=black!12,line width=4pt]  (T1) -- (T3) -- (B3) -- (B1) -- (T1);
\draw[blue] (T2) -- (B2);
\draw[blue] (T3) -- (B3);
\foreach \i in {1,2,3}  { \filldraw[fill=white,draw=black,line width = 1pt] (T\i) circle (4pt); \filldraw[fill=white,draw=black,line width = 1pt]  (B\i) circle (4pt); } 
\end{tikzpicture}\end{array}
=
\begin{array}{c}\begin{tikzpicture}[xscale=.45,yscale=.45,line width=1.0pt] 
\foreach \i in {1,2,3}  { \path (\i,1.25) coordinate (T\i); \path (\i,.25) coordinate (B\i); } 
\filldraw[fill= black!12,draw=black!12,line width=4pt]  (T1) -- (T3) -- (B3) -- (B1) -- (T1);
\draw[blue] (T2) -- (B2);
\draw[blue] (T3) -- (B3);
\foreach \i in {1,2,3}  { \filldraw[fill=black,draw=black,line width = 1pt] (T\i) circle (4pt); \filldraw[fill=black,draw=black,line width = 1pt]  (B\i) circle (4pt); } 
\end{tikzpicture}\end{array}
-
\begin{array}{c}\begin{tikzpicture}[xscale=.45,yscale=.45,line width=1.0pt] 
\foreach \i in {1,2,3}  { \path (\i,1.25) coordinate (T\i); \path (\i,.25) coordinate (B\i); } 
\filldraw[fill= black!12,draw=black!12,line width=4pt]  (T1) -- (T3) -- (B3) -- (B1) -- (T1);
\draw[blue] (T1) -- (B1);
\draw[blue] (T2) -- (B2);
\draw[blue] (T3) -- (B3);
\foreach \i in {1,2,3}  { \filldraw[fill=black,draw=black,line width = 1pt] (T\i) circle (4pt); \filldraw[fill=black,draw=black,line width = 1pt]  (B\i) circle (4pt); } 
\end{tikzpicture}\end{array}
-
\begin{array}{c}\begin{tikzpicture}[xscale=.45,yscale=.45,line width=1.0pt] 
\foreach \i in {1,2,3}  { \path (\i,1.25) coordinate (T\i); \path (\i,.25) coordinate (B\i); } 
\filldraw[fill= black!12,draw=black!12,line width=4pt]  (T1) -- (T3) -- (B3) -- (B1) -- (T1);
\draw[blue] (T1)--(T2);
\draw[blue] (T2)--(B2);
\draw[blue] (T3) -- (B3);
\foreach \i in {1,2,3}  { \filldraw[fill=black,draw=black,line width = 1pt] (T\i) circle (4pt); \filldraw[fill=black,draw=black,line width = 1pt]  (B\i) circle (4pt); } 
\end{tikzpicture}\end{array}
-
\begin{array}{c}\begin{tikzpicture}[xscale=.45,yscale=.45,line width=1.0pt] 
\foreach \i in {1,2,3}  { \path (\i,1.25) coordinate (T\i); \path (\i,.25) coordinate (B\i); } 
\filldraw[fill= black!12,draw=black!12,line width=4pt]  (T1) -- (T3) -- (B3) -- (B1) -- (T1);
\draw[blue] (B1)--(B2);
\draw[blue] (T2)--(B2);
\draw[blue] (T3) -- (B3);
\foreach \i in {1,2,3}  { \filldraw[fill=black,draw=black,line width = 1pt] (T\i) circle (4pt); \filldraw[fill=black,draw=black,line width = 1pt]  (B\i) circle (4pt); } 
\end{tikzpicture}\end{array}
-
\begin{array}{c}\begin{tikzpicture}[xscale=.45,yscale=.45,line width=1.0pt] 
\foreach \i in {1,2,3}  { \path (\i,1.25) coordinate (T\i); \path (\i,.25) coordinate (B\i); } 
\filldraw[fill= black!12,draw=black!12,line width=4pt]  (T1) -- (T3) -- (B3) -- (B1) -- (T1);
\draw[blue] (T2) -- (B2);
\draw[blue] (T3) -- (B3);
\draw[blue] (T1) .. controls +(0,-.50) and +(0,-.50) .. (T3);
\foreach \i in {1,2,3}  { \filldraw[fill=black,draw=black,line width = 1pt] (T\i) circle (4pt); \filldraw[fill=black,draw=black,line width = 1pt]  (B\i) circle (4pt); } 
\end{tikzpicture}\end{array}
-
\begin{array}{c}\begin{tikzpicture}[xscale=.45,yscale=.45,line width=1.0pt] 
\foreach \i in {1,2,3}  { \path (\i,1.25) coordinate (T\i); \path (\i,.25) coordinate (B\i); } 
\filldraw[fill= black!12,draw=black!12,line width=4pt]  (T1) -- (T3) -- (B3) -- (B1) -- (T1);
\draw[blue] (T2) -- (B2);
\draw[blue] (T3) -- (B3);
\draw[blue] (B1) .. controls +(0,.50) and +(0,.50) .. (B3);
\foreach \i in {1,2,3}  { \filldraw[fill=black,draw=black,line width = 1pt] (T\i) circle (4pt); \filldraw[fill=black,draw=black,line width = 1pt]  (B\i) circle (4pt); } 
\end{tikzpicture}\end{array}
\\
& \hskip.8in
-
\begin{array}{c}\begin{tikzpicture}[xscale=.45,yscale=.45,line width=1.0pt] 
\foreach \i in {1,2,3}  { \path (\i,1.25) coordinate (T\i); \path (\i,.25) coordinate (B\i); } 
\filldraw[fill= black!12,draw=black!12,line width=4pt]  (T1) -- (T3) -- (B3) -- (B1) -- (T1);
\draw[blue] (T2) -- (B2)--(B3);
\draw[blue] (T2)--(T3) -- (B3);
\foreach \i in {1,2,3}  { \filldraw[fill=black,draw=black,line width = 1pt] (T\i) circle (4pt); \filldraw[fill=black,draw=black,line width = 1pt]  (B\i) circle (4pt); } 
\end{tikzpicture}\end{array} 
+ 
\begin{array}{c}\begin{tikzpicture}[xscale=.45,yscale=.45,line width=1.0pt] 
\foreach \i in {1,2,3}  { \path (\i,1.25) coordinate (T\i); \path (\i,.25) coordinate (B\i); } 
\filldraw[fill= black!12,draw=black!12,line width=4pt]  (T1) -- (T3) -- (B3) -- (B1) -- (T1);
\draw[blue] (T1) -- (B1);
\draw[blue] (T3)--(T2) -- (B2) -- (B3);
\draw[blue] (T3) -- (B3);
\foreach \i in {1,2,3}  { \filldraw[fill=black,draw=black,line width = 1pt] (T\i) circle (4pt); \filldraw[fill=black,draw=black,line width = 1pt]  (B\i) circle (4pt); } 
\end{tikzpicture}\end{array}
+
\begin{array}{c}\begin{tikzpicture}[xscale=.45,yscale=.45,line width=1.0pt] 
\foreach \i in {1,2,3}  { \path (\i,1.25) coordinate (T\i); \path (\i,.25) coordinate (B\i); } 
\filldraw[fill= black!12,draw=black!12,line width=4pt]  (T1) -- (T3) -- (B3) -- (B1) -- (T1);
\draw[blue] (T1)--(T2);
\draw[blue] (T2)--(B2);
\draw[blue] (T3) -- (B3);
\draw[blue] (B1) .. controls +(0,.50) and +(0,.50) .. (B3);
\foreach \i in {1,2,3}  { \filldraw[fill=black,draw=black,line width = 1pt] (T\i) circle (4pt); \filldraw[fill=black,draw=black,line width = 1pt]  (B\i) circle (4pt); } 
\end{tikzpicture}\end{array}
+
\begin{array}{c}\begin{tikzpicture}[xscale=.45,yscale=.45,line width=1.0pt] 
\foreach \i in {1,2,3}  { \path (\i,1.25) coordinate (T\i); \path (\i,.25) coordinate (B\i); } 
\filldraw[fill= black!12,draw=black!12,line width=4pt]  (T1) -- (T3) -- (B3) -- (B1) -- (T1);
\draw[blue] (B1)--(B2);
\draw[blue] (T2)--(B2);
\draw[blue] (T3) -- (B3);
\draw[blue] (T1) .. controls +(0,-.50) and +(0,-.50) .. (T3);
\foreach \i in {1,2,3}  { \filldraw[fill=black,draw=black,line width = 1pt] (T\i) circle (4pt); \filldraw[fill=black,draw=black,line width = 1pt]  (B\i) circle (4pt); } 
\end{tikzpicture}\end{array}
+2
\begin{array}{c}\begin{tikzpicture}[xscale=.45,yscale=.45,line width=1.0pt] 
\foreach \i in {1,2,3}  { \path (\i,1.25) coordinate (T\i); \path (\i,.25) coordinate (B\i); } 
\filldraw[fill= black!12,draw=black!12,line width=4pt]  (T1) -- (T3) -- (B3) -- (B1) -- (T1);
\draw[blue] (T1)--(T2) -- (B2) -- (B3);
\draw[blue] (T2)--(T3) -- (B3);
\foreach \i in {1,2,3}  { \filldraw[fill=black,draw=black,line width = 1pt] (T\i) circle (4pt); \filldraw[fill=black,draw=black,line width = 1pt]  (B\i) circle (4pt); } 
\end{tikzpicture}\end{array}
\\
& \hskip.8in
+2
\begin{array}{c}\begin{tikzpicture}[xscale=.45,yscale=.45,line width=1.0pt] 
\foreach \i in {1,2,3}  { \path (\i,1.25) coordinate (T\i); \path (\i,.25) coordinate (B\i); } 
\filldraw[fill= black!12,draw=black!12,line width=4pt]  (T1) -- (T3) -- (B3) -- (B1) -- (T1);
\draw[blue] (B1)--(B2) -- (T2);
\draw[blue] (T2)--(T3) -- (B3)--(B2);
\foreach \i in {1,2,3}  { \filldraw[fill=black,draw=black,line width = 1pt] (T\i) circle (4pt); \filldraw[fill=black,draw=black,line width = 1pt]  (B\i) circle (4pt); } 
\end{tikzpicture}\end{array} +2
\begin{array}{c}\begin{tikzpicture}[xscale=.45,yscale=.45,line width=1.0pt] 
\foreach \i in {1,2,3}  { \path (\i,1.25) coordinate (T\i); \path (\i,.25) coordinate (B\i); } 
\filldraw[fill= black!12,draw=black!12,line width=4pt]  (T1) -- (T3) -- (B3) -- (B1) -- (T1);
\draw[blue] (T1) -- (T2) -- (B2) -- (B1) -- (T1);
\draw[blue] (T3)--(B3);
\foreach \i in {1,2,3}  { \filldraw[fill=black,draw=black,line width = 1pt] (T\i) circle (4pt); \filldraw[fill=black,draw=black,line width = 1pt]  (B\i) circle (4pt); } 
\end{tikzpicture}\end{array}
+ 2 
\begin{array}{c}\begin{tikzpicture}[xscale=.45,yscale=.45,line width=1.0pt] 
\foreach \i in {1,2,3}  { \path (\i,1.25) coordinate (T\i); \path (\i,.25) coordinate (B\i); } 
\filldraw[fill= black!12,draw=black!12,line width=4pt]  (T1) -- (T3) -- (B3) -- (B1) -- (T1);
\draw[blue] (T1) -- (B1);
\draw[blue] (T2) -- (B2);
\draw[blue] (T3) -- (B3);
\draw[blue] (T1) .. controls +(0,-.50) and +(0,-.50) .. (T3);
\draw[blue] (B1) .. controls +(0,.50) and +(0,.50) .. (B3);
\foreach \i in {1,2,3}  { \filldraw[fill=black,draw=black,line width = 1pt] (T\i) circle (4pt); \filldraw[fill=black,draw=black,line width = 1pt]  (B\i) circle (4pt); } 
\end{tikzpicture}\end{array}
- 6
\begin{array}{c}\begin{tikzpicture}[xscale=.45,yscale=.45,line width=1.0pt] 
\foreach \i in {1,2,3}  { \path (\i,1.25) coordinate (T\i); \path (\i,.25) coordinate (B\i); } 
\filldraw[fill= black!12,draw=black!12,line width=4pt]  (T1) -- (T3) -- (B3) -- (B1) -- (T1);
\draw[blue] (T1) -- (B1) -- (B2);
\draw[blue] (T1)--(T2) -- (B2) -- (B3);
\draw[blue] (T2)--(T3) -- (B3);
\foreach \i in {1,2,3}  { \filldraw[fill=black,draw=black,line width = 1pt] (T\i) circle (4pt); \filldraw[fill=black,draw=black,line width = 1pt]  (B\i) circle (4pt); } 
\end{tikzpicture}\end{array}.
\end{align*}
 This dependence relation  is analogous to the  one that 
comes from setting the kernel generator $\sum_{\sigma \in \S_{n+1}} (-1)^{\mathsf{sgn}(\sigma)} \sigma$
of the surjection $\FF \S_k \rightarrow \End_{\mathsf{GL}_n}(\VV^{\ot k})$ ($\VV = \FF^n$) equal to 0 in the Second Fundamental Theorem of Invariant Theory for $\mathsf{GL}_n$. 
}
\end{examp}



\end{document}